\documentclass{amsart}
\usepackage[dvipsnames]{xcolor}
\usepackage{hyperref}
\usepackage{soul}
\usepackage{amsfonts,amscd,amssymb,amsmath,amsthm,mathtools}
\usepackage{xparse,enumerate}
\usepackage{tikz}
\usepackage{pgfplots}

\usetikzlibrary{decorations.pathmorphing}

\usepackage[foot]{amsaddr}
\begin{document}

\newtheorem{theorem}{Theorem}[section]
\newtheorem{lemma}[theorem]{Lemma}
\newtheorem{corollary}[theorem]{Corollary}
\newtheorem{conjecture}[theorem]{Conjecture}
\newtheorem{cor}[theorem]{Corollary}
\newtheorem{proposition}[theorem]{Proposition}
\theoremstyle{definition}
\newtheorem{definition}[theorem]{Definition}
\newtheorem{example}[theorem]{Example}
\newtheorem{claim}[theorem]{Claim}
\newtheorem{remark}[theorem]{Remark}

\newenvironment{pfofthm}[1]
{\par\vskip2\parsep\noindent{\sc Proof of\ #1. }}{{\hfill
$\Box$}
\par\vskip2\parsep}
\newenvironment{pfoflem}[1]
{\par\vskip2\parsep\noindent{\sc Proof of Lemma\ #1. }}{{\hfill
$\Box$}
\par\vskip2\parsep}


\newcommand{\R}{\mathbb{R}}
\newcommand{\T}{\mathcal{T}}
\newcommand{\C}{\mathcal{C}}
\newcommand{\G}{\mathcal{G}}
\newcommand{\Z}{\mathbb{Z}}
\newcommand{\Q}{\mathbb{Q}}
\newcommand{\E}{\mathbb E}
\newcommand{\N}{\mathbb N}

\newcommand{\Def}{\overset{\Delta}{=}}

\newcommand{\supp}{\operatorname{supp}}
\newcommand{\sgn}{\operatorname{sgn}}


\newcommand{\Prob}{\Pr}
\newcommand{\Var}{\operatorname{Var}}
\newcommand{\Exp}{\mathbb{E}}
\newcommand{\expect}{\mathbb{E}}
\newcommand{\1}{\mathbf{1}}
\newcommand{\prob}{\Pr}
\newcommand{\pr}{\Pr}
\newcommand{\filt}{\mathscr{F}}
\DeclareDocumentCommand \one { o }
{%
\IfNoValueTF {#1}
{\mathbf{1}  }
{\mathbf{1}\left\{ {#1} \right\} }%
}
\newcommand{\Bernoulli}{\operatorname{Bernoulli}}
\newcommand{\Binomial}{\operatorname{Binom}}
\newcommand{\Binom}{\Binomial}
\newcommand{\Poisson}{\operatorname{Poisson}}
\newcommand{\Exponential}{\operatorname{Exp}}


\newcommand{\Ai}{\operatorname{Ai}}
\newcommand{\tr}{\operatorname{tr}}
\renewcommand{\det}{\operatorname{det}}

\newcommand{\Image}{\operatorname{Image}}
\newcommand{\Span}{\operatorname{Span}}


\DeclareDocumentCommand \LEIG { O{N} } { \bar{\lambda_1}^{({#1})}}
\newcommand{\LEIGc}{ x_{\ast}}
\DeclareDocumentCommand \LBE { O{k} } { \mathcal{E}_{#1}}
\DeclareDocumentCommand \zM { O{ } } { {\tilde{z}_{#1}} }
\DeclareDocumentCommand \yM { O{ } } { {\tilde{y}_{#1}} }
\DeclareDocumentCommand \wM {  } { {\tilde{w} }}
\DeclareDocumentCommand \zm { O{ } } { {{\zeta}_{#1}} }
\DeclareDocumentCommand \ym { O{ } } { {{x}_{#1}} }
\DeclareDocumentCommand \G { o } 
{ 
  \IfNoValueTF {#1}
  { G }
  { G_{#1}( \zM[#1], \yM[#1] )}
}

\newcommand{\odK}{ K_o }
\newcommand{\odI}{ Y }
\newcommand{\scK}{ \tilde{K} }

\newcommand{\atomvar}{ Z }
\newcommand{\GUEPP}{ \mathcal{G}}
\newcommand{\PP}{ \mathcal{X}}
\DeclareDocumentCommand \levraw { O{N} }
{ \tilde{\lambda}^{ ({#1})} }
\DeclareDocumentCommand \lev { O{N} }
{%
  \lambda^{({ {#1} })} 
}

\DeclareDocumentCommand \HSN { m } { \left\|{ #1 }\right\|_{\operatorname{HS}} }
\DeclareDocumentCommand \nuclear { O{\cdot} } { \left\|{ #1 }\right\|_{\nu} }
\DeclareDocumentCommand \opnorm { O{\cdot} } { \left\|{ #1 }\right\|_{\operatorname{op}} }
\newcommand{\Kairy}{ {K}_{\text{Airy}} }
\newcommand{\RES}{\Xi}
\newcommand{\mub}{t_{*}}
\newcommand{\TW}{\textrm{TW}}

\newcommand{\Id}{\operatorname{Id}}
\newcommand{\phiop}{ \boldsymbol{\phi}}
\newcommand{\Laplacian}{ \boldsymbol{\Delta}}
\newcommand{\phiopRZ}{ {\boldsymbol{\phi}}^{R,\zeta}}
\newcommand{\Kop}{ \mathbf{K}}
\newcommand{\Aiop}{\mathbf{{Ai}}}
\newcommand{\Dop}{ \mathbf{D}}
\newcommand{\Mop}{ \mathbf{M}}
\newcommand{\Rop}{ \mathbf{R}}

\newcommand{\osc}{ \varphi }

\DeclareDocumentCommand \Proj { m } { \pi_{#1} } 

\title[GUE extremal eigenvalues and fractional logarithm]{Extremal eigenvalue correlations in the GUE minor process and a law of fractional logarithm}
\author{Elliot Paquette}
\address{Department of Mathematics, Weizmann Institute of Science}
\email{elliot.paquette@gmail.com}
\author{Ofer Zeitouni}
\address{Department of Mathematics, Weizmann Institute of Science and Courant institute, NYU}
\email{ofer.zeitouni@weizmann.ac.il}
\thanks{The work of both authors was supported by a grant from the Israel Science foundation.
EP gratefully acknowledges the support of NSF Postdoctoral Fellowship DMS-1304057.
}
\date{May 19, 2015}
\maketitle

\begin{abstract}
  Let $\lambda^{(N)}$ be the largest eigenvalue of the $N \times N$ GUE matrix
  which is the $N$th element of the GUE minor process, rescaled to converge to the standard Tracy-Widom distribution. We consider the sequence 
  $\{\lambda^{(N)}\}_{N\geq 1}$ and prove a law of fractional logarithm
  for the limsup:
 $$ \limsup_{N \to \infty} \frac{\lev}{\left( \log N \right)^{2/3} } = 
    \left(\frac{1}{4}\right)^{2/3}\,,\quad
    \mbox{\it 
  almost surely}. $$
  For the liminf, we prove the weaker result that 
   there are constants $c_1, c_2 > 0$ so that
    $$
    -c_1 \leq \liminf_{N \to \infty} 
    \frac{\lev}{\left( \log N \right)^{1/3} } \leq -c_2\,,\quad \mbox{\it
    almost surely.}$$
    We conjecture that in fact, $c_1=c_2=4^{1/3}$.
\end{abstract}
\section{Introduction}
Let $S_n=\sum_{i=1}^n X_i$ be a random walk with i.i.d.\ increments of 
zero mean and unit variance. The celebrated Hartman--Wintner \cite{HW} 
law of
the iterated logarithm (LIL)  states that
$$\limsup_{n\to\infty} \frac{S_n}{\sqrt{2n\log\log n}}=1\,, \quad 
\mbox{\it almost surely}.$$
(Earlier versions of the LIL for bounded increments
were given by Khinchine and by Kolmogorov.)
Since $W_n:=S_n/\sqrt{n}$ is asymptotically standard normal,
the  LIL can be considered as a gauge of the extremal fluctuations
of sequence $\{W_n\}$.

In this paper, we investigate the analogous question for the largest
eigenvalue of the 
minor process of the Gaussian unitary ensemble (GUE) of random matrices. 
We begin by introducing some notation.
Let $\{\atomvar_{i,j}\}_{i,j=1}^\infty$ 
be a doubly infinite array of random variables where
\begin{enumerate}
  \item $\atomvar_{i,j}$ for $i > j$ is a complex 
    centered
    Gaussian of absolute variance $\frac 12$ 
    (that is, the real and imaginary parts of 
    $\atomvar_{i,j}$ are independent centered Gaussian of variance $1/4$),
  \item $\atomvar_{i,i}$ for $i \geq 1$ is a centered real Gaussian
    of variance $1/2$,
  \item $\{\atomvar_{i,j}\}_{i \geq j}$ are mutually independent, and
  \item $\atomvar_{i,j} = \overline{ \atomvar_{j,i} }$ for all $i,j \geq 1.$
\end{enumerate}
Let $\levraw$ be the largest eigenvalue of the $N \times N$ 
Hermitian matrix $G_N=(\atomvar_{i,j})_{i,j=1}^N.$ (The latter 
is a standard GUE(N) matrix.)
Center and scale $\levraw$ by defining
$$\lev = (\levraw - \sqrt{2N})\sqrt{2}N^{1/6}.$$  
A fundamental result in random matrix theory, due to Tracy
and Widom \cite{TW}, is the statement that
$\lev$  
converge in distribution as $N\to \infty$ to a Tracy--Widom variable. 
We study in this paper the analogue of the LIL for the sequence
$\{\lev\}_{N\geq 1}$.

Two ingredients enter into the proof of the Hartman--Wintner LIL: first,
the tail behavior of the sequence $\{W_n\}_{n\geq 1}$ (in the moderate deviations regime) is Gaussian and second,
the correlation between $W_n$ and $W_{n+m}$  begins to decay
only when $m$ is of order $n$. Both facts change when
one deals with the sequence $\{\lev\}$; further, because the Tracy--Widom
has differing (and non Gaussian) behavior in the upper and lower
tails,  extremal fluctuations of $\{\lev\}$ are not symmetric.

Our main result for the upper limit of $\{\lev\}$
is a complete analogue of 
the Hartman--Wintner LIL, except that the iterated logarithm is replaced
by a fractional power of the logarithm.
\begin{theorem}
With notation as above, we have
\begin{equation}
  \label{eq-onestar}
    \limsup_{N \to \infty} \frac{\lev}{\left( \log N \right)^{2/3} } = 
    \left(\frac{1}{4}\right)^{2/3}\,,\quad
    \mbox{\it 
  almost surely.}
\end{equation} 
  \label{thm:ue}
\end{theorem}
For the lower limit of $\{\lev\}$, we have less precise results. 
\begin{theorem}
  There are constants $c_1, c_2 > 0$ so that
  \begin{equation} 
    \label{eq-twostar}
    -c_1 \leq \liminf_{N \to \infty} \frac{\lev}{\left( \log N \right)^{1/3} } \leq -c_2\,,\quad \mbox{\it
    almost surely.}
  \end{equation} 
  \label{thm:le}
\end{theorem}
\noindent
That the scaling of the logarithm in Theorems  \ref{thm:ue} and 
\ref{thm:le} should be different is natural: indeed, for the
Tracy--Widom law $P_{\TW}$ it is known (see \cite[Exercise 3.8.3]{AGZ})
that
\begin{equation} 
	\label{eq-tails}
	\lim_{s\to\infty} \frac{1}{s^{3/2}} \log P_{\TW} 
\left( (s,\infty) \right)=-\frac43\,,\quad
	\lim_{s\to\infty} \frac{1}{s^{3}} \log P_{\TW} 
	\left( (-\infty,-s)) \right)=-\frac1{12}\,,\quad
\end{equation}
The different powers of $s$ in the exponent translate eventually to different
scalings for the logarithm.

The proof of 
Theorems \ref{thm:ue} and 
\ref{thm:le} relies on the joint determinantal structure 
of
the eigenvalues of the matrices 
$\{G_N\}_{N\geq 1}$, which we now describe  following 
\cite{FN,JN,JNerratum}.
Let $\Lambda = \N \times \R.$  We represent the eigenvalues of the 
sequence of matrices $\{G_N\}$ as
a point process $\GUEPP$ on $\Lambda$ by representing for every $N \in \N$ 
the eigenvalues of $G_N$ as points on the line $\left\{ N \right\} \times \R.$  
The process $\GUEPP$, referred to as the GUE minor process,  is determinantal
(with explicit kernel $K$, 
see \eqref{eq:kernel}),  
and various aspects of it have been studied, see 
\cite{JN,FN,Borodin}. 

As is the case for the Hartman-Wintner LIL, three ingredients are needed in 
proving Theorems \ref{thm:ue} and \ref{thm:le}. First, one needs a version
of \eqref{eq-tails} for the distribution of $\lev$, in the form
\begin{align}
	\label{eq-levuppertail}
  C_1(s)e^{-c_u s^{3/2}}&\leq  \Pr(\lev \in \left( (s,\infty) \right)
  \leq 
  C_2(s)e^{-c_u s^{3/2}}\,,\\
	\label{eq-levlowertail}
  C_3(s)e^{-c_l s^{3}}&\leq  \Pr(\lev \in \left( (-\infty,-s) \right)
  \leq 
  C_4(s)e^{-c_l s^{3}}\,,
\end{align}
which are uniform in the range $s\in [0, (\log N)^\gamma]$
for appropriate $\gamma$, and where $c_u=4/3$, $c_l=1/12$,
and $|\log(C_i(s))|=O( \log s)$. 

Second, 
one argues that there is a subsequence 
$N_k=k^\alpha$ sufficiently sparse (with $\alpha>1$) so
that the events
\begin{equation}
  \mathcal{F}_k =  \left\{ \lev[N_k] \geq c_1 (\log N_k)^{2/3} \right\}
  \label{eq:Fk}
\end{equation}
and
\begin{equation}
  \mathcal{E}_k =  \left\{ \lev[N_k] < -c_2 (\log N_k)^{1/3} \right\}
  \label{eq:Ek}
\end{equation}
are approximately independent, that is that
\begin{equation}
  \label{eq-Fk}
  \Pr \left( \mathcal{F}_k \cap \mathcal{F}_\ell \right) =
  \Pr \left( \mathcal{F}_k\right) \Pr \left(\mathcal{F}_\ell \right)( 1+ o(1)),
\end{equation}
with a similar estimate for $\mathcal{E}_k$. 
This leads to a lower bound for 
$\limsup_{k\to\infty}$ $ \lev[N_k] 
(\log N_k)^{-2/3} $
and to an upper bound for
\(
\liminf_{k\to\infty} \lev[N_k] (\log N_k)^{-1/3}.
\)
Due to work of~\cite{FN}, we know 
that the correlations of $\lev$ and $\lev[N+\Theta(N^{2/3})]$ 
are nontrivial and nondegenerate in the limit. 
This leads to the choice $\alpha=3 + \epsilon$.
The challenge 
however 
is to extend the decorrelation to the tail events $\mathcal{F}_k$ and
$\mathcal{E}_k$.

Third, we must show that along a subsequence $N_k = k^{\alpha}$ with $\alpha = 3 -\epsilon,$ the behavior of $\left\{\lev[N_k]\right\}_{k=1}^\infty$ determines the behavior of $\left\{ \lev[N] \right\}_{N=1}^\infty.$  In the case of the $\limsup,$ this means that only finitely many of the events 
\[
 \mathcal{F}_k'
 =
 \left\{ \exists\ N ~:~ N_{k-1} < N < N_k, \lev[N] \geq (c_1 +\delta) (\log N)^{2/3} \right\} \cap \mathcal{F}_k^c
\]
occur almost surely.  To do this, we must in effect show that 
\[
	\Pr\left( 
	\mathcal{F}_k'
	\right)
	\ll
	\Pr\left( 
	\mathcal{F}_k\right),
\]
which is to say that $\lev[N]$ for $N_{k-1} < N < N_k$ are highly correlated.  This leads to the upper bound for $\limsup_{N \to \infty} \lev[N]\left( \log N \right)^{-2/3}.$
In the case of the $\liminf,$ for which we are unable to prove a sufficiently sharp decorrelation inequality, we produce a lower bound for the $\liminf$ simply by applying Borel-Cantelli over the whole sequence (a slight, suboptimal, improvement, can be attained easily using eigenvalue interlacing).

The proof of all three steps
rely
heavily on the study  of the kernel $K$.
The upper tail 
\eqref{eq-levuppertail}
is considerably simpler to handle because 
$$\Pr(\mathcal{F}_k)=\det(\Id-K|_{ {\mathcal{I}}_k}),$$
where $K|_{ {\mathcal{I}}_k}$ is the  restriction of $K$ to the single interval
$\mathcal{I}_k:=\{N_k\}\times (s_k,\infty)$, while 
the probability in \eqref{eq-Fk} involves restriction of the kernel 
to two lines. In either case, in handling the upper tail one considers
situations in which the kernel is small, 
and thus tail estimates of the form \eqref{eq-levuppertail} and
\eqref{eq-Fk} 
follow from standard approximations of the determinant and 
(known)
asymptotic expansion of the Hermite polynomials.
%
In contrast, for the lower envelope, substantially more work is required, and
the results are not as sharp. The tail estimates \eqref{eq-levlowertail}
cannot be obtained just from approximation of the kernel $K$,
since one now restricts to the interval $\mathcal{J}_k:=\{N_k\}\times 
(-\infty,-s_k)$, in which the entries of
the kernel
are not exponentially small. 

The first step, namely the left tail asymptotics in  
  \eqref{eq-levlowertail},
could be obtained with some (substantial)
effort by the method of \cite{DIK} used to get similar estimates for the 
Laguerre ensemble and strong asymptotics of the limiting Tracy-Widom tail.
Since we were unable to obtain a sharp result in Theorem \ref{thm:le}, 
we instead use the following
uniform tail bound:
\begin{equation}
  C^{-1} e^{-Ct^3}\leq \Pr \left[
    \lev \leq -t
  \right] \leq C e^{-t^3/C}
  \label{eq:ltub}
\end{equation}
for an absolute constant $C>0$ and all 
$t \leq N^{2/3}$, see \cite[Theorem 1,4]{LedouxRider}.

More difficult is the proof of the second step,
namely the proof of decorrelation estimates
  analogous to \eqref{eq-Fk}. 
  For the upper envelope, these decorrelation estimates are relatively straightforward, and we produce essentially sharp results.
  For the lower envelope, 
  the fact that direct estimates on 
  the restriction of the kernel $K$ to $\mathcal{J}_k$ are not sharp enough
  force us to use a sub-optimal sequence $N_k$; using that
  yields Theorem \ref{thm:le}.
Even with this non-optimal subsequence,
obtaining the decorrelation estimate \eqref{eq-Fk} with $\mathcal{E}_k$
replacing $\mathcal{F}_k$ involves a careful analysis which represents
much of the technical work in this article; we detail the main result 
in Section \ref{sec-decor} below, after we introduce some notation.

Finally, for the proof of the third step, which we only do for the upper envelope, we must essentially show how the kernel $K$ restricted to two lines $N_1$ and $N_2$ degenerates when those lines are separated by less than $N_1^{2/3}.$  A more detailed overview of the argument is provided in Section \ref{sec:corproof} below, after introducing notation.

We conclude this introduction
by noting that
working with the optimal sequence $N_k=k^3$ would allow one to prove
the following.
\begin{conjecture}
	\label{conj-lower}
With notation as above,
\[
  \liminf_{N \to \infty} \frac{\lev}{\left(\log N\right)^{1/3}} = -4^{1/3}.
\]
\end{conjecture}

\subsection*{Structure of the paper}
In Section~\ref{sec-decor}, we define the kernel $K$, and we state the decorrelation and correlation estimates that constitute the main technical work of the proofs of Theorems~\ref{thm:ue} and \ref{thm:le}.
In Section~\ref{sec:ue} we prove the upper limit theorem, \ref{thm:ue}, and 
in Section~\ref{sec:le} we prove the lower limit theorem, \ref{thm:le} using these estimates.
In Section~\ref{sec:contours} we give a double contour integral representation of the kernel $K$ the scaled kernel that is approximated by the Airy kernel.  
In Section~\ref{sec:corproof} we prove the correlation inequality Proposition~\ref{prop:rtcorrelation}, assuming Airy type estimates on the kernel $K.$
In Section~\ref{sec:sharp} we prove these Airy type estimates, as well as \eqref{eq-levuppertail}, using an approximate Hankel representation of the kernel and minimum phase deformations.
In Section~\ref{sec:off}, we prove that the portion of $K$ corresponding to lines $\left\{ u_1 \right\}$ and $\left\{ u_2 \right\}$ where $|u_1 - u_2| \gg u_1^{2/3}$ are small, and their magnitude is controlled by the separation between $u_1$ and $u_2.$  This forms the basis of both decorrelation estimates.
In Section~\ref{sec:boundedness}, we show 
that
the other parts of the kernel remain bounded for well separated $u_1$ and $u_2.$
In Section~\ref{sec:decorproof}, we give the proof of the decorrelation estimates, Proposition~\ref{prop:rt_E} and \ref{prop:lt_E}.
\section{The kernel and decorrelation and correlation estimates}
\label{sec-decor}

In this section, we recall
the GUE minor kernel and describe our basic decorrelation
estimates in terms of it.

\subsection{The kernel}
Define the following table of symbols (\cite[equations 4.9-4.13]{FN}):
\begin{align}
        \phi^{(u_1,u_2)}(x,y)  &= 0, \text{ if } u_1 \geq u_2 \\
        \phi^{(u_1,u_2)}(x,y) &= \frac{1}{(u_2-u_1-1)!}(y-x)^{u_2-u_1-1}\one[y>x], \text{ if } u_1 < u_2 \\
        \Psi_j(x) &= e^{-x^2} H_j(x) \text{ if } j \geq 0 \\
        \Psi_j(x) &= \frac{1}{(-j-1)!}\int_x^\infty (y-x)^{-j-1} e^{-y^2}dy \text{ if } j < 0 \\
        \mathcal{N}_j &= 2^j j! \sqrt{\pi} \\
        \Phi_j(x) &= H_j(x)\frac{1}{\mathcal{N}_j}
        \label{eq:rosetta}
\end{align}
The $H_n(x)$ are the Hermite polynomials normalized so that $\int_\R \Psi_j(x) \Phi_k(x)\, dx = \delta_{j,k}.$
The GUE minor kernel is given by (\cite[equation 4.15]{FN}):
\begin{equation}
  \label{eq:kernel}
        K(u_1,y_1 ; u_2, y_2) = -\phi^{(u_1,u_2)}(y_1,y_2)
        + \sum_{l=1}^{u_2} \Psi_{u_1 - l}(y_1) \Phi_{u_2 - l}(y_2).
\end{equation}
In the case that $u_1 \geq u_2,$ this simplifies to be
\[
  K(u_1,y_1 ; u_2, y_2) = 
  e^{-y_1^2}\sum_{l=1}^{u_2} \frac{H_{u_1 - l}(y_1) H_{u_2 - l}(y_2)}{\mathcal{N}_{u_2 - l}},
\]
which can be identified as the usual GUE kernel when $u_1 = u_2.$  Note that we must multiply this kernel by $e^{y_1^2/2 - y_2^2/2}$ to get the usual self-adjoint GUE kernel, but that the Fredholm determinants of this kernel coincide with the usual self-adjoint one as multiplication by $e^{y_1^2/2 - y_2^2/2}$ is a conjugation of the kernel.
\subsection{Decorrelation estimates} 
Define another kernel 
\begin{equation}
  K^D(u_1,y_1 ; u_2, y_2) = \one[u_1 \leq u_2] K(u_1, y_1 ; u_2, y_2).
  \label{eq:KD}
\end{equation}
It is easily verified that $K^D$ induces a determinantal point process $\GUEPP^D$ on $\Lambda,$ which on each line $\left\{ N \right\} \times \R$ is distributed as the $N$-point GUE and for which 
$\left\{ \GUEPP^D \cap 
  (\left\{ N \right\} \times \R )\right\}_{N=1}^\infty$ are mutually independent.
These kernels are not properly scaled to be comparable, however, so we begin by a scaling.  We let $J$ be a scaling factor (see \eqref{eq:J}) and
let $\tilde K$ be given by
\begin{align}
  \tilde K(u_1, y_1 ; u_2, y_2)
  &=
  \frac{J(u_2,y_2)}{J(u_1,y_1)}
  K(u_1, y_1 ; u_2, y_2) \text{ and } \nonumber \\ 
  {\tilde K}^D(u_1, y_1 ; u_2, y_2)
  &=
  \frac{J(u_2,y_2)}{J(u_1,y_1)}
  K^D(u_1, y_1 ; u_2, y_2). \nonumber 
\end{align}
Let ${\tilde K}_o$ and ${\tilde K}_e$ be defined analogously.  These scalings do not change the associated Fredholm determinants, and hence the associated point processes are unchanged. 

Define
\[
  E(u_1,t_1 ; u_2, t_2)
  =
  \left|
  \Pr\left[ \lev[u_1] \geq t_1 \text{ and } \lev[u_2] \geq t_2 \right]
  -
  \Pr\left[ \lev[u_1] \geq  t_1\right]
  \Pr\left[ \lev[u_2] \geq t_2\right]
  \right|,
\]
and observe that $E$ can also be expressed as 
\[
  E(u_1,t_1 ; u_2, t_2)
  =
  \left|
  \Pr\left[ \lev[u_1] < t_1 \text{ and } \lev[u_2] < t_2 \right]
  -
  \Pr\left[ \lev[u_1] < t_1\right]
  \Pr\left[ \lev[u_2] < t_2\right]
  \right|.
\]
Write
\[
  I = 
\left\{ u_1 \right\} \times [\sqrt{2u_1} + u_1^{-1/6}t_1/\sqrt{2}, \infty) \cup
\left\{ u_2 \right\} \times [\sqrt{2u_2} + u_2^{-1/6}t_2/\sqrt{2}, \infty).
\]
Then we have the identity
\begin{equation}
  E(u_1,t_1;u_2,t_2)
  =
  \left|
  \det( I - \tilde{K}\vert_I)
  -\det( I - {\tilde K}^D\vert_I)
  \right|.
  \label{eq:k_diff}
\end{equation}
Hence by giving pointwise estimates on the kernels and using norm estimates for the differences of Fredholm determinants, we may in turn estimate $E.$
Our main decorrelation estimates are the following.  For the right tail,
\begin{proposition}
	For any $R >0,$
	there are constants $C>0$ and $u_0 >0$ sufficiently large so that for all $0 \leq t_1 \leq R(\log u_1)^{2/3},$ all $0 \leq t_2 \leq R(\log u_2)^{2/3}$ and all $u_1 \geq u_2 + u_2^{2/3} e^{(\log u_1)^{2/3}} \geq u_0,$  
  \[
    \left|
    E(u_1,t_1;u_2,t_2)
    \right| \leq C\frac{u_1^{1/12}u_2^{1/12}}{u_1^{1/2} - u_2^{1/2}} 
    e^{ C( \log u_1)^{5/6} -\tfrac 23 (t_1^{3/2} + t_2^{3/2})}.
  \]
  \label{prop:rt_E}
\end{proposition}
\noindent Note that up to polynomial factors in $t_1$,
$e^{-\tfrac 23 t_1^{3/2}} \sim \Pr\left[ \lev[u_1] \geq t_1\right]^{1/2}.$
For the left tail, we get the same bound, although we lose a multiplicative factor:
\begin{proposition}
  There are constants $C>0$ and $u_0 >0$ sufficiently large so that for all $0 \leq t_1 \leq (\log u_1)^{5/12},$ all $0 \leq t_2 \leq (\log u_2)^{5/12}$ and all $u_1 \geq u_2 + u_2^{2/3} e^{(\log u_1)^{2/3}} \geq u_0,$  
  \[
    \left|
    E(u_1,-t_1;u_2,-t_2)
    \right| \leq C\frac{u_1^{1/12}u_2^{1/12}}{u_1^{1/2} - u_2^{1/2}} e^{ C( \log u_1)^{5/6} }.
  \]
  \label{prop:lt_E}
\end{proposition}

\subsection{Correlation estimate}
\label{sec-cor}
We also show correlation estimates for $\lev[u_i]$ when $u_1^{1/3} \ll u_2 - u_1 \ll u_1^{2/3}.$  It turns out not to be necessary to show an estimate for smaller values of $u_2 - u_1,$ as for those values we can use eigenvalue interlacing. 

Define
\begin{align*}
  F(u_1,t_1 ; u_2, t_2)
  &=
  \Pr\left[ \lev[u_1] \geq t_1 \text{ and } \lev[u_2] < t_2 \right] \\
  &=
  \left|
  \Pr\left[ \lev[u_1] < t_1 \text{ and } \lev[u_2] < t_2 \right]
  -
  \Pr\left[ \lev[u_2] < t_2\right]
  \right|.
\end{align*}
We seek to show that this is much smaller in order than 
$\Pr\left[ \lev[u_1] > t_1\right].$  
See Section~\ref{sec:corproof} for an overview of the approach.

Our main correlation result is:
\begin{proposition}
	\label{prop:rtcorrelation}
	 In what follows, we let $\Delta u = u_2 - u_1$ and use $u=u_1.$ 
  For any $0<\beta < \delta < \tfrac16$ and $\epsilon >0$ there is a $C>0$ sufficiently large so that 
  \begin{enumerate}
	  \item for all $u_1,u_2 \in \mathbb{N}$ with $u^{1/3+\delta} \leq \Delta u \leq u^{2/3-\delta},$ and
	  \item for all $t_1,t_2 \in \R$ and $0 \leq \Delta t \leq 1$ with 
		  \[
			  \epsilon (\log u)^{2/3} \leq t_2 \leq t_2 + \Delta t \leq t_1 \leq \frac{1}{\epsilon} (\log u)^{2/3},
		  \]
  \end{enumerate}
  we have that
  \[
F(u_1,t_1 ; u_2, t_2)
\leq
C
\left[\frac{(\Delta u)}{u^{2/3-\beta}}
  +
  \Pr\left[ Z > \Delta t u^{1/3}/\sqrt{\Delta u} \right]
\right]
e^{-\tfrac 23( (t_1)^{3/2} + t_2^{3/2})},
  \]
  where $Z$ is a standard normal variable.
\end{proposition}
\noindent The proof is given in Section \ref{sec:corproof}.
\section{Proof of the upper limit, Theorem \ref{thm:ue}}
\label{sec:ue}

\begin{proof}[Proof of Theorem~\ref{thm:ue}]
	Fix $\alpha > 3,$ and
	define $N_k = \lceil k^{\alpha} \rceil$ for all $k \in \mathbb{N}.$ 
	Let $c_* = \left(\frac{1}{4}\right)^{2/3},$ and for some fixed $c < c_*$ define
  \[
  \mathcal{E}_k =  \left\{ \lev[N_k] \geq c (\log N_k)^{2/3} \right\}.
  \]
  Define $S_N = \sum_{k=1}^N \one[ \mathcal{E}_k ].$  
  We will show that $S_N \to \infty$ in probability for sufficiently small $c$ by a second moment calculation, from which it follows that infinitely many $\mathcal{E}_k$ occur almost surely.  Further, we will show that by making $\alpha$ close to $3,$ we can take $c$ close to $c_*.$  Hence we will have shown that
  \[
	  \limsup_{N \to \infty} \frac{\lev[N]}{(\log N)^{2/3}} \geq c_*.
  \]
  
  By \eqref{eq-levuppertail} (proven in Lemma~\ref{lem:levuppertail}),
  \[
    \Pr \left( \mathcal{E}_k \right) 
    = \Omega\left( N_k^{-\tfrac 43 c^{3/2}} \right)
    = \Omega\left( k^{- \alpha \tfrac43 c^{3/2}} \right).
  \]
  Letting $\beta = 4c^{3/2},$ which we observe has $\beta < 1,$ 
  \begin{equation}
	  \Exp S_N = \Omega\left( N^{1 - \alpha \tfrac{\beta}{3} } \right).
    \label{eq:rt_expSN}
  \end{equation}

  As for the variance, we have that
  \begin{align*}
    \Var( S_N ) 
    &=\Exp S_N^2-(\Exp S_N)^2\\
    & \leq
     \Exp S_N + 2\sum_{k=1}^N \sum_{\ell > k}^N
    \left[
    \Pr \left( \mathcal{E}_k \cap \mathcal{E}_\ell \right) 
    -
    \Pr \left( \mathcal{E}_k \right) 
    \Pr \left( \mathcal{E}_\ell \right) 
  \right] \\
    &= \Exp S_N + 2\sum_{k=1}^N \sum_{\ell > k}^N
    E(N_k, c(\log N_k)^{2/3}; N_\ell, c( \log N_\ell)^{2/3}).
    \label{eq:rt_var1}
  \end{align*}
  As $\alpha > 3,$ we may 
  apply Proposition \ref{prop:rt_E} to get that for any $\delta>0$
  \begin{align*}
    \sum_{\ell > k}^N
    E(N_k, c(\log N_k)^{2/3}; N_\ell, c( \log N_\ell)^{2/3})
    &=
    \sum_{\ell > k}^N
    O\left(
    \frac{\ell^{\alpha/12-\beta/2}k^{\alpha/12-\beta/2}}{\ell^{\alpha/2} - k^{\alpha/2}} N^{\delta}
    \right).
  \end{align*}
  Divide this sum into those terms $\ell < 2k$ and those terms $\ell \geq 2k.$  For terms less than $2k,$ use that $\ell^{\alpha/2} \geq k^{\alpha/2} + (\alpha/2 -1)(\ell-k)k^{\alpha/2-1}.$  For terms $\ell \geq 2k,$ just use that $\ell^{\alpha/2} - k^{\alpha/2} = \Omega( \ell^{\alpha/2}).$  Hence we have that
 \begin{align*}
    \sum_{\ell > k}^N
    \frac{\ell^{\alpha/12-\beta/2}k^{\alpha/12-\beta/2}}{\ell^{\alpha/2} - k^{\alpha/2}}
    &\leq
    \sum_{\ell > k}^{2k}
    \frac{\ell^{\alpha/12-\beta/2}k^{\alpha/12-\beta/2}}{(\alpha/2-1)(\ell-k)}
    +\sum_{\ell > 2k}^{N}
    O\left(\frac{\ell^{\alpha/12-\beta/2}k^{1-5\alpha/12-\beta/2}}{\ell^{\alpha/2}}\right) \\
    &=O\left(k^{1-\alpha/3-\beta} \log k\right).
  \end{align*}
  Hence, applying this to the variance, we have
  \begin{align*}
    \Var( S_N ) 
    &\leq 
    \Exp S_N + O\left(N^{2-\alpha/3 -\beta+ 2\delta}\right).
  \end{align*}
  As we may shrink $\delta$ to be as small as desired, it suffices to have $2-\alpha/3 - \beta < 2 - 2\beta \alpha/3$ in order to have $\Var(S_N) = o((\Exp S_N)^2).$ Hence provided that $4c^{3/2} = \beta < \frac{\alpha}{2\alpha -3},$ we have that $\mathcal{E}_k$ occur infinitely often.  As we may take $\alpha$ arbitrarily close to $3,$ we may make $c$ as close to $c_*$ as desired.

  We now turn to showing that 
  \[
	  \limsup_{N \to \infty} \frac{\lev[N]}{(\log N)^{2/3}} \leq c_*,
  \]
  Fix $\alpha < 3$ and define $N_k=\lceil k^{\alpha} \rceil.$  Fix $\delta>0$ to be chosen later, and define 
  \(
  \mathcal{N}_k = \left\{ N_k - j\lceil N_k^{1/3}\rceil~:~0\leq j \leq N_k^\delta \right\}.
  \)
  Fix $c > c_*$ and define  
  \[
  \mathcal{E}_k = \left\{ \exists~j \in \mathcal{N}_k~:~ \lev[j] \geq c(\log j)^{2/3} \right\}.
  \]
  Then from \eqref{eq-levuppertail} (proved in
  Lemma~\ref{lem:levuppertail}),
  we have
  \[
    \Pr \left( \mathcal{E}_k \right) 
    = O\left( N_k^{-\tfrac 43 c^{3/2}+\delta} \right)
    = O\left( k^{- \alpha (\tfrac43 c^{3/2}-\delta)} \right).
  \]
  We thus see that
  for any choice of $c > c_*$ we can choose $\alpha$ sufficiently close to $3$ and $\delta$ sufficiently close to $0$ that this is summable in $k.$  Hence by Borel-Cantelli, only finitely many $\mathcal{E}_k$ occur almost surely.  
  
  As we wish to bound the $\limsup$ from above, we need to control of $\lev[n]$ for all $\mathbb{N}$.  We do this by first extending control to a denser net of $\mathbb{N}$ using Proposition~\ref{prop:rtcorrelation}.  Having done so, 
  we will have a sufficiently dense 
  net
  that we can apply eigenvalue interlacing to conclude 
  the upper bound
  for the full sequence.  
  
  Define 
  \(
  \mathcal{A}_k = \left\{ 
    N_k - j\lceil N_k^{1/3}\rceil - \ell \lceil N_k^{1/3+\delta}\rceil ~:~
    0\leq j \leq N_k^\delta, 
    0 \leq \ell \leq N_k^{1/3-2\delta}
  \right\}
  \)
  and define
  \(
  \mathcal{A} = \cup_{k=1}^\infty \mathcal{A}_k.
  \)
  We claim that for $\delta$ sufficiently small, $\mathcal{A}$ has the property that for all $n \in \mathbb{N}$ larger than some $n_0,$ there is a $j \in \mathcal{A}$ so that $j \geq n \geq j - 2n^{1/3}.$  On the one hand, the spacing between consecutive elements of $\mathcal{A}_k$ is never more than $\lceil N_k^{1/3}\rceil.$  On the other hand, 
  \[
    \min \mathcal{A}_k = N_k - N_k^{2/3-\delta} + O(N_k^{1/3-\delta}).
  \]
  Hence, by making $\delta$ sufficiently small, we have that $N_{k-1} \geq \min \mathcal{A}_k$ for all $k$ large.  Thus for all $n$ with $N_{k-1} < n \leq N_k$ for $k$ sufficiently large, we have shown that there is a $j \in \mathcal{A}_k$ so that $j \geq n \geq j -\lceil N_k^{1/3}\rceil.$  Since $N_{k}/N_{k-1}\to 1,$ we may bound $\lceil N_k^{1/3}\rceil \leq 2n^{1/3}$ for all $k$ sufficiently large.

  We will eventually show that for $\delta$ sufficiently small, there are almost surely only finitely many $j \in \mathcal{A}$ so that 
  $\lev[j] > (c+\delta)(\log {j})^{2/3}.$  First, we will show how this implies there are only finitely many $n \in \mathbb{N}$ so that $\lev[n] > 
  (c+2\delta)(\log {n})^{2/3}.$  Using the property shown above for $\mathcal{A}$, we have that for any $n > n_0$ random, there is a $j \in \mathcal{A}$ with $j \geq n \geq j - 2n^{1/3}$ having $\lev[j] \leq (c+\delta)(\log j)^{2/3}.$  Recall that the unscaled eigenvalues satisfy $\levraw[n] \leq \levraw[j],$ and hence
  \begin{align*}
    \lev[n]
    &\leq \left( \sqrt{2j} - \sqrt{2n} \right)(\sqrt{2})n^{1/6}
    + (c+\delta)(\log j)^{2/3}\left( \frac{n}{j} \right)^{1/6} \\
    &\leq \left( j - n \right)\frac{n^{1/6}}{\sqrt{2}n^{1/2}}
    + (c+\delta)(\log(n + (j-n)) )^{2/3}\left( \frac{n}{j} \right)^{1/6} \\
    &\leq \sqrt{2}
    + (c+\delta)(\log(n) + 2n^{-2/3} )^{2/3}\\
    &\leq (c+2\delta)(\log n)^{2/3},
  \end{align*}
  for all $n$ sufficiently large.  Thus if we show that almost surely only finitely many $j \in \mathcal{A},$ then we conclude that almost surely
  \[
    \limsup_{n\to\infty} \frac{\lev[n]}{(\log n)^{2/3}} \leq c+2\delta.
  \]
  As we may make $c$ as close to $c_*$ and $\delta$ as close to $0$ as we wish, this will complete the proof.

  As for the claim about $\mathcal{A},$ we define for any $k \in \mathbb{N},$ any $j \in \mathcal{N}_k$ the set of numbers
  \[
	  \mathcal{U}_{k,j} =
	  \left\{ 
	  j-\ell\lceil N_k^{1/3+\delta}\rceil~:~ 
	    0 \leq \ell \leq \lfloor N_k^{1/3-2\delta}\rfloor
	  \right\},
  \]
  and the event
  \[
    \mathcal{E}_{k,j} = \left\{ 
	    \exists\ n \in \mathcal{U}_{k,j}
	    ~:~
	    \lev[n] > (c+\delta)(\log n)^{2/3}
      \text{ and }
      \lev[j] \leq c(\log j)^{2/3}
  \right\}.
  \]
  We begin by estimating $\Pr\left( \mathcal{E}_{k,j} \right).$  To this end, we will do a dyadic decomposition of $\mathcal{U}_{k,j}$.  Let $u_* = \min \mathcal{U}_{k,j}$ and $u^* = \max \mathcal{U}_{k,j}.$   Define $n_{0,\ell} = u_*$
  and define $n_{2^\ell,\ell} = u^*$
  for all integers $\ell \geq 0.$  Now define, inductively on $\ell$:
  \begin{itemize}
    \item 
      For all $0 \leq i \leq 2^{\ell-1},$ define $n_{2i,\ell} = n_{i,\ell-1}.$
    \item
      For all $0 \leq i < 2^{\ell-1},$ define
      $n_{2i+1,\ell}$ as a median of $\mathcal{U}_{k,j} \cap (n_{i,\ell-1}, n_{i+1,\ell-1}),$ if one exists or $n_{2i+2,\ell}$ otherwise.
  \end{itemize}
  As $|\mathcal{U}_{k,\ell}| \leq N_k^{1/3}$ for all $k$ large, there is a $C>0$ so that for all $\ell > C\log N_k$ and all $0 \leq i < 2^{\ell-1},$ $n_{2i+1,\ell} = n_{2i+2,\ell}$  In particular, we have that
  \[
	  \mathcal{U}_{k,j} \subseteq \left\{ 
		  n_{2i+1,\ell}~:~ 1 \leq \ell \leq C\log N_k, 0\leq i \leq 2^{\ell-1}
	  \right\}.
  \]
  Set 
    $\bar \beta = \tfrac 43 c^{3/2} > \tfrac{1}{3}.$ 
  Set $t_0 = c(\log j)^{2/3}$ and define $t_\ell = t_0 + \frac{\ell}{C\log N_k}$for all $\ell > 0.$  Then we have the estimate for all $k$ sufficiently large that
  \begin{align*}
	  \Pr\left( \mathcal{E}_{k,j} \right)
	  &\leq 
	  \sum_{\ell=1}^{\lfloor C\log N_k \rfloor}
	  \sum_{i=0}^{2^{\ell-1}}
	  \Pr\left( 
	  \lev[{n_{2i+1,\ell}}] > t_\ell
      \text{ and }
      \lev[{n_{2i+2,\ell}}] \leq t_{\ell-1}
	  \right). \\
	  \intertext{Applying Proposition~\ref{prop:rtcorrelation} 
	  with $\beta=\delta/2$,}
	  &\leq 
	  \sum_{\ell=1}^{\lfloor C\log N_k \rfloor}
	  \sum_{i=0}^{2^{\ell-1}}
	  O\left(\frac{ {n_{2i+1,\ell}} -{n_{2i+2,\ell}}}
	  { N_k^{2/3-{2.5}\delta}  } N_k^{-{\bar
	  \beta}} \right)\\
 &\leq 
	  \sum_{\ell=1}^{\lfloor C\log N_k \rfloor}
	  O(N_k^{{2\delta-\bar\beta}}) \\
 &\leq 
 O(N_k^{{2\delta-\bar\beta}}\log N_k  ).
  \end{align*}

    Hence, summing over $j \in \mathcal{N}_k$ we have
  \begin{align*}
    \sum_{j \in \mathcal{N}_k} 
    \Pr\left(  \mathcal{E}_{k,j}\right)
    &=
    O\left( N_k^{{3\delta-\bar\beta}}\right).
  \end{align*}
Hence, for $\delta$ sufficiently small, this is summable 
in $k$, and the proof is complete.
  \end{proof}

\section{Proof of the lower limit, Theorem \ref{thm:le}}
\label{sec:le}
This proof is nearly identical to the previous one, but with some small numerical changes to account for the differences in Proposition~\ref{prop:rt_E} and \ref{prop:lt_E}.
\begin{proof}
	By 
	the Borel-Cantelli lemma, 
	the existence of $c_1$ follows from \eqref{eq:ltub}.  The proof is therefore devoted to showing the existence of $c_2.$  This proof is nearly identical to part of the proof of Theorem~\ref{thm:ue}.  Let $\alpha > 6$ be fixed, and define $N_k = \lceil k^{\alpha} \rceil$ for all $k \in \mathbb{N}.$  For some $c_2 >0$ to be determined, define the event
  \[
  \mathcal{E}_k =  \left\{ \lev[N_k] \leq -c_2 (\log N_k)^{1/3} \right\}.
  \]
  Define $S_N = \sum_{k=1}^N \one[ \mathcal{E}_k ].$  
  We will show that $S_N \to \infty$ in probability for sufficiently small $c_2$ by a second moment calculation.
  
  From \cite[Theorem 4]{LedouxRider}
  there is some $\beta > 0$ so that
  \[
    \Pr \left( \mathcal{E}_k \right) 
    = \Omega\left( N_k^{-\beta c_2^3} \right)
    = \Omega\left( k^{-\alpha \beta c_2^3} \right).
  \]
  Hence we have that for $c_2$ so that $\alpha \beta c_2^3 < 1,$ 
  \begin{equation}
    \Exp S_N = \Omega\left( N^{1 -\alpha \beta c_2^3} \right).
    \label{eq:expSN}
  \end{equation}

  As for the variance, we have that
  \begin{align*}
    \Var( S_N ) 
    &\leq \Exp S_N + 2\sum_{k=1}^N \sum_{\ell > k}^N
    \left[
    \Pr \left( \mathcal{E}_k \cap \mathcal{E}_\ell \right) 
    -
    \Pr \left( \mathcal{E}_k \right) 
    \Pr \left( \mathcal{E}_\ell \right) 
  \right] \\
    &= \Exp S_N + 2\sum_{k=1}^N \sum_{\ell > k}^N
    E(N_k, -c_2(\log N_k)^{1/3}; N_\ell, -c_2( \log N_\ell)^{1/3}).
    \label{eq:var1}
  \end{align*}
  Applying Proposition \ref{prop:lt_E}, we have that for any $\delta>0$
  \begin{align*}
    \sum_{\ell > k}^N
    E(N_k, -c_2(\log N_k)^{1/3}; N_\ell, -c_2( \log N_\ell)^{1/3})
    &=
    \sum_{\ell > k}^N
    O\left(
    \frac{\ell^{\alpha/12}k^{\alpha/12}}{\ell^{\alpha/2} - k^{\alpha/2}} N^{\delta}
    \right).
  \end{align*}
  Divide this sum into those terms $\ell < 2k$ and those terms $\ell \geq 2k.$  For terms less than $2k,$ use that $\ell^{\alpha/2} \geq k^{\alpha/2} + (\alpha/2 -1)(\ell-k)k^{\alpha/2-1}.$  For terms $\ell \geq 2k,$ just use that $\ell^{\alpha/2} - k^{\alpha/2} = \Omega( \ell^{\alpha/2}).$  Hence we have that
 \begin{align*}
    \sum_{\ell > k}^N
    \frac{\ell^{\alpha/12}k^{\alpha/12}}{\ell^{\alpha/2} - k^{\alpha/2}}
    &\leq
    \sum_{\ell > k}^{2k}
    \frac{\ell^{\alpha/12}k^{1-5\alpha/12}}{(\alpha/2-1)(\ell-k)}
    +\sum_{\ell > 2k}^{N}
    O\left(\frac{\ell^{\alpha/12}k^{\alpha/12}}{\ell^{\alpha/2}}\right) \\
    &=O\left(k^{1-\alpha/3} \log k\right).
  \end{align*}
  Hence, applying this to the variance, we have
  \begin{align*}
    \Var( S_N ) 
    &\leq
    \Exp S_N + O\left(N^{2-\alpha/3 + 2\delta}\right).
  \end{align*}
  As we may shrink $\delta$ to be as small as desired, it suffices to have $2-\alpha/3 < 2 - 2\alpha\beta c_2^3$ in order to have $\Var(S_N) = o((\Exp S_N)^2).$  This requires that $\beta c_2^3 < 1.$  Conversely, letting $c_2$ be any positive number satisfying $ \alpha \beta c_2^3 < 1,$ we have that $S_N \to \infty$ in probability.
\end{proof}

\section{Contour integral representations for the kernel}
\label{sec:contours}
We begin with the following identity for $\Psi_j(x).$  
\begin{lemma}
  \label{lem:Psi}
  For all integer $j,$
\begin{equation*}
  \Psi_j(x) = \frac{2^j}{\sqrt{\pi} i} \int_\ell s^j e^{s^2-2xs}\,ds.
\end{equation*}
The contour $\ell$ is any vertical line in the complex plane, travelled in the direction of increasing imaginary part, whose real part is positive.
\end{lemma}
\begin{proof}
  In the case that $j \geq 0,$ this formula is standard.  The case $j < 0$ follows from (2) of \cite{JNerratum}.
\end{proof}

As for $\Phi_j(x)$, we can represent a Hermite polynomial as
\begin{equation*}
  H_j(x) = \frac{j!}{2\pi i} \oint \frac{e^{-z^2 + 2xz}}{z^j} \frac{dz}{z},
\end{equation*}
where the contour is any that winds once around $0.$  Thus, we have the representation
\begin{equation}
  \label{eq:Phi}
  \Phi_j(x) = \frac{2^{-j}}{2\pi^{3/2} i} \oint \frac{e^{-z^2 + 2xz}}{z^j} \frac{dz}{z}.
\end{equation}

Expand \eqref{eq:kernel} by replacing $\Psi_j(x)$ and $\Phi_j(x)$ with Lemma \eqref{lem:Psi} and \eqref{eq:Phi}.  This gives the representation
\begin{equation*}
        \phi+K=
        \frac{2^{u_1-u_2}}{2(\pi i)^2}
        \sum_{k=1}^{u_2}
        \oint\limits_{z_2}\int\limits_{z_1}
        \frac{e^{z_1^2 - 2z_1y_1}}{e^{z_2^2 - 2z_2y_2}}
        \left( \frac{z_2}{z_1} \right)^k
        \frac{z_1^{u_1}}{z_2^{u_2}}
        \frac{dz_1dz_2}{z_2}.
\end{equation*}
Taking the contours so that the 
$z_1$ and $z_2$ contours do not intersect,
and evaluating the contour integral over $z_2$ first, we see, using the analyticity of the integrand,
that 
the last expression is not changed if the sum is extended to $\infty.$ 
Taking the contours so that $|z_1| > |z_2|,$ the series is uniformly convergent.  As this is a geometric series, we arrive at the equation
\begin{equation}
  \label{eq:phi+K}
        \phi + K=
        \frac{2^{u_1-u_2}}{2(\pi i)^2}
        \oint\limits_{z_2}\int\limits_{z_1}
        \frac{e^{z_1^2 - 2z_1y_1}}{e^{z_2^2 - 2z_2y_2}}
        \frac{z_1^{u_1}}{z_2^{u_2}}
        \frac{dz_1dz_2}{z_1 - z_2}.
\end{equation}
The $z_2$ integral is taken over a closed loop that winds once around $0,$ and the $z_1$ integral is taken over a vertical line with real part larger than any part of the $z_2$ contour.

\subsection{ Contour deformation }

At this point, we will deform the contours
to be $\yM[i]$-independent, approximate minimum phase contours (see Figure~\ref{fig:contours_1geq2}).  We will use these contours, or slight deformations of them, for most of our estimates.  For $\yM[i]$ positive, we will also use more exact $\yM[i]$-dependent approximate minimum phase contours in Section~\ref{sec:sharp}.

Fix parameters $\delta_1 > 0$ and $\delta_2 > 0$ to be determined later (see the proof of Lemma~\ref{lem:Ko_pointwise}).  Define the following collection of straight-line contours:
\begin{align}
  \label{eq:contours}
  {\gamma}_1 &= u_1^{1/2}\left[\tfrac{1}{\sqrt{2}},\tfrac{1}{\sqrt{2}}+\tfrac{\delta_1}{\sqrt{2}} e^{i\pi/3}\right],  
  &  {\gamma}_1^e &= u_1^{1/2}\left(\tfrac{1}{\sqrt{2}} + \tfrac{\delta_1}{\sqrt{2}} e^{i\pi/3}\right) + i\R_+, \\
  {\gamma}_2 &= u_2^{1/2}\left[\tfrac{1}{\sqrt{2}},\tfrac{1}{\sqrt{2}}+\tfrac{\delta_2}{\sqrt{2}} e^{i 2\pi/3}\right],  
  &  {\gamma}_2^e &= u_2^{1/2}\left(\tfrac{1}{\sqrt{2}} + \tfrac{\delta_2}{\sqrt{2}} e^{i 2\pi/3}\right) + \R_{-}, 
\end{align}
Define $\gamma_1^c$ and $\gamma_2^c$ to be the piecewise linear contours 
\begin{align*}
  \gamma_1^c &= \overline{\gamma_1^e} \cup \overline{\gamma_1} \cup \gamma_1 \cup \gamma_1^e, \\
  \gamma_2^c &= \overline{\gamma_2^e} \cup \overline{\gamma_2} \cup \gamma_2 \cup \gamma_2^e, 
\end{align*}
oriented to have non-decreasing imaginary part.  Define
\begin{equation}
  \label{eq:odK}
  \odK(u_1,y_1 ; u_2, y_2 ) = \frac{2^{u_1-u_2}}{2(\pi i)^2}
        \int\limits_{\gamma_2^c}\int\limits_{\gamma_1^c}
        \frac{e^{z_1^2 - 2z_1y_1}}{e^{z_2^2 - 2z_2y_2}}
        \frac{z_1^{u_1}}{z_2^{u_2}}
        \frac{dz_1dz_2}{z_1 - z_2}.
      \end{equation}
and, define $K_e = K - K_o$

When $u_1 \geq u_2,$ it is easily seen that the contours in \eqref{eq:phi+K} can be deformed to $\gamma_1^c$ and $\gamma_2^c$ respectively, so that $K_e = 0.$  
When $u_1 < u_2,$ we would still like to use the contours $\gamma_{1}^c$ and $\gamma_{2}^c,$ however, these contours cross, so that deforming the contours contributes a nonzero residue.  Further, when $u_1 < u_2,$ we must account for $\phi.$  For the remainder of the section, we assume that $u_1 < u_2.$

We will begin by giving a representation of $\phi$ which is useful for our purposes.  The contours $\gamma_1^c$ and $\gamma_2^c$ intersect at exactly two points, which are conjugates.  Let $\tau$ be the intersection point with positive imaginary part.  Let $\gamma^r$ be the contour that follows $\gamma_2^c$ from $\overline{\tau}$ to $\tau$ and which follows the vertical line through $\tau$ and $\overline{\tau}$ outside $\gamma_2.$  Orient $\gamma^r$ to have increasing imaginary part.  

Let $\gamma_{+}^r$ be the portion of $\gamma^r$ below $\overline{\tau}$ 
and  above $\tau$,
and let $\gamma_{-}^r$ be the portion of $\gamma^r$ that follows $\gamma_2^c.$  By adding a half loop to $\gamma_1^c$ that connects $\tau$ and $\overline{\tau}$ through the right-most component of $\C \setminus \gamma_2^c,$ 
we get the identity
\begin{equation}
  \phi + K - K_o = \frac{1}{\pi i} \int_{\gamma_{-}^r} \frac{e^{2z_2(y_2 - y_1)}}{(2z_2)^{u_2 - u_1}}\, d z_2.
  \label{eq:cont_res}
\end{equation}

We next represent $\phi$ as an integral.
From the residue theorem, we have that 
\begin{equation}
  \phi^{(u_1,u_2)}(y_1,y_2)
  = \frac{\one[y_2 > y_1]}{2\pi i}
  \oint \frac{e^{\xi (y_2 - y_1)}}{\xi^{u_2 - u1}}\, d\xi,
  \label{eq:phi_res_0}
\end{equation}
with the contour positively winding once around $0.$  As we have that 
$u_1 < u_2$ we can deform this contour to follow $\gamma^r$.\footnote{In the case
$u_2=u_1+1$, we must take the principal value at infinity.}
Additionally setting $\xi = 2z_2,$ we have
\begin{equation}
  \phi^{(u_1,u_2)}(y_1,y_2)
  = \frac{\one[y_2 > y_1]}{\pi i}
  \operatorname{pv}
  \int_{\gamma^r} \frac{e^{2z_2(y_2 - y_1)}}{(2z_2)^{u_2 - u_1}}\, d z_2.
  \label{eq:phi_res}
\end{equation}

Combining \eqref{eq:phi_res} and \eqref{eq:cont_res}, we have the following piecewise representation of $K_e$ when $u_1 < u_2-1$:\footnote{Again, principal values at infinity need to be used if $u_2=u_1+1$}
\begin{equation}
  K_e(u_1,y_1;u_2,y_2) 
  = \begin{cases}
    -\frac{1}{\pi i} \int_{\gamma_{+}^r} \frac{e^{2z_2(y_2 - y_1)}}{(2z_2)^{u_2 - u_1}}\, d z_2, & y_2 > y_1, \\
\frac{1}{\pi i} \int_{\gamma_{-}^r} \frac{e^{2z_2(y_2 - y_1)}}{(2z_2)^{u_2 - u_1}}\, d z_2, & y_2 \leq y_1.
  \end{cases}
  \label{eq:Ke}
\end{equation}

\begin{figure}
\begin{center}
\begin{tikzpicture}
  [ endstyle/.style={circle,fill=blue,inner sep=1.5pt}
  ]

  \begin{scope}[shift={(-1.5,0)}]
    \draw [dotted,-] (-1.0,0) -- (3.0,0);
    \draw [dotted,->] (3.0,0) -- (5.0,0) node [above left]  {$\Re\{z\}$};
    \draw [dotted,->] (0,-1.7) -- (0,1.7) node [below right] {$\Im\{z\}$};

  \end{scope}

  \begin{scope}[shift={(1.0,0)}]

  \coordinate [endstyle] (z1top) at (60:0.75cm)  {};
  \node [endstyle] (z1bot) at (300:0.75cm) {};
  \node [endstyle] (z1origin) at (0,0) {};
  \node at (-0.4,-0.4) {$\sqrt{\tfrac{u_1}{2}}$};
  \draw [->] (z1origin) -- node[anchor=east] {$\gamma_1$} (z1top);
  \draw [->] (z1top) -- node[anchor=east] {$\gamma_1^e$} +(0,1cm);
  \draw [->] (z1bot) -- (z1origin);
  \draw [<-] (z1bot) -- +(0,-1cm);
  \end{scope}

  \begin{scope}[shift={(-1.0,0)}]
  \coordinate [endstyle] (z2top) at (120:1.5cm)  {};
  \node [endstyle] (z2bot) at (240:1.5cm) {};
  \node [endstyle] (z2origin) at (0,0) {};
  \node at (0.2,-0.4) {$\sqrt{\tfrac{u_2}{2}}$};
  \draw [->] (z2origin) -- node[anchor=west] {$\gamma_2$} (z2top);
  \draw [->] (z2top) -- node[anchor=north] {$\gamma_2^e$} +(-1cm,0);
  \draw [->] (z2bot) -- (z2origin);
  \draw [<-] (z2bot) -- +(-1cm,0);
  \end{scope}

\end{tikzpicture}
\end{center}
\caption{
  The contours over which we will eventually estimate $K(u_1,y_1;u_2,y_2),$ with $u_1 \geq u_2.$ The values of $\delta_1$ and $\delta_2$ are fixed positive constants determined in the proof of Lemma~\ref{lem:Ko_pointwise}.
}
\label{fig:contours_1geq2}
\end{figure}
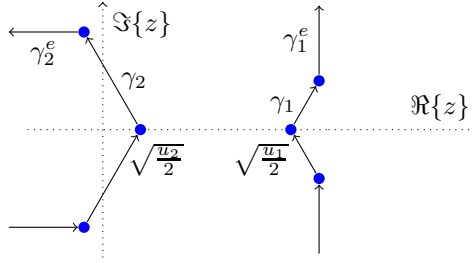

\begin{figure}
\begin{center}
\begin{tikzpicture}
  [ endstyle/.style={circle,fill=black,inner sep=1.0pt},
          taustyle/.style={circle,fill=blue,inner sep=1.5pt},
          g1/.style={black},
          g2/.style={blue},
          gr/.style={very thick,Maroon,dashed}
  ]

  \begin{scope}[shift={(-2,0)}]
    \draw [dotted,-] (-2.5,0) -- (1.0,0) node [above left]  {$\Re\{z\}$};
    \draw [dotted,->] (-1.7,-2.0) -- (-1.7,2.0) node [above left] {$\Im\{z\}$};

  \begin{scope}[shift={(-1.5,0)}]

  \coordinate [endstyle] (z1top) at (60:1.5cm)  {};
  \node [endstyle] (z1bot) at (300:1.5cm) {};
  \node [endstyle] (z1origin) at (0,0) {};
  \draw [g1,->] (z1bot) -- (z1origin) --  (z1top) -- +(0,0.5cm);
  \draw [g1,-] (z1bot) -- node[anchor=west] {$\gamma_1^c$}  +(0,-0.5cm);
  \end{scope}

  \begin{scope}[shift={(-1.0,0)}]
  \coordinate [endstyle] (z2top) at (120:2.0cm)  {};
  \node [endstyle] (z2bot) at (240:2.0cm) {};
  \node [endstyle] (z2origin) at (0,0) {};
  \draw [g2,->] (z2bot) -- (z2origin) -- (z2top) -- node[anchor=north] {$\gamma_2^c$} +(-0.5cm,0);
  \draw [g2,-] (z2bot) -- +(-0.5cm,0);

  \node [taustyle] (tau) at (120:0.5cm) {};
  \node at (0.05,0.45) {$\tau$};
  \node [taustyle] (taub) at (240:0.5cm) {};
  \node at (0.05,-0.4) {$\overline{\tau}$};
  \draw [gr,->] (taub) -- (z2origin) -- (tau) -- +(0,1.5cm);
  \draw [gr,-] (taub) -- node[below left] {$\gamma^r$} +(0,-1.5cm);

  \end{scope}
  \end{scope}

  \begin{scope}[shift={(3,0)}]
    \draw [dotted,-] (-2.5,0) -- (1.0,0) node [above left]  {$\Re\{z\}$};
    \draw [dotted,->] (-1.7,-2.0) -- (-1.7,2.0) node [above left] {$\Im\{z\}$};

  \begin{scope}[shift={(-1.65,0)}]

  \coordinate [endstyle] (z1top) at (60:0.35cm)  {};
  \node [endstyle] (z1bot) at (300:0.35cm) {};
  \node [endstyle] (z1origin) at (0,0) {};
  \draw [g1,-] (z1bot) -- (z1origin) node[anchor=east] {$\gamma_1^c$}   --  (z1top) -- +(0,1.5cm);
  \draw [g1,-] (z1bot) -- +(0,-1.5cm);
  \end{scope}

  \begin{scope}[shift={(-1.0,0)}]
  \coordinate [endstyle] (z2top) at (120:2.0cm)  {};
  \node [endstyle] (z2bot) at (240:2.0cm) {};
  \node [endstyle] (z2origin) at (0,0) {};
  \draw [g2,->] (z2bot) -- (z2origin) -- (z2top) -- node[anchor=north] {$\gamma_2^c$} +(-0.5cm,0);
  \draw [g2,-] (z2bot) -- +(-0.5cm,0);

  \node [taustyle] (tau) at (120:0.95cm) {};
  \node at (-0.2,0.85) {$\tau$};
  \node [taustyle] (taub) at (240:0.95cm) {};
  \node at (-0.2,-0.8) {$\overline{\tau}$};
  \draw [gr,->] (taub) -- (z2origin) -- (tau) -- +(0,1.1cm);
  \draw [gr,-] (taub) -- node[below right] {$\gamma^r$} +(0,-1.1cm);

  \end{scope}
  \end{scope}

\end{tikzpicture}
\end{center}
\caption{
  The contours $\gamma_1^c,\gamma_2^c$ and $\gamma^r$ when $u_1 \leq u_2;$ the true picture will be one of these. }
\label{fig:contours_1leq2}
\end{figure}
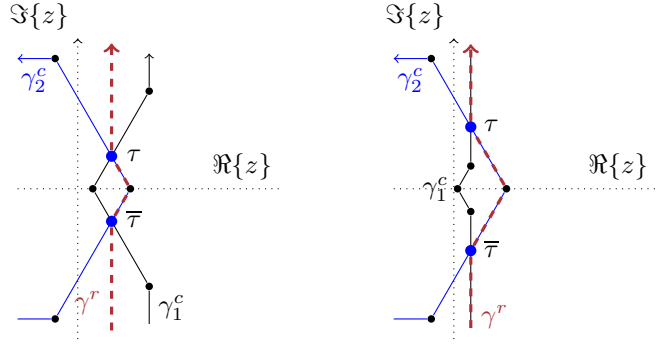

\subsection{Scaling}

Define the scaled variables 
\begin{equation}
  \label{eq-scaling}
  \zM[i] = 2^{1/2}u_i^{-1/6}z_i - u_i^{1/3}\,,\quad
  \yM[i] = 2^{1/2}u_i^{1/6}y_i - 2 u_i^{2/3}\,.
\end{equation}
  Substituting these variables into the integrand, we have
\begin{align}
  u \log z + z^2 - 2zy
  \label{eq:replacement}
  &=\tfrac{u}{2}(\log \tfrac{u}{ 2} + 1) - \sqrt{2u}y \\
  \notag
  &+u\left(
  \log(1+u^{-1/3}\zM) - u^{-1/3}\zM + \frac{1}{2} u^{-2/3}\zM^2 - \frac{\zM \yM}{u}
  \right).
\end{align}

Define 
\begin{align}
  \G[i] &=  
  \log(1+u_i^{-1/3}\zM[i]) - u_i^{-1/3}\zM[i] + \frac{1}{2} u_i^{-2/3}\zM[i]^2 - \frac{\zM[i] \yM[i]}{u_i},
  \text{ and } \label{eq:G} \\
  J(u_i,y_i) &= 2^{u_i} 
  \exp(\tfrac{u_i}{2}(\log \tfrac{u_i}{2} + 1) - \sqrt{2u_i}y_i), \label{eq:J}
\end{align}
for $i=1,2$
so that we may rewrite \eqref{eq:odK} as  
\begin{equation}
  \odK = \frac{1}{2(\pi i)^2}\frac{J(u_1,y_1)}{J(u_2,y_2)} 
  \iint
  \frac{e^{u_1 \G[1]}}{e^{u_2 \G[2]}} \frac{dz_1 dz_2}{z_1 - z_2}.
  \label{eq:hermite3}
\end{equation}

\begin{remark}
  \label{rem:airylimit}
  The Airy kernel limit can be seen from this representation (c.f.\ Lemma~\ref{lem:quant_Airy}, where a different representation is used).  Taylor expanding the log in $G$ around $\zM = 0,$ one sees
\begin{equation}
  \label{eq:taylor}
  u_i \G[i] = -\zM[i]\yM[i] + \frac{1}{3}\zM[i]^3 + O(u_i^{-1/3}\zM[i]^4),
\end{equation}
noting the error is uniform in $\yM[i].$  On the contours $\gamma_1$ and $\gamma_2,$ as well as their conjugates below the axis, one gets that this 
error is order $\zM[i]^4 = o(u_i^{-1}).$  One can argue that
\[
 \iint
 \frac{e^{u_1\G[1]}}{e^{u_2\G[2]}} \frac{dz_1 dz_2}{z_1 - z_2}
  =o(1) 
  +\iint\limits_{\gamma_2,\gamma_1}
  \frac
  {e^{{\zM[2]^3}/{3} - \zM[2] \yM[2] }}
  {e^{{\zM[1]^3}/{3} - \zM[1] \yM[1] }}
  \frac{dz_1 dz_2}{z_1 - z_2}.
\]
The Airy function, meanwhile, has the following representation (see \cite[9.5.4]{DLMF})
\[
  \Ai(y) = \frac{1}{2\pi i} \int_{\infty e^{-i\pi/3}}^{\infty e^{i \pi /3}}
  e^{{z^3}/{3} - zy}\,dz,
\]
from which point it can be deduced that the kernel in question converges to the Airy kernel when $u_1$ and $u_2$ go to infinity with $u_1 - u_2 = o(u_1^{2/3}).$  
\end{remark}

\section{Proof of the right tail correlation estimate for $u_1^{1/3} \ll u_2 - u_1 \ll u_1^{2/3}$ }
\label{sec:corproof}

\subsection{Overview}

Throughout this section we will assume $u_2 \geq u_1$ and write $\Delta u = u_2 - u_1$ and $u=u_1.$  Also, introduce the measures
$\mu_i(d\yM[i]) = d\yM[i]/(\sqrt{2}u_i^{1/6})$ for $i=1,2.$  Our main goal is to prove Proposition~\ref{prop:rtcorrelation}.

As in \eqref{eq:k_diff}, the joint probability can be expressed by $\det( \Id - \tilde{K}\vert_I).$  It is convenient to express the kernel $\Id-\tilde{K}\vert_I$ as a $2 \times 2$ matrix of kernels.  This acts on vectors of elements of $L^2(dy_1) \oplus L^2(dy_2)$ by first performing matrix multiplication and then by the usual integration.  Define ${\tilde K}$ and ${\tilde \phi}$ as
\begin{align*}
  {\tilde K}(u_1,\yM[1];u_2,\yM[2]) &= 
  \frac{J(u_2,y_2)}{J(u_1,y_1)}\left(\phi^{(u_1,u_2)}(y_1,y_2) + K(u_1,y_1;u_2,y_2)\right), \\
  {\tilde \phi}(u_1,\yM[1];u_2,\yM[2]) &= 
  \frac{J(u_2,y_2)}{J(u_1,y_1)} \phi^{(u_1,u_2)}(y_1,y_2).
\end{align*}
Implicitly, we shift and scale the action of these kernels on the $L^2$ integrating them against functions in the $\yM[i]$ coordinates.  Hence, the measures on the underlying $L^2$ spaces are now $L^2(d\mu_1) \oplus L^2(d\mu_2).$  Let $\Proj{i}$ denote the multiplication operator by the characteristic function $\one[ \yM[i] \geq t_i]$ for $i =1,2.$  
Then,
\begin{equation}
  \Id - \tilde{K}\vert_I
  \longleftrightarrow
  \begin{bmatrix}
    \Id - \Proj{1}{\tilde K}(u_1,\cdot;u_1,\cdot)\Proj{1}& 
    \Proj{1}\bigl({\tilde \phi}(u_1,\cdot;u_2,\cdot) - {\tilde K}(u_1,\cdot;u_2,\cdot)\bigr)\Proj{2} \\
  -\Proj{2}{\tilde K}(u_2,\cdot;u_1,\cdot)\Proj{1}
  &
   \Id - \Proj{2}{\tilde K}(u_2,\cdot;u_2,\cdot)\Proj{2} \\
  \end{bmatrix}.
  \label{eq:blocks}
\end{equation}

As we are only interested in the determinant of this operator, we can subtract an operator multiple of the second row from the first.  Working in the case that $\Proj{1}\Proj{2} = \Proj{1},$ we will subtract the left-multiplication of the second row by $\Proj{1}{\tilde \phi}(u_1,\cdot;u_2,\cdot)$ from the first.  As all the ${\tilde K}$ terms are nearly the Airy kernel (explicit estimates are given in Section~\ref{sec:sharp}), the differences between the various ${\tilde K}$ will be smaller in norm than the kernels themselves.  Further, ${\tilde \phi}$ behaves like an approximation to the identity for a certain nice class of functions.  Hence, after doing this row operation, the matrix of kernels is approximately lower triangular, and its determinant is hence very nearly the determinant of its lower-right block.  This allows us to estimate
\[
  F(u_1,t_1 ; u_2, t_2)
  =
  \left|
  \det\left(\Id - \tilde{K}\vert_I\right)
  -
  \det\left( \Id - \Proj{2}{\tilde K}(u_2,\cdot;u_2,\cdot)\Proj{2}\right)
  \right|.
\]

Let $\phiop$ denote the operator $L^2(d\mu_2) \to L^2(d\mu_1)$ given by
\begin{equation}
  \label{eq:phiop}
  \phiop[f](\yM[1]) 
  = \int_\R
  {\tilde \phi}(u_1,\yM[1];u_2,\yM[2])f(\yM[2])\,\mu_2(d\yM[2]).
\end{equation}
The exact sense in which $\phiop \approx \Id$ is given by Lemma~\ref{lem:phi_identity}.  To prove this, we will pass to Fourier space, and so we state our Fourier transform conventions.
 Let $\mathcal{F}$ denote the Fourier transform with the normalization
  \[
    \mathcal{F}[\phi](\xi) = \frac{1}{2\pi}\int_\R e^{-i \xi x} \phi(x)\,dx.
  \]
  With this normalization $\mathcal{F}$ has $L^2(dx)\to L^2(dx)$ operator norm $(2\pi)^{-1/2},$ and its inverse carries no factors of $\pi.$
  Define the $\|\cdot\|_{H^2}$ norm by
  \[
	  \|f\|_{H^2} = \|(I-\Laplacian)f\|_{L^2(dx)},
  \]
  where $\Laplacian$ is the $1$-dimension Laplacian $\Laplacian f(x) = \partial_x^2 f(x).$  Let $H^2$ denote the corresponding subspace of $L^2$ given by taking the closure of the $C_c^\infty$ functions under $H^2.$  By considering the Fourier transform, we have that the $\|\cdot\|_{H^2}$ norm is equivalent to the norm
  \[
	  f \mapsto
  \| f(x) \|_{L^2(dx)}
  +\| \partial_x f(x) \|_{L^2(dx)}
  +\| \partial_x^2 f(x) \|_{L^2(dx)}.
  \]

  Recall that the inverse Laplacian $(I-\Laplacian)^{-1}$ operator on $L^2(dx)$ can be defined as the Fourier multiplier operator
  \[
	  \mathcal{F}[(I-\Laplacian)^{-1}f](\xi) = 
	  \frac{c}{1+\xi^2}
	  \mathcal{F}[f](\xi)
  \]
  for some constant $c.$  Alternatively, we can write it in convolution form as
  \[
	  (I-\Laplacian)^{-1}f = c e^{-|\cdot|} * f
  \]
  for some other constant $c.$
  \subsection{$\phi$ is an approximate identity}
  \begin{lemma}[Approximate identity estimates for $\phiop$]
  \label{lem:phi_identity}
  For any $0<\beta < \delta < \tfrac16$ there is a $C>0$ sufficiently large so that for all $u_1,u_2 \in \mathbb{N}$ with $u^{1/3+\delta} \leq \Delta u \leq u^{2/3-\delta},$ the following hold.
  \begin{enumerate}[(i)]
	  \item For all $\yM[i]$ ,
      \[
	      |{\tilde \phi}( u_1, \yM[1] ; u_2, \yM[2])|
\leq 
C{\sqrt{u}}\exp\left( 
-\frac{\Delta u}{Cu^{2/3}}\left( \yM[1] + \yM[2] \right)
\right).
\]

    \item For any $|\yM[i]| < u^{\beta},$ $i=1,2,$  
      \[
	|{\tilde \phi}( u_1, \yM[1] ; u_2, \yM[2])|
\leq 
C 
\sqrt{\frac{u}{\Delta u}}
\exp\left(
-\frac{u^{2/3}}{2\Delta u}\left(\yM[1] - \yM[2]\right)^2
\right)
+C\exp\left( -\frac{(\Delta u)^{1/3}}{C} \right).
      \]
    \item For any $f \in C^1(\R)$ supported on $[-u^{\beta},\infty)$ with absolutely continuous derivative $f',$
      \begin{multline*}
      \left\|
      \one[ |\cdot| \leq u^{\beta} ]
      (\phiop[f](\cdot) - f(\cdot))
      \right\|_{L^2(\mu_1)}
      \\
      \leq C\tfrac{\Delta u}{u^{2/3-\beta}} 
      \left( 
      \|f\|_{L^2(\mu_2)}
      +\|f'\|_{L^2(\mu_2)}
      +\|f''\|_{L^2(\mu_2)}
       \right).
     \end{multline*}
     \item For any $g \in L^2(dx)$ supported on $[0,\infty),$ 
		     \[
\left\|
      \one[ |\cdot| \leq u^{\beta} ]
      (\phiop - \Id)(\Id - \Laplacian)^{-1}[g](\cdot)
      \right\|_{L^2(\mu_1)}
      \leq C\tfrac{\Delta u}{u^{2/3-\beta}}
      \|g\|_{L^2(\mu_2)}.
		     \]
  \end{enumerate}
\end{lemma}
\begin{proof}
  From \eqref{eq:contours}
  we have that $\gamma_1$ and $\gamma_2$ intersect for all $u_1$ sufficiently large.  Hence $\tau$ is given by 
 \begin{equation}
   \tau = \frac{\sqrt{u_1} + \sqrt{u_2}}{2\sqrt{2}} + i \frac{\sqrt{3}}{2\sqrt{2}}({\sqrt{u_2} - \sqrt{u_1}}).
   \label{eq:tau}
 \end{equation}
 When $y_1 \leq y_2,$ a deformation of the contour in \eqref{eq:phi_res} gives 
  \begin{equation}
    \label{eq:Kec1}
    {\tilde \phi}( u_1, \yM[1] ; u_2, \yM[2])
    =
    \frac{J(u_2,y_2)}{J(u_1,y_1)}
    \frac{1}{\pi i} \int\limits_{\Re z_2 = \Re \tau} \frac{e^{2z_2(y_2 - y_1)}}{(2z_2)^{u_2 - u_1}}\, d z_2.
  \end{equation}
  Note that in conclusions (ii), (iii) and (iv) of the lemma, we consider $|\yM[1]| < u^{\beta}$ and $\yM[2] > -u^{\beta}.$  For the $y_i$ and the $u_i$ that we consider, we have, using that $\Delta u \geq u^{1/3+\delta}$ in the second inequality,
  \begin{align*}
    y_1
    &\leq \sqrt{2u_1} + \frac{u_1^{\beta}}{\sqrt{2}u_1^{1/6}} \\
    &\leq \sqrt{2u_2} + \frac{u_1^{\beta}}{\sqrt{2}u_1^{1/6}} - \Omega\left( \frac{u_1^{\delta}}{u_1^{1/6}} \right) \\
    &\leq \sqrt{2u_2} + \frac{\yM[2]}{\sqrt{2}u_1^{1/6}} = y_2.
  \end{align*}
  Hence expression \eqref{eq:Kec1} always holds for these $\yM[i].$  For conclusion (i), the bound is trivially satisfied in the case $y_1 > y_2,$ and hence it suffices to consider the case $y_1 \leq y_2.$

  Using the definition \eqref{eq:J} of $J(u_i,y_i),$ we have
  \begin{equation}
    \label{eq:Kec2}
    {\tilde \phi}
    =
    \exp( \xi(u_1,\yM[1];u_2,\yM[2]))
\frac{1}{\pi i} 
\int\limits_{\Re z_2 = \Re \tau} e^{2i\Im z_2(y_2 - y_1)}\left(\frac{\Re \tau}{z_2}\right)^{u_2 - u_1}\, d z_2.
  \end{equation}
  where
  \[
    \exp( \xi(u_1,\yM[1];u_2,\yM[2]))
    =
\left(\frac{u_2}{u_1}\right)^{\tfrac{u_1}{2}}
\left(\frac{2|\Re \tau|^2}{eu_2}\right)^{\tfrac{u_1 - u_2}{2}}
e^{2 \Re \tau (y_2 - y_1) - \sqrt{2u_2}y_2 + \sqrt{2u_1}y_1}.
  \]
  Expanding these definitions, we have that
\begin{equation}
  \xi(u_1,\yM[1];u_2,\yM[2])
=
-\frac{\sqrt{u_2} - \sqrt{u_1}}{2}
\left(
\frac{\yM[1]}{u_1^{1/6}}
+\frac{\yM[2]}{u_2^{1/6}}
\right)
+O\left( \frac{(\Delta u)^3}{u^2} \right).
\label{eq:Kexi}
  \end{equation}
  Conclusion (i) of the Lemma now follows from bounding the integral in \eqref{eq:Kec2} by absolute value and \eqref{eq:Kexi}.

  We will eventually truncate the integral over $\gamma_{+}^r$ into $|\Im z_2| \leq R$ and $|\Im z_2| > R,$ where $R=R(u)$ to be chosen later satisfies $R^3 \Delta u / u^{3/2} = O(1).$  Note that this implies that $R = o(u^{1/2}).$ 

  Define $\zeta(w,u)$ implicitly by
  \[
    \Delta u \log\left( 1 + \frac{iw}{\Re \tau}\right)
    = 
    i\Delta u \frac{w}{\Re \tau}
    +\frac{\Delta u}{2}\left(\frac{w}{\Re \tau}\right)^2
    +
    \zeta(w, u).
  \]
  Then for each $u,$ $\zeta(w,u)$ is analytic in $w$ for all $w$ with $\Im w < \Re \tau$ and $|\zeta(w,u)| = O\left(\Delta u\cdot |w|^3/u^{3/2} \right)$ uniformly in $|w| \leq 2R.$  Note that $| \Re \zeta(w,u) | = O(\Delta u w^4/u^2)$ for real $w.$ 
  We use $\zeta$ to express \eqref{eq:Kec2} as
  \begin{equation}
    \label{eq:Kec3}
    {\tilde \phi}
    =
\frac{e^{\xi}}{\pi} 
\int\limits_{\R}
\exp\left(-\zeta + iw\left(2(y_2 - y_1)-\frac{\Delta u}{\Re \tau}\right)
-\frac{\Delta u}{2}\left(\frac{w}{\Re \tau}\right)^2\right)
\, dw.
  \end{equation}
 Define
    \begin{equation}
    \label{eq:KeH}
    \mathcal{H}
    = 2(y_2 - y_1) - \frac{\Delta u}{\Re \tau}.
  \end{equation}
  Using the analyticity of the integrand and the polynomial decay of the integrand as $|\Re w| \to \infty,$  we may make the replacement $w \mapsto w + i\mathcal{H} \frac{\Re \tau}{\Delta u}$ in \eqref{eq:Kec3}, provided $\mathcal{H} < \Delta u,$ to get 
\begin{equation}
    \label{eq:Kez_loc}
    {\tilde \phi}
    =
\frac{e^{\xi}}{\pi} 
\int\limits_{\R}
\exp\left(-\zeta(w+i\mathcal{H},u)
-\frac{\Delta u}{2}\left(\frac{w}{\Re \tau}\right)^2
-\frac{(\mathcal{H}\Re \tau )^2}{2\Delta u}
\right)
\, dw.
  \end{equation}

  We truncate this integral into $|w| > R$ and $|w| < R.$  Let 
\begin{equation}
    \label{eq:Kez_loc_R}
    {\tilde \phi}^{R}
    =
\frac{e^{\xi}}{\pi} 
\int\limits_{-R}^R
\exp\left(-\zeta(w+i\mathcal{H},u)
-\frac{\Delta u}{2}\left(\frac{w}{\Re \tau}\right)^2
-\frac{(\mathcal{H}\Re \tau )^2}{2\Delta u}
\right)
\, dw.
  \end{equation}
  For $|\mathcal{H}| < R,$ we have that
  $|\zeta(w+i\mathcal{H},u)| = O(\Delta u R^3/u^{3/2})=O(1),$ giving
\begin{align}
    \label{eq:Kez_loc_R2}
    {\tilde \phi}^{R}
    &\leq
    \frac{\exp\left(\xi + O(1)
-\frac{(\Re \tau \mathcal{H})^2}{2\Delta u}\right)
}{\pi} 
\int\limits_{\R}
\exp\left(
-\frac{\Delta u}{2}\left(\frac{w}{\Re \tau}\right)^2
\right)
\, dw \\
&\leq 
\frac{\Re \tau\sqrt{2}}{\sqrt{\pi \Delta u}}
\exp\left(\xi + O\left(1\right)
-\frac{(\mathcal{H}\Re \tau )^2}{2\Delta u}\right).
\nonumber
  \end{align}
  Recalling \eqref{eq:KeH}, we have
  \[
    \mathcal{H}
    = 2(y_2 - y_1) - \frac{\Delta u}{\Re \tau}
    = \sqrt{2}\left( 
    \frac{\yM[2]}{u_2^{1/6}}
    -\frac{\yM[1]}{u_1^{1/6}}
    \right).
  \]
  Hence with $|\yM[i]| < u^{\beta},$ we have that
  \[
    \mathcal{H} = \frac{\sqrt{2}}{u^{1/6}}( \yM[2] - \yM[1]) + O\left( \frac{u^{\beta}\Delta u}{u^{7/6}} \right).
  \]
  Applying this to \eqref{eq:Kez_loc_R2} and that $\xi=O(1)$ for these $\yM[i]$, we have that 
\begin{align}
    \label{eq:Kez_loc_R3}
    {\tilde \phi}^{R}
\leq 
\frac{\Re \tau}{\sqrt{\Delta u}}
\exp\left(O\left(1\right)
-
\frac{u^{2/3}}{\Delta u}
\frac{(\yM[2]-\yM[1])^2}{2}\right).
  \end{align}

  As for the portion of the integral with $|w| > R,$ note that by the definition of $\zeta,$ we have 
  \[
    \Re\left[
\zeta(w+i\mathcal{H},u)
+\frac{\Delta u}{2}\left(\frac{w}{\Re \tau}\right)^2
+\frac{(\mathcal{H}\Re \tau )^2}{2\Delta u}
\right]
= \Delta u \log \left| 1+ \frac{iw - \mathcal{H}}{\Re \tau}\right|.
  \]
Hence, we get the pointwise bound
\begin{equation}
\left|
{\tilde \phi}
-
{\tilde \phi}^{R}
\right|
\leq
\frac{2e^{\xi}}{\pi}
\int\limits_R^\infty
\left| 1+ \frac{iw - \mathcal{H}}{\Re \tau}\right|^{-\Delta u}
dw.
  \label{eq:Kez_tail}
\end{equation}
For $|\yM[i]| < u^{\beta}$ we have $|\mathcal{H}| = o(1).$
Hence for $w > 2\Re \tau,$ the contribution of the integral is $O( (2-o(1))^{-\Delta u})$ as long as $R < 2\Re \tau.$  Fix now $R=u^{1/2}/(\Delta u)^{1/3},$ which satisfies this condition as well the earlier condition that $R^3\Delta u /u^{3/2} = O(1).$  For $w <2 \Re \tau,$ we have that
\[
\left| 1+ \frac{iw - \mathcal{H}}{\Re \tau}\right|
=\exp\left( -\Omega\left( \frac{\mathcal{H}}{\Re \tau} \right) + \frac{w^2}{2(\Re \tau)^2} +O(1)\right)
\]
Hence we get that
\[
  \int\limits_R^{2\Re \tau}
\left| 1+ \frac{iw - \mathcal{H}}{\Re \tau}\right|^{-\Delta u}
dw=\exp\left( -\Omega\left( \frac{R^2}{u} \Delta u \right) \right).
\]
Applying this to \eqref{eq:Kez_tail}, we conclude that for $|\yM[i]| < u^{\beta}$
\begin{equation}
\left|
{\tilde \phi}
-
{\tilde \phi}^{R}
\right|
\leq
\exp\left(-\Omega\left( \frac{R^2}{u} \Delta u \right) \right)
=\exp\left(-\Omega\left( (\Delta u)^{1/3} \right) \right).
  \label{eq:Kez_tail2}
\end{equation}
Together with \eqref{eq:Kez_loc_R2}, this gives conclusion (ii) of the Lemma.

We now turn to the proof of conclusion (iii).  Here we will need to pick a different function $R(u).$  Let $\zeta^R = \zeta(w,u)\one[|w|<R],$ and define $\tilde{\phi}^{R,\zeta}$ by
  \begin{equation}
    \label{eq:phi_RZ}
    {\tilde \phi}^{R,\zeta}(\mathcal{H})
    =
\frac{1}{\pi} 
\int\limits_{\R}
\exp\left(-\zeta^R + iw\mathcal{H}
-\frac{\Delta u}{2}\left(\frac{w}{\Re \tau}\right)^2\right)
\, dw,
  \end{equation}
  noting we have omitted the $e^{\xi}$ from the expression.    This has the form of a Fourier transform of a product of functions evaluated at $\mathcal{H}$.  In particular, let us define the distributions
  \begin{align}
    U(x) &= 2\pi\delta_{0}(x) + \mathcal{F}^{-1}\left[ \bigl( e^{\zeta(\cdot,u)} - 1 \bigr)\one[|\cdot| \leq R] \right](x) \\
    V(x) &= \sqrt{ \frac{(\Re \tau)^2}{2\pi \Delta u}}e^{ -\frac{(\Re \tau)^2}{2\Delta u} x^2}.
    \label{eq:Fourier}
  \end{align}

  This allows us to rewrite \eqref{eq:phi_RZ} as
  \begin{equation}
    {\tilde \phi}^{R,\zeta}
    =
    2\mathcal{F}^{-1}[\mathcal{F}[U]\mathcal{F}[V]](\mathcal{H})
    =
    2(U * V)(\mathcal{H}).
    \label{eq:phi_RZ2}
  \end{equation}

  Let $\phiopRZ$ be the operator from $L^2(\mu_2) \to L^2(\mu_1)$ with kernel ${\tilde \phi}^{R,\zeta}$ defined in the same way as in \eqref{eq:phiop}. 

  Let $q$ be a Schwartz function and consider the action of $\phiopRZ$ on it
  We have that
  \begin{align*}
    \phiopRZ[q](\yM[1])
    &= 
    2
    \int_\R(U * V)\biggl(\frac{\sqrt{2}\yM[2]}{u_2^{1/6}} - \frac{\sqrt{2}\yM[1]}{u_1^{1/6}}\biggr)q(\yM[2])\,\frac{\sqrt{2}d\yM[2]}{u_2^{1/6}} \\
    &= 
    2(U * V * \tilde{q})\left(  \frac{\sqrt{2}\yM[1]}{u_1^{1/6}} \right),
  \end{align*}
  where $\tilde{q}(x) = q\left(  \frac{u_2^{1/6}}{\sqrt{2}} x \right).$
  
  By considering the Fourier transform, we have that 
  \begin{align*}
    \|(V * \tilde{q})(\yM[1]) - \tilde{q}(\yM[1])\|_{L^2(d\yM[1])} 
    &=
    \frac{1}{\sqrt{2\pi}}
    \|(2\pi\mathcal{F}[V](x)-1)\mathcal{F}[\tilde{q}](x)\|_{L^2(dx)} \\
    &\leq
    \sqrt{2\pi}
    \frac{\Delta u}{2(\Re \tau)^2}\|x^2\mathcal{F}[\tilde{q}](x)\|_{L^2(dx)} \\
    &=
    O\bigl(\frac{\Delta u}{u}\bigr)\| \tilde{q}''(\yM[1])\|_{L^2(d\yM[1])}.
  \end{align*}
  Hence by adjusting constants, we have that 
  \begin{align}
    \| V * {\tilde q}(\tfrac{\sqrt{2} \yM[1]}{u_1^{1/6}}) - q(\tfrac{u_2^{1/6} \yM[1]}{u_1^{1/6}})\|_{L^2(mu_1(d\yM[1]))}
    &=
    \| V * {\tilde q}(\tfrac{\sqrt{2} \yM[1]}{u_1^{1/6}}) - {\tilde q}(\tfrac{\sqrt{2}\yM[1]}{u_1^{1/6}})\|_{L^2(\mu_1(d\yM[1]))} \nonumber \\
    &= \| V * {\tilde q}(x) - {\tilde q}(x)\|_{L^2(dx)} \nonumber \\
    &\leq O\bigl(\frac{\Delta u}{u}\bigr)\| \partial_x^2 \tilde{q}(x)\|_{L^2(dx)}\nonumber \\
    &= O\bigl(\frac{\Delta u}{u^{2/3}}\bigr)\| {q}''(\tfrac{u_2^{1/6}x}{\sqrt{2}})\|_{L^2(dx)}\nonumber \\
    &= O\bigl(\frac{\Delta u}{u^{2/3}}\bigr)\| {q}''(\yM[2])\|_{L^2(\mu_2(d\yM[2]))}
    \label{eq:H2boundq}
  \end{align}
  Meanwhile, setting 
  $\alpha = \tfrac{u_2^{1/6}}{u_1^{1/6}}$ so that $\alpha - 1 = O(\Delta u/u),$ we have that
  for $|\yM[i]| \leq u^\beta,$ 
  \begin{align}
    \|\one[|\yM[1]| \leq u^{\beta}]|q(\yM[1]) -  q(\tfrac{u_2^{1/6} \yM[1]}{u_1^{1/6}})|\|^2_{L^2(\mu_1(d\yM[1])}
    &\leq
    \int\limits_{-u^{\beta}}^{u^{\beta}}
    \biggl(
    \int\limits_{\yM[1]}^{\alpha \yM[1]}
    q'(x)\,dx
    \biggr)^2
\mu_1(d\yM[1])
    \nonumber \\
    &\leq
    \int\limits_{-u^{\beta}}^{u^{\beta}}
    (\alpha - 1)\yM[1]
    \int\limits_{\yM[1]}^{\alpha \yM[1]}
    \bigl(
    q'(x)
    \bigr)^2
\,dx\,\mu_1(d\yM[1]).
    \nonumber \\
    \intertext{Changing the order of integration and estimating,}
\|\one[|\yM[1]| \leq u^{\beta}]|q(\yM[1]) -  q(\tfrac{u_2^{1/6} \yM[1]}{u_1^{1/6}})|\|^2_{L^2(\mu_1(d\yM[1])}
    &\leq
    \int\limits_{-\alpha u^{\beta}}^{\alpha u^{\beta}}
    \bigl(
    q'(x)
    \bigr)^2
    \int\limits_{x}^{\alpha x}
    (\alpha - 1)\yM[1]\,\mu_1(d\yM[1])
    \,dx
    \nonumber \\
    &\leq
    \int\limits_{-\alpha u^{\beta}}^{\alpha u^{\beta}}
    O(\alpha - 1)^2
    x^2
    \bigl(
    q'(x)
    \bigr)^2\,\mu_1(dx)
    \nonumber \\
    &\leq O\bigl(\frac{\Delta u}{u^{1-\beta}}\bigr)^2\|q'(\yM[2])\|^2_{L^2(d\mu_2)}.
    \label{eq:annoyed}
  \end{align}
  Finally, as we have that $(\Delta u)R^4/u^2 = o(1)$ it follows that
  \begin{align}
    \nonumber
    \|((U - 2\pi\delta_0) * q)(x)\|_{L^2(dx)}
    &=
    \|\bigl(e^{\zeta(x,u)}-1\bigr)\one[|x| \leq R]\mathcal{F}[q](x)\|_{L^2(dx)} \\
    \nonumber
    &\leq O\bigl(\frac{\Delta u R^4}{u^{2}}\bigr) \|\mathcal{F}[q](x)\|_{L^2(dx)} \\
    \nonumber
    &= O\bigl(\frac{\Delta u R^4}{u^{2}}\bigr) \|q(x)\|_{L^2(dx)}.
\intertext{
	We will take $R = u^{1/2+\delta/4}/\sqrt{\Delta u},$ so that }
    \|((U - 2\pi\delta_0) * q)(x)\|_{L^2(dx)}
    &= O\bigl(\frac{u^{\delta}}{\Delta u}\bigr) \|q(x)\|_{L^2(dx)}
    = O\bigl(\frac{\Delta u}{u^{2/3}}\bigr) \|q(x)\|_{L^2(dx)}.
    \label{eq:goawayU}
  \end{align}

  Combining \eqref{eq:H2boundq}, \eqref{eq:annoyed}, and \eqref{eq:goawayU}, we have that for all $q$ in $H^2,$
  \begin{multline}
    \| \one[|\cdot| \leq u^{\beta}]|\phiopRZ[q](\cdot) - {q}(\cdot)|\|_{L^2(d\mu_1)} \\
    \leq 
O\bigl(\frac{\Delta u}{u^{2/3}}\bigr)
\left( 
\| q\|_{L^2(d\mu_2)}
+\| q'\|_{L^2(d\mu_2)}
+\| q''\|_{L^2(d\mu_2)}
\right).
    \label{eq:phi_RZ3}
  \end{multline}

  We now proceed to compare the action of ${\tilde \phi}^{R,\zeta}$ with that of ${\tilde \phi}.$  Following a similar progression as taken in deriving \eqref{eq:Kez_tail2} from \eqref{eq:Kez_tail}, we have \emph{uniformly} in $\yM[i]$ that
  \begin{equation*}
   \left|
   e^{-\xi}
   {\tilde \phi}
-
{\tilde \phi}^{R,\zeta}
\right|
\leq
\frac{1}{\pi}
\int\limits_R^\infty
\left| 1+ \frac{iw}{\Re \tau}\right|^{-\Delta u}
dw
\leq
\exp\left(-\Omega\left( \frac{R^2}{u} \Delta u \right) \right)
=
\exp(-\Omega(u^{\delta/2})).
  \end{equation*}
  In particular, we can use this pointwise bound to give an estimate on the Hilbert-Schmidt norm of the difference of these kernels restricted to $\yM[i] > -u^{\beta}$ by
  \begin{align}
    \iint_{\yM[i] > -u^{\beta}}
    &
    \left|
    {\tilde \phi}(u_1,\yM[1];u_2,\yM[2])
    -
    e^{\xi(u_1,\yM[1];u_2,\yM[2])}
    {\tilde \phi}^{R,\zeta}(u_1,\yM[1];u_2,\yM[2])
    \right|^2\,
    \mu_1(d\yM[1])
    \mu_2(d\yM[2])  \nonumber \\
    &\leq
    e^{-\Omega(u^{\delta/2})}
    \|\one[\yM[i] > -u^{\beta},i=1,2] 
    e^{\xi(u_1,\yM[1];u_2,\yM[2])}
    \|^2_{L^2(\mu_1(\yM[1])\times \mu_2(\yM[2]))} \nonumber \\
    &\leq
    e^{-\Omega(u^{\delta/2})}
    e^{O(\log u)}.
\label{eq:phi_RZ_tail0}
  \end{align}
 Define
 \[
   \xi_2(\yM[2]) = -\frac{\sqrt{u_2} - \sqrt{u_1}}{2}
\frac{\yM[2]}{u_2^{1/6}},
 \]
 and let $\xi_1(\yM[1]) = \xi(u_1,\yM[1];u_2,\yM[2])-\xi_2(\yM[2]),$ noting that the right hand side is independent of $\yM[2].$

 By \eqref{eq:phi_RZ_tail0} we have that for any $q \in L^2(\mu_1)$ supported on $[-u^{\beta},\infty)$
\begin{equation}
   \left\|
  \phiop[q](\cdot)-
  e^{\xi_1(\cdot)}\phiopRZ[e^{\xi_2(\cdot)}q](\cdot)
      \right\|_{L^2(d\mu_1)}
      \leq 
    e^{-\Omega(u^{\delta/2})}
    \|q\|_{L^2(d\mu_2)}.
\label{eq:phi_RZ_tail}
\end{equation}
For $|\yM[1]| < u^\beta,$ $|\xi_1(\yM[1])| = O(1).$ 
Combining this observation with \eqref{eq:phi_RZ3}, we get
\begin{align}
\bigl\|
e^{\xi_1(\yM[1])}&\one[|\yM[1]| \leq u^\beta]
\left(\phiopRZ[e^{\xi_2(\cdot)}q](\yM[1])
  -e^{\xi_2(\yM[1])}q(\yM[1])\right)
  \bigr\|_{L^2(\mu_1(d\yM[1]))} \nonumber\\
  &\leq
  e^{O(1)}
\left\|
\one[|\yM[1]| \leq u^\beta]
\left(\phiopRZ[e^{\xi_2}q](\yM[1])
  -e^{\xi_2(\yM[1])}q(\yM[1])\right)
  \right\|_{L^2(\mu_1(d\yM[1]))} \nonumber \\
&\leq
O\left( \frac{\Delta u}{u^{5/6}} \right)
\left\|
e^{\xi_2}q
\right\|_{H^2}.
\label{eq:phi_xi_RZ}
\end{align}
Observe that for $q$ supported on $[-u^\beta, \infty),$ we have that $\|e^{\xi_2(x)}q(x)\|_{H^2} = O(1)\|q(x)\|_{H^2}.$

  It remains to compare $q$ with $e^{\xi_2}q.$
Using that $\xi_2(\yM[1])=o(1)$ for $|\yM[1]| < u^{\beta},$ there is a constant $C>0$ so that for all these $\yM[i],$ $|e^{\xi_2(\yM[1])}-1| \leq C\frac{\Delta u}{u^{2/3}}|\yM[1]|.$ Hence
  \begin{align}
 \bigl\|
\one[|\yM[1]| \leq u^\beta]
&
\left(
e^{\xi_2(\yM[1])}
-1
\right)
q(\yM[1])
  \bigr\|_{L^2(\mu_1(d\yM[1]))} 
  \leq
  O\left( \frac{\Delta u}{u^{2/3-\beta}} \right)
   \bigl\|q\bigr\|_{L^2(\mu_1)}. 
    \label{eq:exiq_q}
  \end{align}
Combining \eqref{eq:exiq_q}, \eqref{eq:phi_xi_RZ}, and \eqref{eq:phi_RZ_tail}, conclusion (iii) follows.

For conclusion (iv), note that this does not directly follow from (iii), as $f = (\Id - \Laplacian)^{-1}g$ generally has full support.  However, for $g$ that is supported on $[0,\infty),$ $f$ will be exponentially small on $(-\infty, -u^{\beta}],$ and so the conclusion will follow from this and (i).

To this end, let $\rho \in C^{\infty}$ be an increasing function that is $0$ on $(-\infty,-u^{\beta}]$ and $1$ on $[-u^{\beta} +1,\infty).$  Then by (iii), we have
	\begin{equation}
	\left\|
      \one[ |\cdot| \leq u^{\beta} ]
      (\phiop - \Id)[\rho f](\cdot)
      \right\|_{L^2(\mu_1)}
      \leq C\tfrac{\Delta u}{u^{2/3-\beta}}
      \|g\|_{L^2(\mu_2)}.	
		\label{eq:rhof}
	\end{equation}

	On the other hand,
	\begin{align}
		\sup_{\yM[2] \in \R}
		\left| (1-\rho)(\yM[2])f(\yM[2])\right|
		&\leq
		\sup_{\yM[2] \leq -u^{\beta}+1}
		\left|f(\yM[2])\right| 
		\nonumber \\
		&\leq
		\sup_{\yM[2] \leq -u^{\beta}+1}
		c\int_\R e^{-|\yM[2]-\wM|}
		|g(\wM)|\,d\wM
		\nonumber \\
		&\leq
		O(e^{-u^{\beta}/2})
		\sup_{\yM[2] \leq -u^{\beta}+1}
		c\int_\R e^{-|\yM[2]-\wM|/2}
		|g(\wM)|\,d\wM
		\nonumber \\
		&=
		O(e^{-u^{\beta}/2})
		\|g\|_{L^2},
		\label{eq:1rhof}
	\end{align}
	where in the last step we have applied H\"older's inequality.  Hence by conclusion (i) and \eqref{eq:1rhof}, we conclude
	\begin{equation}
	\left\|
      \one[ |\cdot| \leq u^{\beta} ]
      (\phiop - \Id)[(1-\rho) f](\cdot)
      \right\|_{L^2(\mu_1)}
      = e^{O(\log u)-u^{\beta}/2}
      \|g\|_{L^2(\mu_2)}.	
		\label{eq:1rhof2}
	\end{equation}
	Combining \eqref{eq:rhof} and \eqref{eq:1rhof2}, the conclusion follows.

\end{proof}

\subsection{Proof of correlation proposition}
\begin{proof}[Proof of Proposition~\ref{prop:rtcorrelation}]
	The domain $I$ is given in $\yM[i]$ coordinates by
\[
  I = 
\left\{ u_1 \right\} \times [t_1, \infty) \cup
\left\{ u_2 \right\} \times [t_2, \infty).
\]
Define a new domain $I'$ by
\[
  I' = 
  \left\{ u_1 \right\} \times [t_1,u^{\beta}/2] \cup
\left\{ u_2 \right\} \times [t_2, \infty).
\]
Then we have that
\begin{align*}
	|\det(\Id - {\tilde K}_{I})
	-\det(\Id - {\tilde K}_{I'})|
	&\leq \Pr\left[ \lev[u_1] \geq u^{\beta}/2\right] \\
	&=O(e^{-\Omega(u^{3\beta/2})}).
\end{align*}
As the bound we produce on $F(u_1,t_1 ; u_2, t_2)$ for the range of $t_i$ we consider decays no faster than some power of $u,$ we may instead consider bounding
\[
	F(u_1,t_1 ; u_2, t_2)' = |\det(\Id - {\tilde K}_{I'})-\Pr\left[ \lev[u_2] \geq t_2\right]|.
\]

Recall \eqref{eq:blocks}.  Subtract a left multiple $\pi_1'\phiop$ of the second row from the first, and then apply the Schur complement formula.  This gives the identity
\begin{equation}
	\det(\Id - {\tilde K}_{I'})
	=
	\det(\Id - \pi_2{\tilde K}(u_2,\cdot;u_2,\cdot)\pi_2)
	\det(\Id - \Dop_1 + \Dop_2 \Rop \Mop),
	\label{eq:Schur}
\end{equation}
where the operators $\Dop_1 : L^2(\mu_1) \to L^2(\mu_1),$ $\Dop_2:L^2(\mu_2) \to L^2(\mu_1),$ $\Mop: L^2(\mu_1) \to L^2(\mu_2)$ and $\Rop : L^2(\mu_2) \to L^2(\mu_2)$ are given by
\begin{align*}
\Dop_1 &=
\pi_1'
\left(
\phiop
\pi_2
{\tilde K}(u_2,\cdot;u_1,\cdot)
-
{\tilde K}(u_1,\cdot;u_1,\cdot)
\right)\pi_1' \\
\Dop_2 &=
\pi_1'
\left(
\phiop
\pi_2
{\tilde K}(u_2,\cdot;u_2,\cdot)
-
{\tilde K}(u_1,\cdot;u_2,\cdot)
\right)\pi_2 \\
 \Mop &=
 \pi_2{\tilde K}(u_2,\cdot;u_1,\cdot) \pi_1' \\
\Rop &= (\Id - \pi_2{\tilde K}(u_2,\cdot;u_2,\cdot) \pi_2)^{-1}.
\end{align*}
As a consequence, we may bound
\[
	F(u_1,t_1 ; u_2, t_2)'
	\leq
	\Pr\left[ \lev[u_2] \geq t_2\right]|\det(\Id - \Dop_1 + \Dop_2\Rop \Mop)-1|,
\]
and so we turn to estimating the difference of this determinant with $1$.

Let $\nuclear$ denote the nuclear norm.
For any nuclear operators $A$ and $B,$
\[
	|\det(\Id+A)-\det(\Id+B)|
	\leq \nuclear[A-B]e^{1+\nuclear[A] + \nuclear[B]}
\]
(see \cite[(3.7)]{Simon}).
Hence we have the bound
\begin{equation}
|\det(\Id - \Dop_1 + \Dop_2\Rop \Mop) -1|
=O( \nuclear[\Dop_1] +
\nuclear[\Dop_2]
\opnorm[\Rop]
\nuclear[\Mop]),
\label{eq:Schurpieces}
\end{equation}
provided the $\nuclear[-\Dop_1 + \Dop_2 \Rop \Mop]$ is uniformly bounded.  Here we have used the H\"older inequality for Schatten norms and the bound $\opnorm \leq \nuclear.$  

For $\Rop,$ from Lemma~\ref{lem:quant_Airy}, we have 
\(
\nuclear[\Mop_3] 
= O(e^{-\tfrac 4 3 (t_2)^{3/2}}),
\)
and hence 
\begin{align}
\opnorm[\Rop]
&\leq \frac{1}{1-
	\opnorm[\pi_2{\tilde K}(u_2,\cdot;u_2,\cdot) \pi_2]} \nonumber \\
&\leq \frac{1}{1-
	\nuclear[\pi_2{\tilde K}(u_2,\cdot;u_2,\cdot) \pi_2]} \nonumber \\
&\leq 1 + O(e^{-\tfrac 4 3 (t_2)^{3/2}})
	\label{eq:Rop}
\end{align}
For $\Mop,$ from Lemma~\ref{lem:quant_Airy}, we have
\begin{equation}
	\nuclear[\Mop]
	=O\left(
	e^{-\tfrac{2}{3}((t_1)^{3/2}+(t_2)^{3/2})}
	\right).
	\label{eq:Mop}
\end{equation}

The main work is to estimate the nuclear norms of $\Dop_1$ and $\Dop_2.$  We give the proof for $\Dop_1.$  The proof for $\Dop_2$ follows from an identical argument.
We begin by writing
\[
\Dop_1 =
\Dop_1'
+
\pi_1'
\left(
{\tilde K}(u_2,\cdot;u_1,\cdot)
-
{\tilde K}(u_1,\cdot;u_1,\cdot)
\right)\pi_1', 
\]
where
\[
\Dop_1'
=
\pi_1'
(\phiop - \Id)
\pi_2
{\tilde K}(u_2,\cdot;u_1,\cdot)\pi_1'.
\]
Then by Lemma~\ref{lem:Airy_difference}, we have
\begin{equation}
	\nuclear[\Dop_1-\Dop'_1]
	=O\left(e^{-\tfrac{4}{3}(t_1)^{3/2}}\right).
	\label{eq:DopDop}
\end{equation}

Let $\rho \in C^\infty$ be an increasing function which is $0$ on $(-\infty,t_2-1]$ and which is $1$ on $[t_2, \infty).$ We can clearly choose $\rho$ so that its derivatives are bounded independently of $t_2.$  We now divide $\Dop_1' = \Dop_1'' + \Dop_1'''$ where
	\begin{align*}
		\Dop_1''&=
\pi_1'
(\phiop - \Id)
\rho
{\tilde K}(u_2,\cdot;u_1,\cdot)\pi_1', \\
\Dop_1'''
&=
\pi_1'
(\phiop - \Id)
(\pi_2-\rho)
{\tilde K}(u_2,\cdot;u_1,\cdot)\pi_1'.
	\end{align*}

For $\Dop_1''$, we begin by applying Lemma \ref{lem:phi_identity}, part (iv), to conclude that
\begin{align}
	\nuclear[\Dop_1'']
	&\leq C\frac{\Delta u}{u^{2/3-\beta}}\nuclear[(\Id - \Laplacian)\rho {\tilde K}(u_2,\cdot;u_1,\cdot)\pi_1']. \nonumber\\
	\intertext{Applying Lemma~\ref{lem:quant_Airy} and using the boundedness of the derivatives of $\rho,$ we have}
	\nuclear[\Dop_1'']
	&\leq C\frac{\Delta u}{u^{2/3-\beta}}
	e^{-\tfrac{2}{3}((t_1)^{3/2}+(t_2)^{3/2})}. 
	\label{eq:Dop2}
\end{align}

Define $\pi_2'$ to be the restriction operator to the interval $[t_2 -1,t_2],$ and note that $\pi_2' (\pi_2 - \rho) = \pi_2  - \rho = (\pi_2 - \rho)\pi_2'.$  Also observe that $\pi_1'(\pi_2-\rho) = 0.$  Hence we may write
\begin{align}
	\nuclear[\Dop_1''']
	&=
\nuclear[
\pi_1'
\phiop
(\pi_2-\rho)
{\tilde K}(u_2,\cdot;u_1,\cdot)\pi_1'
] \nonumber \\
&\leq
\opnorm[
\pi_1'
\phiop
\pi_2']
\opnorm[
(\pi_2-\rho)
]
\nuclear[
\pi_2'
{\tilde K}(u_2,\cdot;u_1,\cdot)\pi_1'
]
\label{eq:Dop3}
\end{align}
The operator norm of $\pi_2 - \rho$ is at most $1$, and the nuclear norm of the ${\tilde K}$ term can be controlled using Lemma~\ref{lem:quant_Airy}.  It just remains to estimate the operator norm of $\pi_1'\phiop\pi_2'$.

By Lemma~\ref{lem:phi_identity} part (ii), we have a pointwise estimate on the kernel of $\pi_1'\phiop\pi_2$, given by
\[
	|{\tilde \phi}( u_1, \yM[1] ; u_2, \yM[2])|
\leq 
C 
\sqrt{\frac{u}{\Delta u}}
\exp\left(
-\frac{u^{2/3}}{2\Delta u}\left(\yM[1] - \yM[2]\right)^2
\right)
+C\exp\left( -\frac{(\Delta u)^{1/3}}{C} \right)
\]
As we are working on a domain of $\yM[i]$ for which $|\yM[i]| < u^{\beta},$the contribution of the $O(e^{-\Omega( (\Delta u)^{1/3})})$ term to the operator norm of $\pi_1'\phiop\pi_2'$ is still $O(e^{-\Omega( (\Delta u)^{1/3})}),$ which can be seen by computing a Hilbert-Schmidt norm.  Hence we can estimate
\[
  \|\pi_1'\phiop\pi_2'[f]\|_{L^2}
	\leq
	C\|\phiop'[|f|]\|_{L^2}
+O(e^{-\Omega( (\Delta u)^{1/3})})\|f\|_{L^2},
\]
where $\phiop'$ is the convolution operator
\[
	\phiop'[f]
	=
	\sqrt{\frac{u^{2/3}}{2\pi\Delta u}}
	e^{-\frac{u^{2/3}}{2\Delta u}\left(\cdot \right)^2}
	\one[\cdot \geq t_1 - t_2]
	* f.
\]
By Young's inequality the $L^2 \to L^2$ operator norm of $\phiop'$ is just given by its $L^1$ norm.  Hence
\[
	\opnorm[\phiop'] \leq \Pr[ Z >  u^{1/3}(\Delta u)^{-1/2} \Delta t],
\]
where $Z$ is a standard normal variable.  Hence we have shown that
\begin{equation}
	\opnorm[
\pi_1'
\phiop
\pi_2']
\leq C\Pr[ Z >  u^{1/3}(\Delta u)^{-1/2} \Delta t]
+ O(e^{-\Omega( (\Delta u)^{1/3})}).
	\label{eq:Z}
\end{equation}
Combining this equation with \eqref{eq:Dop3}, \eqref{eq:Dop2}, and \eqref{eq:DopDop} we have
\[
	\nuclear[\Dop_1]
	\leq
	C
	\left[\frac{\Delta u}{u^{2/3-\beta}}
+
\Pr[ Z >  u^{1/3}(\Delta u)^{-1/2} \Delta t]
\right]
	e^{-\tfrac{2}{3}((t_1)^{3/2}+(t_2)^{3/2})}
\]
The same argument shows the same bound for $\Dop_2.$  Hence, combining these bounds with \eqref{eq:Schurpieces}, \eqref{eq:Rop}, and \eqref{eq:Mop}, the proof is complete. 
\end{proof}

\section{Sharp uniform estimates of $\tilde K$ in the right tail}
\label{sec:sharp}
In this section we give some sharp estimates relevant to the right tail of the largest eigenvalue distribution.  
Our first estimate is a bound on the nuclear norm of the derivatives of ${\tilde K}.$  In the case $u_1=u_2,$ these are standard, and the bound here is a small extension of them.
\begin{lemma}
  For each $\delta>0$ and for each integer $\ell \geq 0,$ there is a constant $C>0$ so that for all $u_1,u_2 \in \mathbb{N}$ satisfying $|u_2 - u_1| = O(u_1^{2/3-\delta}),$
  and all $t_1,t_2 > 1$
\[
\left\|
	\pi_1 \partial_{\yM[1]}^\ell \frac{{\tilde K}(u_1,\yM[1];u_2,\yM[2])}{\sqrt{2}u_2^{1/6}} \pi_2
	\right\|_{\nu} \\
	\leq Ct_1^{\ell}\xi'(u_1,t_1) \xi'(u_2,t_2)
\]
  where
  \[
    \xi'(u_i, t_i) = C(e^{-\tfrac 23 t_i^{3/2}}
    +e^{-u_i^{1/12}t_i/C}).
  \]
\label{lem:quant_Airy}
\end{lemma}

The second bound is a quantitative convergence of ${\tilde K}$ to the Airy kernel.  Again, such bounds have been proven in the diagonal case.
\begin{lemma}
  For each $\delta>0$ and for each integer $\ell \geq 0,$ there is a constant $C>0$ so that for all $u_1,u_2 \in \mathbb{N}$ satisfying $|u_2 - u_1| = O(u_1^{2/3-\delta}),$ and all $t_1,t_2 > 1$
  \begin{multline*}
\left\|
\pi_1 \frac{{\tilde K}(u_1,\cdot;u_2,\cdot)}{\sqrt{2}u_2^{1/6}} \pi_2
  -
  \pi_1 {\tilde K}_{\operatorname{Airy}}(u_1,\cdot;u_2,\cdot) \pi_2
 \right\|_{\nu} \\
 \leq C\biggl(\frac{(\log u_1)^8}{u_1^{1/3}}e^{-\tfrac{2}{3}((t_1)^{3/2}+(t_2)^{3/2})} + e^{-(\log u_1)^2/C}\biggr).
\end{multline*}
\label{lem:Airy_difference}
\end{lemma}

The work done for proving Lemmas \ref{lem:quant_Airy} and 
\ref{lem:Airy_difference} will allow us 
to give a quick proof of
the following uniform tail bounds for the largest eigenvalue of GUE;
these imply \eqref{eq-levuppertail}.
\begin{lemma}
  There are constants $C>0$ and $\delta > 0$ so that for all $1 \leq t \leq \delta u^{1/6}$ and all $u \in \mathbb{N},$
  \[
    \frac{1}{C}e^{-\frac{4}{3}t^{3/2}}
    \leq
    \Pr\left[ \lev[u] > t \right] 
    \leq 
    Ce^{-\frac{4}{3}t^{3/2}}.
  \]
  \label{lem:levuppertail}
\end{lemma}
We note that Lemma \ref{lem:levuppertail} could also 
be deduced from the uniform Plancherel-Rotach asymptotics for Hermite
polynomials contained  in \cite{Skovgaard,Sun}. 
For completeness, we provide a self-contained proof of the lemma at the 
end of this section.

Our proofs in this section are based on a different representation of $\tilde K$ than the double-contour integral formulae used in Section~\ref{sec:contours}. 
Recall from \eqref{eq:phi+K} that
we have the representation for $\phi + K:$
\begin{align*}
  \label{eq:phi+K}
       (\phi+ K)(u_1,y_1;u_2,y_2)
        &=
        \frac{2^{u_1-u_2}}{2(\pi i)^2}
        \oint\limits\int\limits
        \frac{e^{z_1^2 - 2z_1y_1}}{e^{z_2^2 - 2z_2y_2}}
        \frac{z_1^{u_1}}{z_2^{u_2}}
        \frac{dz_1dz_2}{z_1 - z_2} \\
        &=
        \frac{2^{u_1-u_2}}{(\pi i)^2}
        \oint\limits\int\limits
        \int\limits_0^\infty
        \frac{e^{z_1^2 - 2z_1(y_1+w)}}{e^{z_2^2 - 2z_2(y_2+w)}}
        \frac{z_1^{u_1}}{z_2^{u_2}}
        \,
        {dwdz_1dz_2} \\
&=
        \frac{2^{u_1-u_2}}{(\pi i)^2}
        \int\limits_0^\infty
        \oint\limits\int\limits
        \frac{e^{z_1^2 - 2z_1(y_1+w)}}{e^{z_2^2 - 2z_2(y_2+w)}}
        \frac{z_1^{u_1}}{z_2^{u_2}}
        \,
        {dz_1dz_2dw}. \\
	\intertext{We now scale the $w$ variable, introducing $\wM = \sqrt{2}u_1^{1/6}w.$ 
	We also recall the notation $\G[i]$ used in \eqref{eq:G}.  In terms of these variables, we have}
       (\phi+ K)(u_1,y_1;u_2,y_2)
        &=
	\frac{J(u_1,y_1)}{J(u_2,y_2)}
	\frac{1}{(\pi i)^2}
        \int\limits_0^\infty
        \oint\limits\int\limits
\frac{e^{u_1 \G[1]}}{e^{u_2 \G[2]}} 
        \frac{e^{-2z_1w}}{e^{-2z_2w}}
        \,
        {dz_1dz_2dw} \\
&=
	\frac{J(u_1,y_1)}{J(u_2,y_2)}
	\frac{1}{(\pi i)^2}
        \int\limits_0^\infty
        \oint\limits\int\limits
\frac{e^{u_1 \G[1]}}{e^{u_2 \G[2]}} 
\frac{e^{-(\zM[1] + u_1^{1/3})\wM}
\,
{dz_1 dz_2dw}
}{e^{-(\zM[2] + u_2^{1/3})u_1^{-1/6}u_2^{1/6}\wM}}.
\end{align*}

Hence, changing the integration to be over $\wM,$ 
we arrive at the following expression for ${\tilde K}:$
\begin{equation}
  \frac{{\tilde K}(u_1,\yM[1];u_2,\yM[2])}{\sqrt{2}u_2^{1/6}}
  =
  \frac{1}{(2\pi i)^2}
  \int\limits_0^\infty
        \oint\limits_{z_2}\int\limits_{z_1}
\frac{e^{u_1 \G[1]}}{e^{u_2 \G[2]}} 
\frac{e^{-\zM[1]\wM}
\,
{d\zM[1] d\zM[2]d\wM }
}{e^{-\zM[2]\wM + \xi_2(\zM[2],\wM,u_1,u_2)}},
  \label{eq:KHankel}
\end{equation}
where 
\begin{equation}
  \xi_2(\zM[1],\wM,u_1,u_2)
  = -\zM[2] \wM 
  ({u_2^{1/6}}{u_1^{-1/6}}-1)
  + 
  \left( 
  u_1^{1/3}
  -u_2^{1/2}u_1^{-1/6}
  \right)\wM.
  \label{eq:xiH}
\end{equation}

Recalling that $u_i \G[i] - \zM[i]\wM = u_i G_i(\zM[i],\yM[i]+\wM),$ we define
\begin{align*}
  {\tilde K}_{1}(\yM[1],\wM)
  &=
  \frac{\one[\wM \geq 0]}{2\pi i}
  \int\limits_{\tilde{\gamma}_1}
e^{u_1
G_1(\zM[1],\yM[1]+\wM)
}
d\zM[1], \\
  {\tilde K}_{2}(\wM,\yM[2])
  &=
  \frac{\one[\wM \geq 0]}{2\pi i}
\int\limits_{\tilde{\gamma}_2}
e^{-
u_2
G_2(\zM[2],\yM[2]+\wM)
}
e^{\xi_2(\zM[2],\wM,u_1,u_2)}
d\zM[2].
\end{align*}
The $\tilde{\gamma}_1$ contour is any vertical line for which $\Re \zM[1] > -u_1^{1/3},$ and the $\tilde{\gamma}_2$ contour is any closed loop that encloses $-u_2^{1/3}.$ 
Let $\Kop_1$ and $\Kop_2$ be the corresponding operators from $L^2(dx) \to L^2(dx),$ so that \eqref{eq:KHankel} becomes $\Kop/(\sqrt{2}u_2^{1/6}) = \Kop_1 \cdot \Kop_2.$

The estimates in this section all in a sense rely on precise comparison between ${\tilde K}_i$ and an Airy function. 
Recall that the Airy kernel has the representation
\begin{equation}
  \Kairy(\yM[1],\yM[2])
  =
\int\limits_0^\infty
\Ai(\yM[1]+\wM)
\Ai(\yM[2]+\wM)\,d\wM.
  \label{eq:AiHankel}
\end{equation}
Let $\Aiop$ be the operator with kernel $A(x,y) = \Ai(x+y)\one[y \geq 0].$  
Then $A$ has the representation
\[
  A(\yM[1],\wM)
  =
  \Ai(\yM[1]+\wM)
  = \frac{\one[\wM \geq 0]}{2\pi i} 
  \int_{\infty e^{-i\pi/3}}^{\infty e^{i \pi /3}}
  e^{\zM[1]^3/3 - \zM[1] (\yM[1]+\wM)}\,d\zM[1].
\]
The minimum phase contour for this integral is given by the hyperbola ${\tilde h}_1$
\[
  -\frac{(\Im \zM[1])^2}{3} + (\Re \zM[1])^2 = \yM[1] + \wM
\]
which is asymptotic to the contour used to define $\Ai$ as $\zM[1] \to \infty.$  On this contour we have 
\begin{align}
  \Re\left( \zM[1]^3/3 - \zM[1] (\yM[1]+\wM)\right)
  &=
  -
  \left( \yM[1]+\wM + (\Im \zM[1])^2/3 \right)^{1/2}
  \left( \frac{2(\yM[1]+\wM)}{3} + \frac{8(\Im \zM[1])^2}{9}\right) \nonumber\\
  &\leq
  -\frac{2\left( \yM[1] + \wM \right)^{3/2}}{3}
  -\frac{8(\Im \zM[1])^2 (\yM[1] + \wM)^{1/2}}{9}. \label{eq:minphasebound}
\end{align}

We will essentially use ${\tilde h}_1$ to represent ${\tilde K}_1.$  However, this is a poor choice of contour for large values of $\zM[1].$  Hence, let $R$ be a truncation parameter to be determined later, and let ${\tilde h}_1^R$ be the portion of this contour with imaginary part at most $R$ in absolute value.  Define
\[
A(\yM[1],\wM)^R
  = \frac{\one[\wM \geq 0]}{2\pi i} 
  \int\limits_{ {\tilde h}_1^R}
  e^{\zM[1]^3/3 - \zM[1] (\yM[1]+\wM)}\,d\zM[1].
\]
We will parameterize ${\tilde h}_1$ by its imaginary part, noting the arc length differential is uniformly bounded in this parameterization, so that
\begin{align}
 \left| 
A(\yM[1],\wM)
-
A(\yM[1],\wM)^R
\right|
&\leq O(1)\int_R^\infty
  e^{-\frac{2\left( \yM[1] + \wM \right)^{3/2}}{3}
  -\frac{8t^2 (\yM[1] + \wM)^{1/2}}{9}}\,dt. \nonumber \\
  &\leq \frac{O(1)}{(\yM[1] + \wM)^{1/4}}
  e^{-\frac{2\left( \yM[1] + \wM \right)^{3/2}}{3}
  -\frac{8R^2 (\yM[1] + \wM)^{1/2}}{9}}.
  \label{eq:Atrunc}
\end{align}



\begin{figure}
\begin{tikzpicture}
    \begin{axis}[
		    title=${\tilde h}_2$ and ${\tilde h}_2^e$,
		    width=5cm,
		    height=8cm,
            xmin=-4,xmax=0,
        ymin=-4,ymax=4,
                    axis x line=middle,    
                    axis y line=right,    
                    axis line style={->},
		    xlabel={ $\Re{\tilde z}_2$ },          
		    xtick=\empty,
		    extra x ticks={-1},
		    extra x tick style={
			    grid=major,
			    tick label style={rotate=90,left=1cm,yshift=-0.3cm}
		    },
		    extra x tick labels={$-({\tilde y}_2+\wM)^{1/2}$},
		    extra y ticks={2,-2},
		    extra y tick style={
			    grid=major
		    },
		    extra y tick labels={$R$, $-R$},
		    ytick=\empty
                    ]
        \addplot [red,thick,domain=-3:3] ({-cosh(x)}, {sqrt(3)*sinh(x)});
	\addlegendentry{ ${\tilde h}_2$}
        \addplot [blue,dashed,thick,domain=-1:1,forget plot] ({-cosh(x)}, {sqrt(3)*sinh(x)});
        \addplot [blue,thick,domain=2:3.144] ({-1.52753}, {x});
        \addplot [blue,thick,domain=-3.144:-2] ({-1.52753}, {x});
        \addplot [blue,thick,domain=51.81:308.19,samples=700] ({4*cos(x)-4}, {4*sin(x)});
	\addlegendentry{ ${\tilde h}_2^e$}
        \addplot[red,dotted] expression {sqrt(3)*x};
        \addplot[red,dotted] expression {-sqrt(3)*x};
        \addplot [blue,dotted,domain=-51.81:51.81,samples=100] ({4*cos(x)-4}, {4*sin(x)});
    \end{axis}
\end{tikzpicture}
\begin{tikzpicture}
\begin{axis}[
		    title=${\tilde h}_1$ and ${\tilde h}_1^e$,
		    width=5cm,
		    height=8cm,
            xmin=0,xmax=4,
        ymin=-4,ymax=4,
                    axis x line=middle,    
                    axis y line=left,    
                    axis line style={->}, 
		    xlabel={$\Re {\tilde z}_1$},          
		    xtick=\empty,
		    extra x ticks={1},
		    extra x tick style={
			    grid=major,
			    tick label style={rotate=90,left=1cm,yshift=0.3cm}
		    },
		    extra x tick labels={$({\tilde y}_1+\wM)^{1/2}$},
		    extra y ticks={2,-2},
		    extra y tick style={
			    grid=major
		    },
		    extra y tick labels={$R$, $-R$},
		    ytick=\empty
                    ]
        \addplot [red,thick,domain=-3:3] ({cosh(x)}, {sqrt(3)*sinh(x)});
	\addlegendentry{ ${\tilde h}_1$}
        \addplot [blue,dashed,thick,domain=-1:1,forget plot] ({cosh(x)}, {sqrt(3)*sinh(x)});
        \addplot [blue,thick,domain=2:4] ({1.52753}, {x});
        \addplot [blue,thick,domain=-4:-2] ({1.52753}, {x});
	\addlegendentry{ ${\tilde h}_1^e$}
        \addplot[red,dotted] expression {sqrt(3)*x};
        \addplot[red,dotted] expression {-sqrt(3)*x};
    \end{axis}
\end{tikzpicture}
\caption{Contours used to compare ${\tilde K}_1,$ ${\tilde K}_2$ and $A.$}
\label{fig:h1h2}
\end{figure}
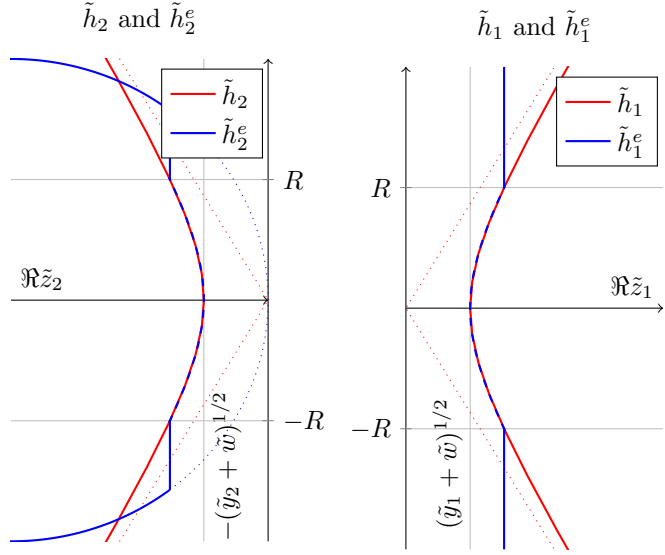

Turning to ${\tilde K}_1,$ we define a contour ${\tilde h}_1^e$ by extending ${\tilde h}_1^R$ to $\pm i\infty$ by vertical lines.  We deform the integral in the definition of ${\tilde K}_1$ to be over ${\tilde h}_1^e:$
\begin{equation}
	  {\tilde K}_1(\yM[1],\wM)
  =
  \frac{\one[\wM \geq 0]}{2\pi i}
  \int\limits_{\tilde{h}_1^e}
e^{u_1
G_1(\zM[1],\yM[1]+\wM)
}
\,d\zM[1]
	\label{eq:h1e}
\end{equation}
and we define
\[
  {\tilde K}_{1}^R(\yM[1],\wM)
  =
  \frac{\one[\wM \geq 0]}{2\pi i}
  \int\limits_{{\tilde h}_1^R}
e^{u_1
G_1(\zM[1],\yM[1]+\wM)
}
d\zM[1].
\]

We will need some estimates which are useful for large values of $\zM[1]$ to control the difference of ${\tilde K}_1$ and ${\tilde K}_1^R.$  To this end, let
\begin{equation}
  F(z) = \Re \left[ 
  \log(1+z) - z + z^2/2
  \right],
  \label{eq:F}
\end{equation}
so that $u_i\Re(\G[i]) = u_iF(u_i^{-1/3}\zM[i]) - \zM[i]\yM[i].$  In Cartesian coordinates, we have
\begin{equation}
   F(x+iy) = \frac{1}{2} \log( (1+x)^2 + y^2) -x 
  + \frac{x^2}{2}
  - \frac{y^2}{2}.
  \label{eq:Fxy}
\end{equation}
We estimate $F(x+iy),$ for $x > 0$ by
\begin{align}
  F(x+iy)
  &=
  \frac{1}{2} \log( (1+x)^2 ) 
  +\frac{1}{2} \log\left( 1 + \frac{y^2}{(1+x)^2}\right) 
  -x 
  + \frac{x^2}{2}
  - \frac{y^2}{2} \nonumber \\
  &\leq
\frac{y^2}{2}\left( \frac{1}{(1+x)^2} - 1 \right)  \nonumber \\
&\leq -\frac{xy^2}{1+x}  \label{eq:Feasy}.
\end{align}

The real parts of the endpoints of ${\tilde h}_1^R,$ are at least $R/\sqrt{3}.$
Hence on the vertical portions of the ${\tilde h}_1^e$, by \eqref{eq:Feasy}, we have
\[
  u_1\Re(G_1(\zM[1],\yM[1]))
  \leq -\frac{R(\Im \zM[1])^2}{1+Ru_1^{-1/3}} - R\yM[1]/\sqrt{3}.
\]
Ergo
\begin{align}
 \left| 
 {\tilde K}_1(\yM[1],\wM)
-
{\tilde K}_1^R(\yM[1],\wM)
\right|
&\leq O(1)\int_R^\infty
e^{-u_1\Re(G_1(\zM[1],\yM[1]+\wM))}\,d(\Im \zM[1]). \nonumber \\
&= O(R^{-1/2})e^{-\Omega(R^3)-R(\yM[1]+\wM)/\sqrt{3}},
  \label{eq:K1trunc}
\end{align}
provided we have $R=O(u_1^{1/3}).$

For ${\tilde K}_2,$ we begin by representing the Airy function using a rotated contour:
\[
  A(\wM, \yM[2])^*
  =
  \Ai(\yM[2]+\wM)
  = \frac{\one[\wM \geq 0]}{2\pi i} 
  \int_{\infty e^{-i2\pi/3}}^{\infty e^{i 2\pi /3}}
  e^{-\zM[2]^3/3 + \zM[2] (\yM[2]+\wM)}\,d\zM[2].
\]
Once again the minimum phase contour for this integral is given by a hyperbola ${\tilde h}_2$ satisfying the same equation
\[
 -\frac{(\Im \zM[2])^2}{3} + (\Re \zM[2])^2 = \yM[2] + \wM,
\]
although we now take the branch opening to the left.  Letting ${\tilde h}_2^R$ be the portion of the hyperbola with imaginary part at most $R$ and defining \( A(\wM, \yM[2])^{*R} \) to be the restriction of the integral to
${\tilde h}_2^R,$ we get exactly the same bound as \eqref{eq:Atrunc}:
\begin{align}
 \left| 
A(\yM[1],\wM)^*
-
A(\yM[1],\wM)^{*R}
\right|
  \leq \frac{O(1)}{(\yM[1] + \wM)^{1/4}}
  e^{-\frac{2\left( \yM[1] + \wM \right)^{3/2}}{3}
  -\frac{8R^2 (\yM[1] + \wM)^{1/2}}{9}}.
  \label{eq:AStartrunc}
\end{align}

Define a contour ${\tilde h}_2^e$ that extends ${\tilde h}_2^R$ first along a vertical line and then along the circle $\tilde S,$ given by $|1+u_2^{-1/3}\zM[2]| = 1-u_2^{-1/3}$
(see Figure \ref{fig:h1h2}).  On $\tilde S$, we have $\Re \zM[2] \leq -1.$
Provided we take $R = o(u_2^{1/3}),$ this is well-defined.  

For $F,$ we get from \eqref{eq:F} or \eqref{eq:Fxy} that for $\zM[2]$ on this circle,
\begin{equation}
	\label{eq:FS}
  F(u_2^{-1/3}\zM[2])
  =\log(1-u_2^{-1/3}) - \frac{(1-u_2^{-1/3})^2 - 1}{2}
  =O(u_2^{-1}).
\end{equation}
Hence we have
\begin{equation}
	\label{eq:xiS}
  -u_2\Re(G_2(\zM[2],\yM[2]))
  = O(1) + \zM[2] \yM[2] \leq O(1) - \yM[2].
\end{equation}
By \eqref{eq:FS}, and the max-modulus principle, we have for $\zM[2] \in {\tilde h}_2^e \setminus {\tilde h}_2^R$
\[
	-u_2\Re( \G[2]) \leq O(1) - R\yM[2].
\]
As for $\xi_2,$ since $\Re \zM[2] \geq -2u_2^{1/3},$ we have
\begin{align*}
  \xi_2 
  &\leq \left( 2u_2^{1/2}u_1^{-1/6} - 2u_2^{1/3} + u_1^{1/3}-u_2^{1/2}u_1^{-1/6} \right) \wM \\
  &\leq O\left( \frac{u_2-u_1}{u_1^{2/3}}\right) \wM.
\end{align*}

Applying this bound and \eqref{eq:xiS}, we get
\begin{align}
 \left| 
 {\tilde K}_2(\yM[2],\wM)
-
{\tilde K}_2^R(\yM[2],\wM)
\right|
&\leq \int\limits_{ {\tilde h}_2^e \setminus {\tilde h}_2^R}
e^{-u_2\Re(G_2(\zM[2],\yM[2]+\wM))}
e^{\Re \xi_2}
\,|d\zM[2]|
\nonumber \\
&\leq \int\limits_{ {\tilde h}_2^e \setminus {\tilde h}_2^R}
e^{O(1) -(R-o(1))(\yM[2]+\wM)}
\,|d\zM[2]|
\nonumber \\
&\leq 
O(u_2^{1/3})e^{-(R-o(1))(\yM[2]+\wM)}.
  \label{eq:K2trunc}
\end{align}

\begin{proof}[Proof of~\ref{lem:quant_Airy}]
  Recall that \eqref{eq:KHankel} can be expressed as $\Kop/(\sqrt{2}u_2^{1/6}) = \Kop_1 \cdot \Kop_2.$  Hence, we have the estimate
  \[
    \nuclear[ \pi_1 \partial^\ell \Kop/(\sqrt{2}u_2^{1/6}) \pi_2]
    \leq 
    \HSN{\pi_1 \partial^\ell \Kop_1}
  \HSN{\Kop_2 \pi_2}.
  \]
  (As before, $\nuclear[\cdot]$ denotes the nuclear norm.)

We start with estimating the second Hilbert-Schmidt norm.  By \eqref{eq:K2trunc}, we have that the Hilbert-Schmidt norm of ${\tilde K}_2 - {\tilde K}_2^R$ is at most $O(u^{1/3}/R^{1/2}e^{-0.99 Rt_2}).$  As for the norm of ${\tilde K}_2^R,$ 
define
\[
  \zeta_2' = u_2\G[2] + \zM[2]\yM[2] - \frac{1}{3}\zM[2]^3. 
\]
Expanding $G_2$ as in \eqref{eq:taylor} we have
\(
\zeta_2' = O(u_2^{-1/3}\zM[2]^4).
\)
Hence on ${\tilde h}_2^R,$ with $R = O(u_2^{1/12})$ we have

\begin{align}
  |{\tilde K}_2^R(\yM[2], \wM)|
&\leq
\int\limits_{ {\tilde h}_2^R}
e^{\Re \zeta_2'}
e^{
  \Re\left( \zM[2]^3/3 - \zM[2] (\yM[2]+\wM)\right)
}
\,|d\zM[2]| 
\nonumber \\
&\leq
e^{O(1)}
\int_{-R}^R
 e^{-\frac{2\left( \yM[2] + \wM \right)^{3/2}}{3}
  -\frac{8t^2 (\yM[2] + \wM)^{1/2}}{9}}\,dt.
\nonumber \\
&\leq
e^{O(1)}
\frac{
  e^{-\frac{2\left( \yM[2] + \wM \right)^{3/2}}{3}}}
  {(\yM[2] + \wM)^{1/4}}.
  \label{eq:K2bulkbound}
\end{align}
Hence we get that the Hilbert-Schmidt norm of ${\tilde K}_2^R$ is $O(e^{-\tfrac 23 t_2^{3/2}}),$
so that $\HSN{\Kop_2 \pi_2} = O(e^{-\tfrac 23 t_2^{3/2}} + e^{-0.98 u_2^{1/12}t_2}).$

As for $\HSN{\pi_1 \partial^{\ell} \Kop_1},$ observe that we have the kernel representation
\[
  \partial_{\yM[1]}^{\ell} {\tilde K}_{1}(\yM[1],\wM)
  =
  \frac{\one[\wM \geq 0]}{2\pi i}
  \int\limits_{\tilde{\gamma}_1}
  \zM[1]^\ell
e^{u_1
G_1(\zM[1],\yM[1]+\wM)
}
\,d\zM[1].
\]
Hence, the same truncation approach as used in \eqref{eq:K1trunc} for the $\ell = 0$ case works here.  Further, the same argument as given for $\Kop_2$ shows that
\(
\HSN{\pi_1 \partial^{\ell} \Kop_1} t_1^{-\ell}
= O_\ell(e^{-\tfrac 23 t_1^{3/2}} + e^{-u_1^{1/12}t_1/C})
\)
for some constant $C>0,$
which concludes the proof.
\end{proof}

\begin{proof}[Proof of~\ref{lem:Airy_difference}]
We can bound the nuclear norm quantity we seek to estimate as
\begin{align}
  \nuclear[
    \pi_1 {\tilde K}(u_1,\cdot;u_2,\cdot) \pi_2
  -
\pi_1 {\Kairy}(\cdot, \cdot) \pi_2] \hspace{-1.5in}&\hspace{1.5in} \nonumber \\
&=
\nuclear[
\pi_1\Kop_1 \cdot \Kop_2\pi_2 - \pi_1\Aiop \cdot \Aiop^*\pi_2
] \nonumber \\
&\leq
\nuclear[
\pi_1(\Kop_1-\Aiop) \cdot \Kop_2\pi_2 
]+
\nuclear[
\pi_1(\Aiop) \cdot (\Kop_2-\Aiop^*)\pi_2 
]\nonumber  \\
&\leq
\HSN{\pi_1(\Kop_1-\Aiop)}
\HSN{\Kop_2\pi_2}
+
\HSN{\pi_1\Aiop}
\HSN{(\Kop_2-\Aiop^*)\pi_2}.
  \label{eq:nuclearHS}
\end{align}

By Lemma~\ref{lem:quant_Airy}, we can control the
$\HSN{\Kop_2\pi_2} = O(e^{-\tfrac 23 t_2^{3/2}} + e^{-\Omega(u^{1/12})}).$
From the standard Airy asymptotic
\begin{equation}
  \Ai(\wM) \leq Ce^{-2(\wM)^{3/2}/3},
  \label{eq:Aextreme}
\end{equation}
see \cite[\S 9.7.5]{DLMF}.  This translates immediately into bounds of the form
\(
\HSN{\pi_1\Aiop} = O(e^{-\tfrac 23 t_1^{3/2}}).
\)
Thus it only remains to estimate both of 
$\HSN{\pi_1(\Kop_1-\Aiop)}$ \
and
$ \HSN{(\Kop_2-\Aiop^*)\pi_2}.$ 

%
%
To compare ${\tilde K}_1^R$ and $A^R,$  
define
\[
  \zeta_1' = u_1\G[1] + \zM[1]\yM[1] - \frac{1}{3}\zM[1]^3. 
\]
Expanding $G_1$ as in \eqref{eq:taylor} we have
\(
\zeta_1' = O(u_1^{-1/3}\zM[1]^4).
\)
Provided that $R \geq \Re \zM[1]$ on $\tilde{h}_1^R,$ which forces $\yM[1] +\wM < \tfrac 23 R^2,$ we get that $\zeta_1' = O(u_1^{-1/3}R^4).$
Hence under the assumption $R^4 = O(u_1^{1/3}),$ we have
\begin{align}
   \left| 
 {A}^R(\yM[1],\wM)
-
{\tilde K}_1^R(\yM[1],\wM)
\right|
&\leq
\int\limits_{ {\tilde h}_1^R}
\left|
e^{\zeta_1'}
-1
\right|
e^{
  \Re\left( \zM[1]^3/3 - \zM[1] (\yM[1]+\wM)\right)
}
\,|d\zM[1]| 
\nonumber \\
&\leq
O(R^4u_1^{-1/3})
\int_{-R}^R
 e^{-\frac{2\left( \yM[1] + \wM \right)^{3/2}}{3}
  -\frac{8t^2 (\yM[1] + \wM)^{1/2}}{9}}\,dt.
\nonumber \\
&\leq
O(R^4u_1^{-1/3})
\frac{
  e^{-\frac{2\left( \yM[1] + \wM \right)^{3/2}}{3}}}
  {(\yM[1] + \wM)^{1/4}}.
  \label{eq:AK1bulk}
\end{align}
By taking $R = \log u_1$ in \eqref{eq:Atrunc}, \eqref{eq:K1trunc}, and \eqref{eq:AK1bulk}, we have that for $t_1 > 1,$
\begin{equation}
\HSN{\pi_1(\Kop_1-\Aiop)}
\leq 
e^{-\Omega((\log u_1)^2)}
+ O\left( (\log u_1)^4 u_1^{-1/3} \right)e^{-\frac{2}{3}t_1^{2/3}}.
  \label{eq:HSN1}
\end{equation}

Finally for $\Re \zM[2] \geq -u_2^{1/3},$ we have by \eqref{eq:xiH} that $\Re\xi_2 \leq 0.$  Hence on ${\tilde h}_2^R,$ we can estimate $|e^{\xi_2}| \leq 1.$  Thus,
provided that 
$R^4=O(u_2^{1/3}),$ the same estimates as in \eqref{eq:AK1bulk} give
\begin{align}
   \left| 
   {A}^{*R}(\wM,\yM[2])
-
{\tilde K}_2^R(\wM,\yM[2])
\right|
&\leq
O(R^4u_2^{-1/3})
\frac{
  e^{-\frac{2\left( \yM[2] + \wM \right)^{3/2}}{3}}}
  {(\yM[2] + \wM)^{1/4}}.
  \label{eq:AK2bulk}
\end{align}
Taking $R = (\log u_1)^2$ in \eqref{eq:AStartrunc}, \eqref{eq:K2trunc}, and \eqref{eq:AK2bulk}, we have that for $t_2 > 1,$
\begin{equation}
\HSN{(\Kop_2-\Aiop^*)_2}
\leq 
e^{-\Omega((\log u_1)^2)}
+ O\left( (\log u_1)^8 u_1^{-1/3} \right)e^{-\frac{2}{3}t_2^{2/3}}.
  \label{eq:HSN2}
\end{equation}

Combining \eqref{eq:nuclearHS} with \eqref{eq:HSN1} and \eqref{eq:HSN2}, we have completed the proof.
\end{proof}

\begin{proof}[Proof of Lemma~\ref{lem:levuppertail}]
  We begin by recalling that 
  \[
    \Pr\left[ \lev[u] \leq t \right]
    =\det\left( \Id - \pi \Kop \pi \right),
  \]
  where $\pi$ is the restriction map to $[t, \infty)$ and $\Kop$ is given by kernel ${\tilde K}(u, \cdot;u,\cdot)/(\sqrt{2} u^{1/6})$ acting on $L^2.$  Note that $K(u_1,y_1;u_2,y_2)^{y_2^2/2 - y_1^2/2}$ is self-adjoint and positive definite, and hence so is the kernel restricted to  $\min |y_i| > a$ for any $a$.  This implies all the eigenvalues of $K$ are non-negative, and hence so are all the eigenvalues of $\pi \Kop \pi.$ Thus, we have the representation
    \[
    \det\left( \Id - \pi \Kop \pi \right)
    =\exp\left( -\sum_{n=1}^\infty \frac{\tr( (\pi \Kop \pi)^n)}{n} \right),
    \]
    which can be seen by considering the eigenvalues of the operator $\pi\Kop\pi$ (see also \cite[(3.9)]{GGK}).  We also have that all traces are non-negative and
    \(
    {\tr( (\pi \Kop \pi)^n)}
    \leq
    {\tr( (\pi \Kop \pi))}^n.
    \)
    Hence, we get the simple bounds
    \[
    1 - \tr ( \pi \Kop \pi)
    \leq
    \det\left( \Id - \pi \Kop \pi \right)
    \leq e^{-\tr (\pi \Kop \pi)}.
    \]
    Turning this around,
    \begin{equation}
      (1- e^{-\tr (\pi \Kop \pi)})
      \leq
    \Pr\left[ \lev[u] > t \right]
    \leq \tr ( \pi \Kop \pi).
      \label{eq:trBounds}
    \end{equation}
    Thus it only remains to give upper and lower bounds for the trace.

    The trace is given by
    \begin{equation}
\tr ( \pi \Kop \pi)
= \int_t^\infty {\tilde K}(\yM,\yM)\,d\yM.
= 
\int_t^\infty 
\int_0^\infty 
{\tilde K}_1(\yM,\wM)
{\tilde K}_2(\wM,\yM)
\,d\wM d\yM.
      \label{eq:tr}
    \end{equation}
    Using \eqref{eq:K1trunc} and \eqref{eq:K2trunc}, we have that
    \[
\tr ( \pi \Kop \pi)
= 
\int_t^\infty 
\int_0^\infty 
{\tilde K}_1^R(\yM,\wM)
{\tilde K}_2^R(\wM,\yM)
\,d\wM d\yM
+ O(e^{-\Omega(R t)}).
    \]


Both of ${\tilde K}_1^R(\yM,\wM)$ or ${\tilde K}_2^R(\yM,\wM)$ are real, 
as their integrands commute with conjugation as functions of $\zM[1]$ and the contours $\{{\tilde h}_i^R\}$ are conjugation invariant.
Recall $\zeta_1'$ from \eqref{eq:AK1bulk}, using which we may write
\[
{\tilde K}_1^R(\yM,\wM)
=
\int\limits_{ {\tilde h}_1^R}
e^{\zeta_1'}
e^{
  \zM^3/3 - \zM (\yM+\wM)
}
\,d\zM. 
\]
As \(\zM^3/3 - \zM (\yM+\wM)\) is real-valued on ${\tilde h}_1^R,$
and as $|e^{\zeta_1'}-1| = O(u^{-1/3}R^4)$ on ${\tilde h}_1^R,$ by making $R$ is a sufficiently small multiple of $u^{1/12},$ we may make 
\[
\tfrac 12 
A^R(\yM,\wM)
\leq 
\Re{\tilde K}_1^R(\yM,\wM)
\leq 
2 A^R(\yM,\wM)
\]
for all $\yM > t$ and all $\wM \geq 0.$  A similar statement holds for ${\tilde K}_2^R.$
It then follows that for $R$ so chosen, we have 
\[
  \frac{1}{4} \leq
  \frac{
\int_t^\infty 
\int_0^\infty 
{\tilde K}_1^R(\yM,\wM)
{\tilde K}_2^R(\wM,\yM)
\,d\wM d\yM
}{
\int_t^\infty 
\int_0^\infty 
{A}^R(\yM,\wM)
{A}^R(\wM,\yM)
\,d\wM d\yM
}
\leq 4
\]
for all $t \geq 1.$ 
Using \eqref{eq:Atrunc}, we therefore conclude that there is a $\delta > 0$ sufficiently small and a $C>0$ sufficiently large so that with $R=\delta u^{1/12},$
\begin{equation}
  \frac{1}{C} \tr(\pi \Aiop \pi) - Ce^{-R t/C} - Ce^{-R^2t^{1/2}/C}
  \leq
\tr ( \pi \Kop \pi)
\leq C \tr(\pi \Aiop \pi) + Ce^{-R t/C} + Ce^{-R^2t^{1/2}/C}
  \label{eq:trcomparison}
\end{equation}

The trace of the Airy kernel is given by
\[
\tr(\pi \Aiop \pi)
=
\int_t^\infty 
\int_0^\infty 
\Ai(\yM+\wM)^2
\,d\wM d\yM
=
\int_t^\infty 
(s-t) \Ai(s)^2
\,ds.
\]
Using that $\Ai(s)s^{1/4}e^{\tfrac 2 3 s^{3/2}}$ is bounded above and below by constants for $s \geq 0,$ we therefore have that
\[
  \int_t^\infty 
(s-t) \Ai(s)^2
\,ds
=
e^{-\tfrac 4 3 t^{3/2}}(C+O(t^{-1/2}))
\]
for some constant $C.$  Hence by the positivity and continuity of the trace, we conclude that
\(
\tr(\pi \Aiop \pi)e^{\tfrac 4 3 t^{3/2}}
\)
is bounded above and below by constants.  This and \eqref{eq:trcomparison} completes the proof.
\end{proof}

\section{Offdiagonal kernel estimates for $u_1-u_2 \gg u_1^{2/3}$}
\label{sec:off}
\begin{lemma}
	For all $\yM[1] \geq 0,$ all $\yM[2] \geq 0,$ and all $u_1 > u_2,$
  \[
    \left| {\tilde K}_o( u_1, y_1 ; u_2, y_2) \right|
    \leq \frac{u_1^{1/6}u_2^{1/6}}{u_1^{1/2} - u_2^{1/2}} 
  \xi'(u_1,y_1) \xi'(u_2,y_2)
  \]
  where
  \[
    \xi'(u_i, y_i) = C(e^{-\tfrac 23 \yM[i]^{3/2}}
    +e^{-u_i^{1/12}\yM[i]}).
  \]
  for some absolute constant $C>0.$

   \label{lem:rt_Ko_pointwise}
\end{lemma}

\begin{proof}
	Recall~\eqref{eq:hermite3}, due to which we may express ${\tilde K}_o$ as 
\begin{equation}
  {\tilde K}_o = \frac{1}{2(\pi i)^2}
  \iint
  \frac{e^{u_1\G[1]}}{e^{u_2\G[2]}} \frac{dz_1 dz_2}{z_1 - z_2}
  \label{eq:rt_tildeK_o},
\end{equation}
where we deform the contours to be the same as those in Lemma~\ref{lem:quant_Airy}, with $\wM=0$ (see Figure \ref{fig:h1h2}) and where
$\zM[i]$ is given by \eqref{eq-scaling}.
On these contours, we have that $|z_1 - z_2| \geq 2^{1/2}\left( u_1^{1/2} - u_2^{1/2} \right).$ Hence, changing the integration to be in $\zM$, we have the simple estimate
\begin{equation}
  \left| {\tilde K}_o \right| \leq
  \frac{u_1^{1/6}u_2^{1/6}}{u_1^{1/2} - u_2^{1/2}} 
  \left[ \int_{{\tilde h}_1^e} e^{u_1 \Re\G[1]}\,|d \zM[1]|\right]
  \left[ \int_{{\tilde h}_2^e} e^{-u_2 \Re\G[2]}\,|d \zM[2]|\right].
  \label{eq:rt_tko1}
\end{equation}
Thus to complete the claimed bound, it suffices to show that each integral is bounded by $\xi$ for an appropriately large constant $C$.  

Define
\begin{align*}
	I_1^e &= \int_{{\tilde h}_1^e} e^{u_1 \Re\G[1]}\,|d \zM[1]| \\
	I_1^R &= \int_{{\tilde h}_1^R} e^{u_1 \Re\G[1]}\,|d \zM[1]|,
\end{align*}
and define $I_2^e$ and $I_2^R$ analogously.  Using the same estimate as \eqref{eq:K1trunc}, we have that for $R=O(u_1^{1/3})$
\begin{align}
 \left| 
I_1^e - I_1^R
\right|
= O(R^{-1/2})e^{-\Omega(R^3)-R\yM[1]/\sqrt{3}}.
  \label{eq:I1trunc}
\end{align}
Uniformly for $\zM[1]$ on $\tilde{h}_1^R,$ we have by \eqref{eq:minphasebound} that when $R^4=O(u_1^{1/3}),$
\[
u_1 \Re\G[1]
\leq
  -\frac{2\yM[1]^{3/2}}{3}
  -\frac{8(\Im \zM[1])^2 \yM[1]^{1/2}}{9}
  +O(R^4u_1^{-1/3}).
\]
Hence integrating over $\zM[1],$ we have 
\begin{equation}
	I_1^R \leq e^{-\frac{2\yM[1]^{3/2}}{3} + O(1)}.
	\label{eq:I1R}
\end{equation}
Combining \eqref{eq:I1trunc} and \eqref{eq:I1R} and taking $R=u_1^{1/12},$ we have
\begin{equation}
	I_1^e \leq e^{-\frac{2\yM[1]^{3/2}}{3} + O(1)} + e^{-u_1^{1/12}\yM[1]/\sqrt{3} + O(1)
}
	\label{eq:I1e}
\end{equation}

For $I_2^{e}$ and  $I_2^{R}$ we proceed in the same manner, using the same estimates for $I_2^{R}$ as in $I_1^{R}$ and using the estimate from \eqref{eq:K2trunc} to make the comparison.  Again taking $R=u_2^{1/12}$ we get
\begin{equation}
	I_2^e \leq e^{-\frac{2\yM[2]^{3/2}}{3} + O(1)} + e^{-u_2^{1/12}\yM[2]/\sqrt{3} + O(1)}.
	\label{eq:I2e}
\end{equation}
\end{proof}

\begin{lemma}
  There is an absolute constant $\sqrt{2}>c>0$ so that for all $y_1 \geq cu_1^{1/2},$ all $y_2 \geq cu_2^{1/2},$ and all $u_1 > u_2,$
  \[
    \left| {\tilde K}_o( u_1, y_1 ; u_2, y_2) \right|
    \leq \frac{u_1^{1/6}u_2^{1/6}}{u_1^{1/2} - u_2^{1/2}} \xi(y_1) \xi(y_2)
  \]
  where
  \[
    \xi(y_i) = \begin{cases}
      Ce^{C \yM[i]^{3/2}} & \yM[i] \leq 0 \text{ and } \\
      \frac{C}{1 + \yM[i]} & \yM[i] \geq 0
    \end{cases}
  \]
  for some absolute constant $C>0.$
  \label{lem:Ko_pointwise}
\end{lemma}

\begin{proof}
  Recall~\eqref{eq:hermite3}, due to which we may express ${\tilde K}_o$ as 
\begin{equation}
  {\tilde K}_o = \frac{1}{2(\pi i)^2}
  \iint
  \frac{e^{u_1\G[1]}}{e^{u_2\G[2]}} \frac{dz_1 dz_2}{z_1 - z_2}
  \label{eq:tildeK_o},
\end{equation}
with the contours given in~\eqref{eq:contours}
and $\zM[i]$ given by \eqref{eq-scaling}.
On these contours, we have that $|z_1 - z_2| \geq 2^{1/2}\left( u_1^{1/2} - u_2^{1/2} \right).$ Hence, changing the integration to be in $\zM$, we have the simple estimate
\begin{equation}
  \left| {\tilde K}_o \right| \leq
  \frac{u_1^{1/6}u_2^{1/6}}{u_1^{1/2} - u_2^{1/2}} 
  \left[ \int e^{u_1 \Re\G[1]}\,|d \zM[1]|\right]
  \left[ \int e^{-u_2 \Re\G[2]}\,|d \zM[2]|\right].
  \label{eq:tko1}
\end{equation}
Thus to complete the claimed bound, it suffices to show that each integral is bounded by $\xi$ for an appropriately large constant $C$.  Fix two parameters $\delta_1>0$ and $\delta_2>0$ to be determined.  In terms of these parameters, define the following straight-line contours in $\C,$ which are just the contours from~\eqref{eq:contours} in the $\zM[i]$ variables:
\begin{align*}
  \tilde{\gamma}_1 &= [0,\delta_1 e^{i\pi/3}u_1^{1/3}],  &  \tilde{\gamma}_1^e &= \delta_1 e^{i\pi/3}u_1^{1/3} + i\R_+, \\
  \tilde{\gamma}_2 &= [0,\delta_2 e^{2i\pi/3}u_2^{1/3}], &  \tilde{\gamma}_2^e &= \delta_2 e^{2i\pi/3}u_2^{1/3} + \R_-.
\end{align*}
By conjugate symmetry, it suffices to show that we have a bound of the form $\int_{\tilde{\gamma}_1} e^{u_1 \Re G(\zM[1],\yM[1])}\,|d \zM[1]| \leq \xi(\yM[1])/4,$ appropriately modified for all $4$ contours.

We begin with some preliminaries that will determine how to pick $\delta_1$ and $\delta_2.$
Define $F(z) = \Re\left[\log(1+z) - z  + z^2/2\right].$  From the Taylor expansion of the log, we have that $F(z) = \Re( z^3/3) + O(z^4).$  Hence there are some constants $c_0 > 0$ and $\delta_1 >0$ so that for $|z| \leq \delta_1$ and $\arg(z) = \pi/3$ we have 
\begin{equation*}
  F(z) \leq -c_0 |z|^3.
\end{equation*}
Recall from \eqref{eq:G} that $\Re\G[1] = F(u^{-1/3}\zM[1]) - \frac{\zM[1] \yM[1]}{u_1}.$  Hence applying this bound to $\G[1]$ for $\zM[1] \in \tilde{\gamma}_1$ we have
\begin{equation}
  u_1\Re \G[1] \leq -c_0 |\zM[1]|^3 - \Re \zM[1] \yM[1].
  \label{eq:localGbound}
\end{equation}

Writing $z = x+iy$, we have that $F(x+iy)$ satisfies
\[
  F(x+iy) = \frac{1}{2} \log( (1+x)^2 + y^2) -x 
  + \frac{x^2}{2}
  - \frac{y^2}{2}.
\]
Fix some $x_0>0$ and note that for all $x \geq x_0$ and $y \geq 0$ we have that
\begin{align*}
  \partial_y F(x+iy) &= \frac{y}{(1+x)^2 + y^2} - y \\
  &\leq \frac{-2x_0y - y^3}{(1+x_0)^2+y^2} \\
  &\leq -c(x_0)y
\end{align*}
for some $c(x_0) > 0.$  
Setting $\omega_1 = e^{i \pi / 3}\delta_1,$ we may integrate the previous inequality to arrive at
\begin{equation*}
  F(\omega_1 + iy) \leq -c_0 \delta_1^3 - c_1y^2, 
\end{equation*}
for $y \geq 0.$
Applying this bound to $\G[1]$ for $\zM[1] \in \tilde{\gamma}_1^e,$ which can be expressed as $\zM[1] = \omega_1 u_1^{1/3} + it$ for $t \in \R_+,$ yields
\begin{equation}
  u_1 \Re \G[1] \leq -c_0 u_1 \delta_1^3 -c_1 u_1^{1/3} t^2 -  \frac{\delta_1 u_1^{1/3}}{2} \yM[1].
\label{eq:Gbound} 
\end{equation}

Picking $\delta_2$ requires more effort, as for $\delta_2$ too small, $G(\zM[2],\yM[2])$ on $\tilde{\gamma}_2^e$ can be negative for a large range of $\zM[2]$.  
We will see that we can take $\delta_2 = 2.$  
Write
\[
  f(t) = F(e^{2\pi i/3}t) = \frac{1}{2} \log( 1 - t + t^2) + \frac{t}{2} - \frac{t^2}{4}.
\]
Then we have
\[
  f'(t) = \frac{1}{2} \frac{2t^2 - t^3}{1-t+t^2}.
\]
For $0 \leq t \leq 2,$ we can bound this below by $f'(t) \geq \frac 1 6 (2t^2 - t^3).$  Integrating, we get that $f(t) \geq \frac 1 9 t^3( 1- \frac 3 8 t) \geq \frac{t^3}{36}$ on $t \leq 2.$  Hence, we have shown that for $|z| \leq 2$ with $\arg(z) = 2\pi /3$ we have
\begin{equation}
  F(z) \geq \frac{|z|^3}{36}. \label{eq:localFbound2}
\end{equation}
Hence for $\zM[2] \in \tilde{\gamma}_2$ we have
\begin{equation}
  u_2 \Re \G[2] \geq \frac{ {|\zM[2]|}^3}{36} - \Re \zM[2] \yM[2].
  \label{eq:localGbound2}
\end{equation}

Meanwhile, for $x \leq -1$ and $y \neq 0,$
\begin{align*}
  \partial_x F(x+iy) &= \frac{1+x}{(1+x)^2 + y^2} -1 + x \\
  &\geq -1 + x. 
\end{align*}
Setting $\omega_2 = 2e^{i 2\pi / 3},$ and integrating the previous inequality in $x$ from $\omega_2,$ we get from \eqref{eq:localFbound2}
\begin{equation*}
  F(\omega_2 - x) \geq \frac{2}{9} + \frac{x^2}{2}  
\end{equation*}
for all $x \geq 0.$  Parameterize $\zM[2] \in \tilde{\gamma}_2^e$ as $\zM[2] = \omega_2u_2^{1/3} - t$ for $t\in\R_+$ so that 
\begin{equation}
  u_2 \Re \G[2] \geq \frac{2}{9} u_2 + \frac{u_2^{1/3}t^2}{2} + \frac{u_2^{1/3}}{2}\yM[2] + t\yM[2]. 
  \label{eq:Gbound2}
\end{equation}


\noindent \emph{ The contour $\tilde{\gamma}_1:$ }
We must estimate
\[
  I_1 = 
    \int_{\tilde{\gamma}_1} e^{u_1 \Re G(\zM[1],\yM[1])}\,|d \zM[1]|.
\]
We have, recalling~\eqref{eq:localGbound}, that 
\[
  u_1 \Re G(\zM[1],\yM[1]) \leq -c_0 |\zM[1]|^3 - \Re \zM[1]\yM[1]
\]
uniformly in $\zM[1] \in \tilde{\gamma}_1$ and $\yM[1] \in \R.$   Thus we have that
\[
  I_1 \leq \int_0^{\delta_1 u_1^{1/3}} e^{-\frac{\sqrt{3}}{2} t \yM[1] - c_0{t^{3}}}\,dt.
\]
When $\yM[1] \geq 1,$ we estimate this integral just by bounding $e^{-c_0{t^{3}}} \leq 1.$  For $-1 \leq \yM[1] \leq 1,$ we estimate the integral by bounding $ e^{-\frac{\sqrt{3}}{2} t \yM[1]} \leq 1.$  Combining these bounds, and adjusting the constant $C$ in $\xi$ we can assure $I_1 \leq \xi(\yM[1])/2$ for $\yM[1] \geq -1.$  

For $\yM[1] < -1,$ we write $\eta = \sqrt{-\tfrac{\sqrt{3}}{6c_0} \yM[1]}$ and let $p(t) = -3c_0\left(\tfrac{t^3}{3} - \eta^2 t\right) = -c_0 t^3 - \tfrac{\sqrt{3}}{2}\yM[1] t.$  Then we may expand $p(t)$ as
\[
  \frac{p(t)}{3c_0} = -\frac{1}{3}( t - \eta )^3 - \eta (t - \eta)^2 + \frac{2}{3} \eta^3.
\]
In particular, for $t \geq 0,$ we have that $p(t) \leq 3c_0 \eta^3 - 3c_0\eta(t- \eta)^2.$  Hence we have
  $$I_1 \leq \int_0^\infty e^{p(t)}\,dt 
  \leq e^{3c_0\eta^3} \int_{-\infty}^\infty e^{-3c_0\eta t^2 }\,dt.$$
As $\eta$ is bounded uniformly away from $0$ for $\yM[1] < -1$ we can assure $I_1 \leq \xi(\yM[1])/4$ for $\yM[1] < -1.$

\noindent \emph{ The contour $\tilde{\gamma}_2:$ }
We must estimate
\[
  I_2 = 
    \int_{\tilde{\gamma}_2} e^{-u_2 \Re G(\zM[2],\yM[2])}\,|d \zM[2]|.
\]
From \eqref{eq:localGbound2}, we get $-u_2 \Re \G[2] \leq -\frac{|\zM[2]|^3}{36} + \Re\zM[2]\yM[2]$ on $\tilde{\gamma}_2,$ and so the previous proof applies \emph{mutatis mutandis}.

\noindent \emph{ The contour $\tilde{\gamma}_1^e:$ }
We must estimate
\[
  I_1^e = 
    \int_{\tilde{\gamma}_1} e^{u_1 \Re G(\zM[1],\yM[1])}\,|d \zM[1]|.
\]
We parameterize $\zM[1] \in \tilde{\gamma}_1^e$ by writing $\zM[1] = u_1^{1/3}\omega + it,$ where we recall that $\omega = e^{i \pi /3} \delta_1.$  
From \eqref{eq:Gbound}, 
\[
 u_1 \Re \G[1] \leq -c_0 u_1 \delta_1^3 -c_1 u_1^{1/3} t^2 -  \frac{\delta_1 u_1^{1/3}}{2} \yM[1]
\]
Under the assumption that $\yM[1] \geq -2c_0 u_1^{2/3} \delta_1^2,$ we can pick $C>0$ in $\xi$ sufficiently large that
\[
  I_1^e \leq \exp\left(-c_0 u_1 \delta_1^3 - \frac{\delta_1  u_1^{1/3}}{2} \yM[1]\right)\cdot\int_{0}^\infty e^{-c_1 u_1^{1/3}t^2}\,dt \leq \xi( \yM[1])/4.
\]

\noindent \emph{ The contour $\tilde{\gamma}_2^e:$ }
We must estimate
\[
  I_2^e = 
    \int_{\tilde{\gamma}_2} e^{-u_2 \Re G(\zM[2],\yM[2])}\,|d \zM[2]|.
\]
Using \eqref{eq:Gbound2},
\[
  u_2 \Re \G[2] \geq \frac{2}{9} u_2 + \frac{u_2^{1/3}t^2}{2} - \Re\zM[2]\yM[2]. 
\]
an analogous estimate to that done for $I_1^e$ holds.
\end{proof}

\section{Uniform boundedness of $\tilde K$ for all $u_2 - u_1 \gg u_1^{2/3}$}
\label{sec:boundedness}

We additionally need quantitative bounds for the suprema of ${\tilde K}_o$ and ${\tilde K}_e$ to estimate the difference of determinants.  
For ${\tilde K}_e,$ we have
\begin{lemma}
  Let $\Delta = 4(\sqrt{u_2}-\sqrt{u_1})\sum_{i} (y_i - \sqrt{2u_i})_{-}$ for $i=1,2.$  There is an absolute constant $C>0$ so that
  \[
\left|{\tilde K}_e( u_1, y_1 ; u_2, y_2) \right|
\leq
C\sqrt{u_2}
\exp\left( -\frac{ (\sqrt{u_2} - \sqrt{u_1})^3}{C\sqrt{u_2}} + \Delta \right).
  \]
  \label{lem:Ke_sup}
\end{lemma}
\begin{proof}

  We will proceed by producing bounds for $\tilde{ K }_e$ in terms of $\tau,$
 which we recall is the point of intersection of $\gamma_1^c$ and $\gamma_2^c.$  This location is not completely explicit, as it depends on $\delta_1$ and $\delta_2,$ chosen in Lemma~\ref{lem:Ko_pointwise}.  However, as we chose $\delta_2=2,$ which implies that $\gamma_2$ runs from the real axis to the imaginary axis, we have that $\tau$ is a point on $\gamma_2.$  If $\gamma_1$ and $\gamma_2$ intersect, they do so at the point
 \begin{equation}
   \tau_{0.5} = \frac{\sqrt{u_1} + \sqrt{u_2}}{2\sqrt{2}} + i \frac{\sqrt{3}}{2\sqrt{2}}({\sqrt{u_2} - \sqrt{u_1}})
   \label{eq:tauhalf}
 \end{equation}
 Otherwise $\gamma_1^c$ and $\gamma_2^c$
 intersect at some point on $\gamma_2$ with real part at least $\sqrt{u_1/2},$ and hence they intersect at 
 \begin{equation}
   \tau_\alpha = \frac{\sqrt{u_1} + \alpha(\sqrt{u_2} - \sqrt{u_1})}{\sqrt{2}} +
i \frac{(1-\alpha)\sqrt{3}}{\sqrt{2}}({\sqrt{u_2} - \sqrt{u_1}}),
   \label{eq:taua}
 \end{equation}
 for some $\alpha$ in $[0,0.5].$  

  We begin with the case that $y_1 \leq y_2,$ for which
  \[
    {\tilde K}_e( u_1, y_1 ; u_2, y_2)
    =
    -\frac{J(u_2,y_2)}{J(u_1,y_1)}
    \frac{1}{\pi i} \int_{\gamma_{+}^r} \frac{e^{2z_2(y_2 - y_1)}}{(2z_2)^{u_2 - u_1}}\, d z_2.
  \]
  As we assume that $u_2 - u_1 \geq 2,$ we have that
\begin{equation}
\left|{\tilde K}_e( u_1, y_1 ; u_2, y_2) \right|
\leq 
    \frac{J(u_2,y_2)}{J(u_1,y_1)}
    e^{2 \Re \tau (y_2 - y_1)}
    2^{u_1 - u_2}
    \int_0^\infty |(\tau + it)|^{u_1-u_2}\,dt.
    \label{eq:ke1}
  \end{equation}
  For $t \in [0,|\tau|],$ we bound this integral by taking the supremum.  For $t \geq |\tau|,$ we observe that $|\tau + it| \geq \sqrt{2}|\tau| + \tfrac{1}{\sqrt{2}}(t - |\tau|).$  Integrating, we conclude that
  \begin{align*}
    \int_0^\infty |(\tau + it)|^{u_1-u_2}\,dt.
    &\leq |\tau|^{u_1 -u_2+1}
    + \sqrt{2} \int_0^\infty (\sqrt{2}|\tau| + t)^{u_1 - u_2}\,dt \\
    &\leq 2 |\tau|^{u_1 - u_2 + 1}.
  \end{align*}
  Applying this bound to \eqref{eq:ke1} and using the definition of $J(u,y),$ we can bound
\begin{equation}
\left|{\tilde K}_e( u_1, y_1 ; u_2, y_2) \right|
\leq
2|\tau|
\left(\frac{u_2}{u_1}\right)^{\tfrac{u_1}{2}}
\left(\frac{2|\tau|^2}{eu_2}\right)^{\tfrac{u_1 - u_2}{2}}
e^{2 \Re \tau (y_2 - y_1) - \sqrt{2u_2}y_2 + \sqrt{2u_1}y_1}.
    \label{eq:ke2}
  \end{equation}

  We will now begin the process of substituting $\tau_\alpha$ for $\tau$ and maximizing over $\alpha.$
  Both $y_1$ is not much less than $\sqrt{2u_1},$ and $y_2$ is not much less than $\sqrt{2u_2}.$  Recall that $\Delta = 4(\sqrt{u_2}-\sqrt{u_1})\sum_{i} (y_i - \sqrt{2u_i})_{-}$ for $i=1,2,$ we have the bound that for all $\alpha \in [0,0.5],$
  \begin{equation}
    2 \Re \tau_{\alpha} (y_2 - y_1) - \sqrt{2u_2}y_2 + \sqrt{2u_1}y_1
\leq
-2(\sqrt{u_2} - \sqrt{u_1}) \sqrt{u_2} + 2\alpha (\sqrt{u_2} - \sqrt{u_1})^2 +\Delta.
    \label{eq:ke3}
  \end{equation}

  Define $N(\alpha,t)$ and $H(\alpha,t)$ by
  \begin{align}
    \nonumber
    N(\alpha,t)
    &=
    \left( \alpha + \frac{1-\alpha}{1+t}\right)^2 + 3(1-\alpha)^2\left( 1-\frac{1}{1+t} \right)^2 \\
    H(\alpha,t)
    &=\log(1+t) - \frac{2t+t^2}{2}\log(N(\alpha,t)) - t - \left(\frac{3}{2} -2\alpha\right)t^2.
    \label{eq:ke4}
  \end{align}
  Setting $1+t = \sqrt{u_2/u_1},$ we see that $N$ satisfies
 \begin{align}
    \nonumber
    u_2N(\alpha, \sqrt{u_2/u_1} - 1)
    &=
    \left( \alpha\sqrt{u_2} + (1-\alpha)\sqrt{u_1}\right)^2 + 3(1-\alpha)^2\left( \sqrt{u_2}-\sqrt{u_1} \right)^2 \\
    \nonumber
    &= 2 |\tau_\alpha|^2.
  \end{align}
  Thus combining this bound with \eqref{eq:ke2}, \eqref{eq:ke3}, and \eqref{eq:ke4}, we have
  \begin{equation}
\left|{\tilde K}_e( u_1, y_1 ; u_2, y_2) \right|
\leq
3\sqrt{u_2}
\max_{\alpha \in [0,0.5]}
\exp\left( u_1 H(\alpha, \sqrt{u_2/u_1} - 1) + \Delta \right).
    \label{eq:ke5}
  \end{equation}
  We will see that this bound is monotone increasing in $\alpha$ for $\alpha \in [0,0.5].$  Taking derivatives, we see that
  \begin{align*}
    \partial_\alpha H(\alpha,t) 
    &= 2t^2 - \frac{2t + t^2}{2} \partial_\alpha\left( \log\left( (1+t)^2N(\alpha,t) \right) \right) \\
    &= 2t^2 - (2t + t^2) \frac{t(1+\alpha t)-3(1-\alpha)t^2}{(1+t)^2N(\alpha,t)}\\
    &= \frac{(8\alpha^2 - 16\alpha + 9)t^4+(5-4\alpha)t^3}{(1+t)^2N(\alpha,t)}.
  \end{align*}
  This is positive for all $(\alpha,t) \in [0,0.5]\times[0,\infty),$ and hence we may take $\alpha = 0.5$ in \eqref{eq:ke5}.


    Evaluating $N$ and $H$ at $\alpha=0.5,$ we get that
  \begin{align}
    \nonumber
    N(0.5,t)
    &=
    \left( \frac{1+0.5t}{1+t}\right)^2 + \frac{3}{4}\left( 1-\frac{1}{1+t} \right)^2 = 1 - \frac{t}{(1+t)^2}\\
    H(0.5,t)
    &=\log(1+t) - \frac{2t+t^2}{2}\left( \log(N(0.5,t)) + 1 \right).
    \label{eq:ke1000}
  \end{align}
  We will proceed to bound $H(0.5,t)$ from above using the inequality $\log(1+x) \leq x - \frac{x^2}{2(1+x)}$ valid for all $x \geq 0:$
  \begin{align*}
    H(0.5,t)
    & \leq t - \frac{t^2}{2(1+t)} - t - \frac{t^2}{2} + \frac{2t+t^2}{2}\log\left(1 +  \frac{t}{1+t+t^2}\right) \\
    & \leq - \frac{t^2}{2(1+t)} - \frac{t^2}{2} + \frac{2t+t^2}{2}\frac{t}{1+t+t^2} - \frac{2t+t^2}{4}\frac{t^2}{(1+t)^2(1+t+t^2)} \\
    & = \frac{-2t^3 - 5t^4 - 6t^5 - 2t^6}{4(1+t)^2(1+t+t)^2}\\
    & \leq -c_0 t^3/(1+t),
  \end{align*}
  for some sufficiently small constant $c_0 > 0.$  Applying this inequality to \eqref{eq:ke5}, we have that 
  \begin{equation}
\left|{\tilde K}_e( u_1, y_1 ; u_2, y_2) \right|
\leq
3\sqrt{u_2}
\exp\left( -c_0\frac{ (\sqrt{u_2} - \sqrt{u_1})^3}{\sqrt{u_2}} + \Delta \right),
\label{eq:ke6}
  \end{equation}
  which is the desired bound.

  There still remains to handle the case that $y_1 > y_2.$  We recall that in this case we have by \eqref{eq:Ke}
  \[
    {\tilde K}_e( u_1, y_1 ; u_2, y_2)
    =
    -\frac{J(u_2,y_2)}{J(u_1,y_1)}
    \frac{1}{\pi i} \int_{\gamma_{-}^r} \frac{e^{2z_2(y_2 - y_1)}}{(2z_2)^{u_2 - u_1}}\, d z_2,
  \]
  where we recall that $\gamma_{-}^r$ is the contour that follows $\gamma_2^c$ from $\overline{\tau}$ to $\tau.$  As the integrand is integrable in the right half-plane, we may replace this by three sides of a large rectangle whose top and bottom sides are on the lines $\Im z = \Im \tau$ and $\Im z = -\Im \tau.$  As $u_2 > u_1,$ and $y_1>y_2,$ this integral is convergent and we get the representation
  \begin{equation}
     {\tilde K}_e( u_1, y_1 ; u_2, y_2)
    =
    -\frac{J(u_2,y_2)}{J(u_1,y_1)}
    \frac{2}{\pi} \int_{\tau + \R_+} \Im \frac{e^{2z_2(y_2 - y_1)}}{(2z_2)^{u_2 - u_1}}\, d z_2.
    \label{eq:Kf1}
  \end{equation}
  We can now bound in the same way that we bounded ${\tilde K}_e$ when $y_1 \leq y_2,$ i.e.
\begin{equation}
\left|{\tilde K}_e( u_1, y_1 ; u_2, y_2) \right|
\leq
C|\tau|
\left(\frac{u_2}{u_1}\right)^{\tfrac{u_1}{2}}
\left(\frac{2|\tau|^2}{eu_2}\right)^{\tfrac{u_1 - u_2}{2}}
e^{2 \Re \tau (y_2 - y_1) - \sqrt{2u_2}y_2 + \sqrt{2u_1}y_1},
    \label{eq:kf2}
  \end{equation}
  for some absolute constant $C>0.$  Hence, we again get \eqref{eq:ke6} with some other constant in front, and the proof of the lemma is complete. 
\end{proof}

For the supremum of ${\tilde K}_o$ with $u_1 \leq u_2:$
\begin{lemma}
  Suppose that $\yM[i] \geq -c u_i^{1/3}$ for $i=1,2.$ Let $\xi_+(x) = 1/(1+(x)_+),$ and let $\mu(\yM[1],\yM[2]) = \max(\sqrt{(\yM[1])_{-}},\sqrt{(\yM[2])_{-}}).$ 
  There are absolute constants $c, M, T > 0$ so that the following hold. 
  \begin{enumerate}
    \item If $u_2 \geq M u_1,$ then there is an absolute constant $C>0$ so that 
      \[
        \left| {\tilde K}_o( u_1, y_1 ; u_2, y_2) \right|
      \leq C.
    \]
  \item If $u_2 < Mu_1$ and $\sqrt{u_2} - \sqrt{u_1} \geq Tu_2^{1/6}\mu(\yM[1],\yM[2]),$ then there is an absolute constant $C>0$ so that
  \[
    \left| {\tilde K}_o( u_1, y_1 ; u_2, y_2) \right|
    \leq C \frac{u_1^{1/6}u_2^{1/6}}{u_1^{1/6} + u_2^{1/6}}
    \sqrt{ \xi_+( \tilde{y}_1)
    \xi_+( \tilde{y}_2)}.
  \]
    \item If $u_1=u_2,$ then there is an absolute constant $C>0$ so that
      \[
    \left| {\tilde K}_o( u_1, y_1 ; u_2, y_2) \right|
    \leq Cu_1^{1/6}(
    \sqrt{ \xi_+( \tilde{y}_1)
    \xi_+( \tilde{y}_2)}
    +\mu(\yM[1],\yM[2])
    e^{c(\eta_2 - \eta_1)(\tilde{y}_2 - \tilde{y}_1) }
    ).
      \]
  \end{enumerate}
   \label{lem:Ko_sup}
\end{lemma}
\begin{proof}

  The contours $\gamma_i^c$ are insufficient for this task, as 
  when $\yM[i] < 0$ the contours $\gamma_i^e$ become poor approximations of the true steepest descent contours.  These errors occur in a $\zM[i]$-neighborhood of $0$ of magnitude $O(\sqrt{-\yM[i]}),$ which we can fix by a simple local contour deformation.  
  
  From \eqref{eq:taylor}, we have that
  \[
  u_i \G[i] = -\zM[i]\yM[i] + \frac{1}{3}\zM[i]^3 + O(u_i^{-1/3}\zM[i]^4).
  \]
  Fix $\lambda$ with $\sqrt{3} > \lambda \geq \tfrac{1}{\sqrt{3}},$ a constant to be determined later.
  Set $\tilde{\sigma}_1$ to be the line segment of $\Im \zM[1] = \lambda \Re \zM[1] + \eta_1$ with $\eta_i = \sqrt{(\yM[i])_{-}}$ for $i=1,2$ that connects the real axis to the line through $\tilde{\gamma}_1.$  The point of intersection with the line through $\tilde{\gamma}_1$ occurs at distance $\Theta(\eta_1).$  Hence, on this line segment $O(u_1^{-1/3}\zM[1]^4) = O(1)$ by assumption on $\yM[i],$ and we have uniformly in $\lambda \geq \frac{1}{\sqrt{3}}:$
  \begin{equation}
    \Re u_1 \G[1] \leq \frac{-2\eta_1 u_1^{1/3}}{\sqrt 3} (\Re \zM[1])^2 + C
    \label{eq:ulocalG1bound}
  \end{equation}
  for some absolute constant $C>0$ and all $\zM[1] \in \tilde{\sigma}_1.$
  Likewise, we define $\tilde{\sigma}_2$ to be the line segment of $\Im \zM[2] = -\lambda \Re \zM[2] + \eta_2.$  Doing a similar Taylor expansion, we can see that
\begin{equation}
  -\Re u_2 \G[2] \leq \frac{-2\eta_2 u_2^{1/3}}{\sqrt 3} (\Re \zM[2])^2 + C
  \label{eq:ulocalG2bound}
\end{equation}
for some absolute constant $C>0,$ all $\zM[2] \in \tilde{\sigma}_2$ and all $\lambda \geq \frac{1}{\sqrt{3}}.$  

Define $\sigma_1$ and $\sigma_2$ to be the images of $\tilde{\sigma}_1$ and $\tilde{\sigma}_2$ under the changes of variables $\zM[1] \mapsto z_1$ and $\zM[2] \mapsto z_2.$  
The intersections of $\sigma_i$ and the line through $\gamma_i$ occur at distance $O(u_i^{1/6}\eta_i) = O(u_i^{1/3}).$  Thus by taking $c>0$ sufficiently small, we can assure that $\sigma_i$ and $\gamma_i$ intersect.  Let $\sigma_i^m$ be the portion of $\gamma_i$ between its intersection with $\sigma_1$ and $\gamma_i^e,$ and let $\sigma_i^e=\gamma_i^e.$  Finally, define
\begin{align*}
  \sigma_1^c &= \overline{\sigma_1^e} \cup \overline{\sigma_1^m} \cup \overline{\sigma_1} \cup \sigma_1\cup \sigma_1^m \cup \sigma_1^e, \\
  \sigma_2^c &= \overline{\sigma_2^e} \cup \overline{\sigma_2^m} \cup \overline{\sigma_2} \cup \sigma_2\cup \sigma_2^m \cup \sigma_2^e, 
\end{align*}
oriented in the same way as $\gamma_i^c.$
For notational convenience, when for either $i=1,2$ $\yM[i] \geq 0,$ we let $\sigma_i = \emptyset,$ $\sigma_i^m = \gamma_i$ and $\sigma_i^e = \gamma_i^e.$  

Define 
\begin{equation}
  {\tilde K}_\sigma = \frac{1}{2(\pi i)^2}
  \int_{\sigma_2^c}\int_{\sigma_1^c}
  \frac{e^{u_1\G[1]}}{e^{u_2\G[2]}} \frac{dz_1 dz_2}{z_1 - z_2},
  \label{eq:sigmaK_o}
\end{equation}
and set $\RES = {\tilde K}_o - {\tilde K}_\sigma.$

Fix an $\epsilon > 0$ and define 
  \begin{equation}
    \psi_i(\zM[i],\yM[i]) = \begin{cases} 
      \exp(-\epsilon \eta_i u_i^{1/3}(\Re  \zM[i])^2) &\text{ if } \zM[i] \in \sigma_i \cup \overline{\sigma_i}, \\
      \exp(-\epsilon(|\zM[i]|^3 -  |\zM[i]|(\yM[i])_+ )) & \text{ if }\zM[i] \in \sigma_i^m \cup \overline{\sigma_i^m}, \\
      \exp(-\epsilon u_i^{1/3}(|\zM[i]|^2 + (\yM[i])_+)) & \text{ if } \zM[i] \in \sigma_i^e \cup \overline{\sigma_i^e},
  \end{cases}
    \label{eq:psi}
  \end{equation}
  for $i=1,2.$
  Using \eqref{eq:localGbound} and \eqref{eq:localGbound2} on $\sigma_i^m,$
  \eqref{eq:Gbound} and \eqref{eq:Gbound2} on $\sigma_i^e,$ and
  \eqref{eq:ulocalG1bound} and \eqref{eq:ulocalG2bound} on $\sigma_i,$
  we get that 
\[
  \sup_{{\sigma}_2^c \times {\sigma}_1^c} \biggl\{ \left|
\frac{e^{u_1 G(\zM[1],\yM[1])}}{e^{u_2 G(\zM[2],\yM[2])}} 
\right|\frac{1}{\psi_1(\zM[1],\yM[1])\psi_2(\zM[2],\yM[2])}
\biggr\}  \leq C,
\]
provided $\epsilon >0$ is chosen sufficiently small, $C>0$ is chosen sufficiently large.

Recall that $\tau$ is always at least distance $\Omega( (\sqrt{u_2} - \sqrt{u_1}))$ in $z_i$ coordinates from either of $\sqrt{u_i/2}$ (see \eqref{eq:tauhalf}), and the intersection of $\sigma_i$ and $\gamma_i$ occurs at distance $O(u_i^{1/6}\eta_i).$ 
Hence in the case $\sqrt{u_2} - \sqrt{u_1} \geq Tu_2^{1/6}\cdot \max(\sqrt{(\yM[1])_{-}},\sqrt{(\yM[2])_{-}}),$ we may choose $T>0$ sufficiently large so that the intersections of $\sigma_i$ and $\gamma_i$ occur at points of smaller imaginary part than $\tau.$  In this case, we can perform the deformation from $\gamma_i^c$ to $\sigma_i^c$ without producing any additional residues, so $\RES = 0.$

In the case that $u_1 = u_2,$ we may acquire a residue, which will be given by
\begin{equation}
  \RES(u_1,y_1 ; u_1, y_2) =
    -\frac{J(u_1,y_2)}{J(u_1,y_1)}
  \frac{1}{\pi i} \int_{\ell}{e^{2z_2(y_2 - y_1)}}\, d z_2,
  \label{eq:RES}
\end{equation}
where, after deformation, $\ell$ is a vertical line segment connecting the intersection of $\overline{\sigma_1 \cup \sigma_1^m}$ and $\overline{\sigma_2 \cup \sigma_2^m}$ to the intersection of $\sigma_1 \cup \sigma_1^m$ and $\sigma_2 \cup \sigma_2^m.$  Denote the point of intersection between $\sigma_1 \cup \sigma_1^m$ and $\sigma_2 \cup \sigma_2^m$ by $\zeta.$  If the intersection is between $\sigma_1^m$ and $\sigma_2^m,$ it necessarily occurs at $\zeta = \sqrt{u_1/2},$ in which case there is no residue. Note that each of these contours cross the line $\Re z_1 = \sqrt{u_1/2}$ at $i\eta_1u_1^{1/6}/\sqrt{2}$ and $i\eta_2u_1^{1/6}/\sqrt{2}$ respectively, as it must be $\sigma_1$ and $\sigma_2$ that cross this vertical line.  In particular, we have 
\(
\sgn( \Re \zeta - \sqrt{u_1/2}) = -\sgn(y_2 - y_1).
\)
Further, by this observation, we must have 
\(
\Im \zeta \leq \max(\eta_1, \eta_2)u_1^{1/6}/\sqrt{2}.
\)

With these estimates, we turn to bounding $\RES.$ By a supremum bound of the integrand of \eqref{eq:RES}, we have
\begin{equation*}
  |\RES(u_1,y_1 ; u_1, y_2)| \leq 
  \frac{2}{\pi}\Im \zeta \cdot e^{2(\Re \zeta - \sqrt{u_1/2})(y_2-y_1)}. 
\end{equation*}
Let $\tilde \zeta$ be the position  $\zeta$ in $\zM[1]$ coordinates, so that
\begin{equation*}
  |\RES(u_1,y_1 ; u_1, y_2)| \leq 
  \frac{2}{\pi}\Im \zeta \cdot e^{(\Re {\tilde \zeta})(\yM[2]-\yM[1])}. 
\end{equation*}
There are three possibilities for the location of $\tilde \zeta,$ at the intersection of
\(
\sigma_1 \cap \sigma_2,
\)
\(
\sigma_1^m \cap \sigma_2,
\)
or
\(
\sigma_1 \cap \sigma_2^m.
\)
In each of these cases, we get that $\Re {\tilde \zeta}$ is, respectively, the first, second or third entry of
\[
	\left( 
	\frac{\eta_2 - \eta_1}{2\lambda},
	\frac{\eta_2}{\sqrt{3}+\lambda},
	\frac{-\eta_1}{\sqrt{3}+\lambda}
	\right).
\]
If $\left\{ \zeta \right\} = \sigma_1^m \cap \sigma_2,$ then we have that $\Re \zeta > 0.$  In particular, it must be that $\sigma_1 \cup \sigma_1^m$ crosses $\Re \zm[1] = \sqrt{u_1/2}$ below $\sigma_2 \cup \sigma_2^m,$ and hence $\eta_2 > \eta_1.$  Hence if $\eta_2 > 0,$ we must have $\ym[2] < \ym[1].$  Thus in this case we conclude, for any values of $\yM[i],$ that
\[
	\Re {\tilde \zeta}(\yM[2]-\yM[1]) 
	= \frac{\eta_2}{\sqrt{3}+\lambda}(\yM[2]-\yM[1])
	\leq \frac{\eta_2-\eta_1}{\sqrt{3}+\lambda}(\yM[2]-\yM[1])
\]
The same conclusion holds if instead $\left\{ \zeta \right\} = \sigma_1 \cap \sigma_2^m.$
As $\sgn(\tilde \zeta(\yM[2]-\yM[1])) \leq 0,$ we can therefore bound, in all three cases,
\[
\Re {\tilde \zeta}(\yM[2]-\yM[1]) \leq \frac{\eta_2-\eta_1}{\sqrt{3}+\lambda}(\yM[2]-\yM[1])
\]
Hence we reach the conclusion
\begin{equation}
  |\RES(u_1,y_1 ; u_1, y_2)| \leq u_1^{1/6}\max(\eta_1, \eta_2)
    e^{c(\eta_2 - \eta_1)(\tilde{y}_2 - \tilde{y}_1) }
  \label{eq:RES3}
\end{equation}
with $c = (\sqrt{3} + \lambda)^{-1}.$

  From the definition of $\psi,$ we have
\begin{equation}
  \left| {\tilde K}_\sigma(u_1,y_1;u_2,y_2) \right| \leq
  C\int\limits_{\sigma_2^c \times \sigma_1^c} \left|\frac{ \psi_1(\zM[1],\yM[1])\psi_2(\zM[2],\yM[2]) dz_1 dz_2}{z_1 - z_2}\right|.
  \label{eq:kos1}
\end{equation}
We will begin by showing that there is a $C>0$ so that
\begin{equation}
\int\limits_{\sigma_2^c \times \sigma_1^c} \left|\frac{ \psi_1(\zM[1],\yM[1]) \psi_2(\zM[2],\yM[2]) dz_1 dz_2}{z_1 - z_2}\right|
< C(u_1^{1/6} + u_2^{1/6})
\max(
\xi_+(\yM[1]),
\xi_+(\yM[2]))
\label{eq:kos2}
\end{equation}
for all $u_1,u_2.$  This combined with \eqref{eq:RES3} will complete the $u_1 = u_2$ part of the proof. 

Let $\mu(x)$ be an absolutely continuous finite measure on $\R$ with connected support and with density at most $1.$  The following bound holds for all $c, y>0$ and all such $\mu:$
\begin{align}
  \int_\R \frac{d\mu(x)}{|cx-z|} &\leq \inf_{R >0} \left\{ \frac{8}{c}\log\left( 1+\frac{cR}{d(z,\supp(\mu))} \right) + \frac{2\mu(\R)}{Rc + d(z,\supp(\mu))} \right\}. 
  \label{eq:ib1}
\end{align}
This can be checked by letting $x_0 \in \supp(\mu)$ achieve the minimum distance to $z$ and dividing the integral into an interval around $x_0$ of radius $R$ and the rest of $\R.$  

Both of $\int_{\sigma_i^c} |\psi(\zM[i])\,d\zM[i]| \leq C\xi_+(\yM[i])$ for $i=1,2$ are bounded above by some universal constant $C>0$.  For $\tilde{\sigma}_i^m$ and $\tilde{\sigma}_i^e$ this is clear.  For $\tilde{\sigma}_i,$ the $\eta_i$ in the exponent in $\psi_i$ may cause worry, but the length of the segment is only $O(\eta_i),$ from which one can show that the contribution of this segment to the integral is at most $O(1/u_i^{1/12}).$

Hence for $z_2 \notin \sigma_1^c,$ we can apply \eqref{eq:ib1} to each of the $6$ straight segments of $\sigma_1^c$ to get
\begin{align}
\int\limits_{\sigma_1^c} \left|\frac{ \psi_1(\zM[1],\yM[1]) dz_1}{z_1 - z_2}\right|
&=
\int\limits_{\sigma_1^c} \left|\frac{ \psi_1(\zM[1],\yM[1]) 2^{-1/2}u_1^{1/6} d\zM[1]}{2^{-1/2}u_1^{1/6}\zM[1] + \sqrt{u_1/2} - z_2}\right| \nonumber \\
&\leq 
C \left[\log\left( 1+\frac{Ru_1^{1/6}}{d(z_2,\sigma_1^c)} \right) + \frac{u_1^{1/6}\xi_+(\yM[1])}{Ru_1^{1/6} + d(z_2,\sigma_1^c)} \right] \nonumber \\ 
&\leq 
C \left[\log\left( 1+\frac{Ru_1^{1/6}}{d(z_2,\sigma_1^c)} \right) + \frac{\xi_+(\yM[1])}{R}\right] 
  \label{eq:kos3}
\end{align}
for some absolute constant $C>0$ and any $R>0.$

Under the same assumptions as in \eqref{eq:ib1}, we also have
\begin{equation}
 \int_\R \log\left( 1+ \frac{c}{|x|} \right)\,d\mu(x) \leq (c + \mu(\R)) \log 4.
  \label{eq:ib2}
\end{equation}
We apply this to the integral of \eqref{eq:kos3} over $\sigma_2^c.$  We show the bound explicitly for $\sigma_2^m;$ analogous bounds hold for the other segments.  Set $\zeta \in \sigma_2^m$ to be the point that achieves the minimum distance $\min_{z_2 \in \sigma_2^m} d(z_2,\sigma_1^c).$  This point is unique and we have that $d(z_2, \sigma_1^c) \geq c_0 d(z_2,\zeta)$ for some $c_0 > 0$ and all $z_2 \in \sigma_2^m.$  
Let $\tilde \zeta$ be the image of $\zeta$ under the change of variables $z_2 \mapsto \zM[2]$
Hence, changing variables and applying \eqref{eq:ib2}, there is an absolute constant $C>0$ so that 
\begin{align}
  \int_{\sigma_2^m} \psi(\zM[2]) \log\left( 1+\frac{Ru_1^{1/6}}{d(z_2,\sigma_1^c)} \right)\,|dz_2|
  &\leq
  \int_{{\tilde \sigma}_2^m} \psi(\zM[2]) \frac{u_2^{1/6}}{2^{1/2}}
  \log\left( 1+\frac{u_1^{1/6}2^{1/2}R}{c_0u_2^{1/6} |\zM[1] - \tilde{\zeta}|}
\right)\,|d\zM[2]| \nonumber \\
&\leq Cu_2^{1/6}\left( \xi_+(\yM[2]) + \frac{Ru_1^{1/6}}{u_2^{1/6}} \right). 
\label{eq:kos4}
\end{align}

Combining this with \eqref{eq:kos3} we have
\begin{equation}
\int\limits_{\sigma_2^c \times \sigma_1^c} \left|\frac{ \psi_1(\zM[1],\yM[1]) \psi_2(\zM[2],\yM[2]) dz_1 dz_2}{z_1 - z_2}\right|
< C\left( 
u_2^{1/6} \xi_+(\yM[2])
+ Ru_1^{1/6}
+\frac{\xi_+(\yM[1])\xi_+(\yM[2])}{R}
\right).
\label{eq:kos5}
\end{equation}
Taking $R = \sqrt{ \xi_+(\yM[1])\xi_+(\yM[2])},$ and noting that we could run the same argument by integrating over $\zM[2]$ first, we find this is equivalent to what we set out to show in \eqref{eq:kos2}.  This completes cases (2) and (3), as this for any $M>1,$ this bound is equivalent to the stated one in case (2) after adjusting constants.  

Finally, we turn to case (1), in whose proof we will determine $M.$
Let $V = \left\{ \zM[1] : Mu_1^{1/3} \leq \zM[1] \right\}.$  By making $M$ sufficiently large, we can assure that for all $u_2 \geq Mu_1 \geq u_0$ for some large $u_0:$
\begin{enumerate}
  \item $V \cap \sigma_1^c = V \cap (\sigma_1^e \cup \overline{\sigma_1^e}).$
  \item $\tau$ is the intersection of $\sigma_1^e$ and $\sigma_2^m.$  
  \item $2M u_1^{1/3} < |\tau|.$
\end{enumerate}

It follows that for any $(z_1,z_2) \in Q = (\sigma_1^c \times \sigma_2^c) \setminus (V \times (\sigma_2^m \cup \overline{\sigma_2^m})),$ we have that there is some $c_0(M)$ so that $|z_1 - z_2| \geq c_0(M)u_2^{1/2}.$
Hence for $z \in Q,$ we have
\begin{equation}
\int\limits_{Q} \left|\frac{ \psi_1(\zM[1],\yM[1]) \psi_2(\zM[2],\yM[2]) dz_1 dz_2}{z_1 - z_2}\right|
\leq \frac{u_1^{1/6} u_2^{1/6}}{2c_0(M)u_2^{1/2}},
\label{eq:kosQ}
\end{equation}
which is negligible.  Meanwhile, on either of $V \cap \sigma_1^c$ or $\sigma_2$, we have that $|\zM[i]| = \Omega(u_i^{1/3}),$ for $i=1,2.$  Hence
\begin{align}
  \int_{V \cap \sigma_1^e} |\psi_1(\zM[1],\yM[1])\,d\zM[1]| &\leq \frac{C}{u_1^{1/6}}, \nonumber \\
  \int_{\sigma_2^m} |\psi_2(\zM[2],\yM[2])\,d\zM[2]| &\leq \frac{C}{u_2^{1/6}},
  \nonumber 
\end{align}
for some absolute constant $C>0.$  
Hence for $z_2 \neq \tau,$ setting $R = \sqrt{2}u_1^{-1/6}$ in \eqref{eq:ib1} implies that
\begin{align}
  \int\limits_{V \cap \sigma_1^e} \left|\frac{ \psi_1(\zM[1],\yM[1]) dz_1}{z_1 - z_2}\right|
&=
\int\limits_{V \cap \sigma_1^e} \left|\frac{ \psi_1(\zM[1],\yM[1]) 2^{-1/2}u_1^{1/6} d\zM[1]}{2^{-1/2}u_1^{1/6}\zM[1] + \sqrt{u_1/2} - z_2}\right| \nonumber \\
&\leq 
C \left[\log\left( 1+\frac{1}{d(z_2,\sigma_1^e)} \right) + \frac{1}{1 + d(z_2,\sigma_1^e)} \right] \nonumber \\ 
&\leq 
C \left[\log\left( 1+\frac{1}{d(z_2,\sigma_1^e)} \right) + 1 \right] 
  \nonumber 
\end{align}
for some absolute constant $C>0.$  Thus by \eqref{eq:ib2}, we have that
\[
  \int\limits_{V \cap (\sigma_1^e \times \sigma_2^m)} 
  \left|\frac{ \psi_1(\zM[1],\yM[1]) \psi_2(\zM[2],\yM[2])\, 
  dz_1dz_2}{z_1 - z_2}\right| \leq C
\]
for some absolute constant $C >0.$  Combining this with \eqref{eq:kos5}, we have the desired bound.



\end{proof}

\section{Decorrelation estimate proofs}
\label{sec:decorproof}
In what follows, we set 
    \[
      I_M =\bigcup_{i=1,2} 
  \left\{ u_i \right\} \times 
  \left(\sqrt{2u_i} + u_i^{-1/6}[t_i/\sqrt{2}, (\log u_1 )^{100}]\right).
\]
We also define 
\[
  E_M(u_1,t_1; u_2, t_2)
=
  \left|
  \det( I - \tilde{K}\vert_{I_M})
  -\det( I - {\tilde K}^D\vert_{I_M})
  \right|
\]
By Lemma~\ref{lem:levuppertail}, for all $t_i \leq (\log u_1)^{100}$
\begin{equation}
  \left|
  E(u_1,t_1; u_2, t_2)
  -
  E_M(u_1,t_1; u_2, t_2)
  \right|
  \leq 2Ce^{-\log(u_1)^{150}/C}.
  \label{eq:EvsEM}
\end{equation}
This is smaller than the bounds we wish to show for $E$, and hence it suffices to show the bounds for $E_M.$  

 For trace class kernels $K,L$ on $L^2(I),$ recall that that the $2$-regularized determinant $\det_2(I - K ) = \det( I - K)e^{-\tr K}.$  These determinants satisfy the following perturbation bound:
  \begin{equation}
    \left|
    \det_2( I - K ) 
    -\det_2( I - L ) 
    \right|
    \leq
    \HSN{K - L}\exp \left( \frac{1}{2}( 1 + \HSN{K} + \HSN{L})^2 \right),
    \label{eq:det2lip}
  \end{equation}
  see~\cite[p. 196]{GGK}.  

  To apply \eqref{eq:det2lip}, we begin by estimating the Hilbert-Schmidt norm of $\tilde{K}^D \vert_{I_M}.$

\begin{lemma}
	Provided that $u_1 \geq u_2 + u_2^{2/3} e^{(\log u_1)^{2/3}},$
	then uniformly in $t_i \geq -(\log u_i)^{5/12},$
	\[
\HSN{ \tilde{K}_e \vert_{I_M} }^2 = 
O(e^{ -\Omega( \exp((\log u_1)^{2/3}) }).
	\]
	\label{lem:KeHS}
\end{lemma}
\begin{proof}
For ${\tilde K}_e$ we have, in the notation of Lemma~\ref{lem:Ke_sup}, 
\begin{align*}
  \Delta 
  &\leq 8( \sqrt{u_1} - \sqrt{u_2})u_2^{-1/6}\log(u_2)^{5/12} \\
  &\leq 8e^{(\log u_1)^{2/3}} (\log u_2)^{5/12}.
\end{align*}
The condition that $u_1 \geq u_2 + u_2^{2/3} e^{(\log u_1)^{2/3}}$ implies that $u_1 \geq u_2 + 0.5u_1^{2/3} e^{(\log u_1)^{2/3}}$ once $u_0$ is made sufficiently large, and hence
\[
-\frac{ (\sqrt{u_1} - \sqrt{u_2})^3}{\sqrt{u_1}}
\leq 
-\frac{ (u_1 - u_2)^3 }{8u_1^2}
\leq
-\frac{ e^{3(\log u_1)^{2/3}} }{64}.
\]
Hence we have by Lemma~\ref{lem:Ke_sup} that
\(
\left|{\tilde K}_e( u_2, y_2 ; u_1, y_1) \right|
\leq
e^{ -\Omega( \exp((\log u_1)^{2/3})) }
\)
uniformly over $I_M.$  As we may assure that $\eta > \tfrac 13,$ we have that
\begin{equation}
\HSN{ \tilde{K}_e \vert_{I_M} }^2 = 
O(e^{ -\Omega( \exp((\log u_1)^{2/3}) }),
  \label{eq:tKe}
\end{equation}
as the measure of $I_M$ is $O((\log u_1)^{200}).$
\end{proof}

\begin{lemma}
Provided that $u_1 \geq u_2 + u_2^{2/3} e^{(\log u_1)^{2/3}},$
	then uniformly in $t_i \geq -(\log u_i)^{5/12},$
	\[
\HSN{ \tilde{K}_o \vert_{I_M} }^2 = 
O\left( \log(u_1)^{5/6} \right).
	\]
	\label{lem:KoHS}
\end{lemma}
\begin{proof}
	Set $\tau_i(x) = \frac{x}{\sqrt{2}}{u_i}^{-1/6} + \sqrt{2u_i},$ so that $\tau_{i}(\yM[i]) = y_i.$  Let $\mub = (\log u_1)^{100},$ and consider the following four integrals:
\begin{align*}
  I_{i,j} &=
  \int\limits_{-t_i}^{\mub}
  \int\limits_{-t_j}^{\mub}
  \left|{ \tilde K}_o( u_i, \tau_i(\yM[1]); u_j,\tau_j(\yM[2]))\right|^2\frac{d\yM[1] d\yM[2]}{2u_i^{1/6}u_j^{1/6}}
\end{align*}
for $i,j \in \left\{ 1,2 \right\}.$   As $ \HSN{ \tilde{K}_o \vert_{I_M} }^2 = \sum I_{i,j},$ it suffices to show that each of these integrals has the desired bound.
\noindent \emph{ Bounding $I_{1,1}$ and $I_{2,2}$: }

The details of the proof are nearly identical for $I_{1,1}$ and $I_{2,2},$ and so we give the full proof for just $I_{1,1}.$   
All bounds on $|\tilde{K}_o|$ that we use come from case (3) of Lemma~\ref{lem:Ko_sup}.  We break the integral into four parts, according to the signs of $\yM[i]$ which we denote by $I_{1,1}^{\pm\pm}.$ 
For both $\yM[1] \leq 0$ and $\yM[2] \leq 0$, we have
\[
  \frac{\left|{ \tilde K}_o( u_1, \tau_1(\yM[1]); u_1,\tau_1(\yM[2]))\right|^2}{u_1^{1/3}}
\leq 
  C+C\max( |\yM[1]|,|\yM[2]|)
  e^{-2c(\sqrt{|\yM[1]|} - \sqrt{|\yM[2]|})^2(\sqrt{|\yM[1]|} + \sqrt{|\yM[2]|}) }.
\]
for some $C>0.$ 
Let $s_i = (t_i)_{-}.$
Hence, changing variables in $I_{1,1}^{--}$ by $w_{\pm} = \sqrt{-\yM[1]} \pm \sqrt{-\yM[2]}$ we have, adjusting constants, 
\begin{align*}
  I_{1,1}^{--}
&\leq Cs_1^2 + 
\int\limits_{0}^{2\sqrt{s_1}}
  \int\limits_{0}^{\infty}
  Cw_+^3 e^{-2c(w_{-})^2w_+} {dw_- dw_+} \\
  &\leq Cs_1^{2} + C'\int\limits_{0}^{2\sqrt{s_1}} w_+^{5/2} {dw_+} = O(s_1^2 + s_1^{7/4}).
\end{align*}

For $I_{1,1}^{+-},$ where $\yM[1] \geq 0$ and $\yM[2] \leq 0$, we have
\[
  \frac{\left|{ \tilde K}_o( u_1, \tau_1(\yM[1]); u_1,\tau_1(\yM[2]))\right|^2}{u_1^{1/3}}
\leq 
C\xi_+(\yM[1])+C|\yM[2]|
e^{-2c\yM[1]\sqrt{|\yM[2]|}}.
\]
Therefore changing variables and integrating,
\begin{align*}
  I_{1,1}^{+-}
  &\leq Cs_2\int\limits_0^{\mub} \frac{1}{{1+\yM[1]}}\,d\yM[1] +C 
\int\limits_{0}^{s_2}
  \int\limits_{0}^{\mub}
  \yM[2] e^{-2c\yM[1]\sqrt{\yM[2]}}
  d\yM[1]d\yM[2] \\
  &=O(s_2 \log{\mub} + s_2^{3/2})=O( (\log u_1)^{5/6}).
\end{align*}
A symmetric argument holds for $I_{1,1}^{-+}.$

For $I_{1,1}^{++},$ we have
\[
  \frac{\left|{ \tilde K}_o( u_1, \tau_1(\yM[1]); u_1,\tau_1(\yM[2]))\right|^2}{u_1^{1/3}}
  \leq C\xi_+(\yM[1])\xi_+(\yM[2]).
\]
Changing variables and integrating,
\[
  I_{1,1}^{++}
  \leq C
  \biggl(\int\limits_{0}^{\mub} \frac{d\yM[1]}{1+\yM[1]}\biggr)^2
  = O( (\log \mub)^2).
\]
Thus we have shown that $I_{1,1} =O( (\log u_1)^{5/6}).$

\noindent \emph{ Bounding $I_{2,1}$: }

Here we use cases (1) and (2) of Lemma \ref{lem:Ko_sup}.  These integrals are similar to or simpler than the ones in $I_{1,1}$ and are easily checked to be $O( (\log u_1)^{5/6}).$

\noindent \emph{ Bounding $I_{1,2}$: }

Here we use Lemma \ref{lem:Ko_pointwise}, which when we integrate gives the
following
\begin{equation}
  I_{1,2}
  \leq \frac{u_1^{1/6}u_2^{1/6}}{\left(\sqrt{u_1} - \sqrt{u_2}\right)^2}
  (C+e^{ C (t_1)_{-}^{3/2} + C (t_2)_{-}^{3/2} })
  \label{eq:KKD}
\end{equation}
for some absolute constant $C>0.$
By assumption on $u_1$ and $u_2,$ we have $\sqrt{u_1} - \sqrt{u_2} = \Omega(u_1^{1/6}e^{ (\log u_1)^{2/3} })$ uniformly in $u_2.$  Hence we have
\[
  I_{1,2}
 = e^{-\Omega((\log u_1)^{2/3})}.
\]

\end{proof}

\begin{proof}[Proof of Proposition  \ref{prop:lt_E}]

	The proof follows from \eqref{eq:det2lip}, Lemma \ref{lem:KeHS}, Lemma \ref{lem:KoHS} and the observation that by \eqref{eq:KKD},
\(
\HSN{ {\tilde K}  \vert_{I_M}- {\tilde K}^D \vert_{I_M} } = \sqrt{I_{1,2}}.
\)

\end{proof}

\begin{proof}[Proof of Proposition  \ref{prop:rt_E}]

	The only difference between this case and the one in the proof of Proposition \ref{prop:lt_E} is that we can sharpen the estimate of 
\(
\sqrt{I_{1,2}}=\HSN{ {\tilde K}  \vert_{I_M}- {\tilde K}^D \vert_{I_M} }.
\)
Using Lemma \ref{lem:rt_Ko_pointwise}, we have that for the range of $t_i$ considered, there is a constant $C>0$ so that
\[
	\HSN{{\tilde K}_o \vert_{I_M}}^2
	\leq
C\frac{u_1^{1/6}u_2^{1/6}}{\left(\sqrt{u_1} - \sqrt{u_2}\right)^2}
e^{ -\tfrac 23( (t_1)^{3/2} + (t_2)^{3/2})}.
\]
\end{proof}

\noindent
{\bf Acknowledgement}
We thank Gil Kalai for posing the question \cite{Gil} on Mathoverflow, 
which started this project. We thank Percy Deift 
for several enlightening discussions concerning \cite{DIK},
Paul Bourgade for pointing out \cite{LedouxRider} to us, and Alexei Borodin 
for several discussions concerning Conjecture \ref{conj-lower}.

\small
\setlength\baselineskip{1.1em}
  
\bibliographystyle{alpha}
\bibliography{LIL}
\end{document}